\documentclass[leqno]{article}

\usepackage[a4paper, total={6in, 8in}]{geometry}
\usepackage[english]{babel}
\usepackage{tgpagella}
\usepackage{graphicx}
\usepackage{xcolor}
\usepackage{float}% Required for inserting images
\usepackage{amsmath}
\usepackage{stmaryrd}
\usepackage{parskip}
\usepackage{mathrsfs}
\usepackage{amsfonts}
\usepackage{csquotes}
\usepackage{etoolbox}
\usepackage{hyperref}
\usepackage[pagewise,mathlines]{lineno}

\hypersetup{
  colorlinks=true,
  linkcolor=black,
  citecolor=black
}
\usepackage{xcolor}
\usepackage{tikz-cd}
\tikzcdset{scale cd/.style={every label/.append style={scale=#1},
    cells={nodes={scale=#1}}}}
\usepackage{amsthm}
\usepackage{accents}
\usepackage{mathtools}
\usepackage{authblk}
\newcommand{\ZZ}{\mathbb{Z}}
\newcommand{\CC}{\mathbb{C}}
\newcommand{\NN}{\mathbb{N}}
\newcommand{\QQ}{\mathbb{Q}}
\DeclareMathOperator{\SL}{SL}
\renewcommand{\sl}{\mathfrak{sl}}
\newcommand{\g}{\mathfrak{g}}
\newcommand{\ug}{\mathcal{U}_q({\hat{\mathfrak{g}}})}
\newcommand{\ub}{\mathcal{U}_q({\hat{\mathfrak{b}}})}
\newcommand{\Groth}{\foreignlanguage
{german}{Grothendieck }}
\newcommand{\sh}{\textnormal{sh}}
\newcommand{\loc}{\textnormal{loc}}
\DeclareMathOperator{\supp}{supp}
\newcommand{\grn}{\textnormal{grn}}
\DeclareMathOperator{\Id}{Id}
\DeclareMathOperator{\adj}{adj}

\newtheoremstyle{mytheoremstyle}%
{}%space above theorem
{}%space below theorem
{\itshape}%body font
{}%indentation amount
{\bfseries}%theorem head font
{.}%punctuation separator
{ }%space after theorem head
{\thmname{#1}\thmnumber{ #2}\thmnote{ \textnormal{#3}}}
\theoremstyle{mytheoremstyle}
\newtheorem{theorem}{Theorem}[section]

\newtheorem{corollary}[theorem]{Corollary}

\newtheorem{proposition}[theorem]{Proposition}
\newtheorem{lemma}[theorem]{Lemma}
\newtheorem{conjecture}[theorem]{Conjecture}

\newtheoremstyle{mydefinition}%
{}%space above theorem
{}%space below theorem
{}%body font
{}%indentation amount
{\bfseries}%theorem head font
{.}%punctuation separator
{ }%space after theorem head
{\thmname{#1}\thmnumber{ #2}\thmnote{ \textnormal{#3}}}
\theoremstyle{mydefinition}

\newtheorem{definition}[theorem]{Definition}
\newtheorem{example}[theorem]{Example}
\AtEndEnvironment{example}{\hfill$\diamondsuit$}

\theoremstyle{remark}
\newtheorem{remark}[theorem]{Remark}
\newtheorem{notation}[theorem]{Notation}
\usepackage[backend=biber,style=alphabetic,maxbibnames=9,maxalphanames=9, minalphanames=4]{biblatex}

\title{Quantum cluster algebras\\ and representations of shifted quantum affine algebras}
\author{Francesca Paganelli}
\date{}

\begin{document}
    
    \maketitle
    
    \begin{abstract}
        We construct a new quantization $K_t(\mathcal{O}^{\sh}_{\mathbb{Z}})$ of the Grothendieck ring of the category $\mathcal{O}^{\sh}_{\mathbb{Z}}$ of representations of shifted quantum affine algebras (of simply-laced type). We establish that our quantization is compatible with the quantum Grothendieck ring $K_t(\mathcal{O}^{\mathfrak{b},+}_{\mathbb{Z}})$ for the quantum Borel affine algebra, namely that there is a natural embedding $K_t(\mathcal{O}^{\mathfrak{b},+}_{\mathbb{Z}})\hookrightarrow K_t(\mathcal{O}^{\sh}_{\mathbb{Z}})$. Our construction is partially based on the cluster algebra structure on the classical Grothendieck ring discovered by Geiss--Hernandez--Leclerc. As first applications, we formulate a quantum analogue of $QQ$-systems (that we make completely explicit in type $A_1$). We also prove that the quantum oscillator algebra is isomorphic to a localization of a subalgebra of our quantum Grothendieck ring and that it is also isomorphic to the Berenstein--Zelevinsky's quantum double Bruhat cell $\mathbb{C}_t[\SL_2^{w_0,w_0}]$.
    \end{abstract}
    
    \tableofcontents
    \newpage
    
    \section*{Introduction}

    Let $\g$ be a simple, finite-dimensional, complex Lie algebra of simply-laced type. Shifted quantum affine algebras were introduced by Finkelberg and Tsymbaliuk \cite{ft} in the study of $K$-theoretic Coulomb branches. They are variations of quantum affine algebras depending on a parameter $\mu$ known as shift parameter. For every integral coweight $\mu$ of $\g$, the shifted quantum affine algebra $\mathcal{U}_q^{\mu}(\hat{\g})$ is defined by the same Drinfeld generators as the ordinary quantum affine algebra $\mathcal{U}_q(\hat{\g})$, subject to the same relations, except for a shift in the generating function of the elements $\phi_i$'s (the analogues of Cartan elements in the Drinfeld presentation). Note that since we consider here only Lie algebras of simply-laced type, weights and coweights can be identified. Shifted quantum affine algebras  possess a structure notably different from the ordinary ones. For example, the usual quantum group $\mathcal{U}_q(\sl_2)$ is a subalgebra of $\mathcal{U}_q(\widehat{\sl_2})$, but for the shift $\mu =  -\omega$ (the opposite of the fundamental weight for the Lie algebra $\sl_2$), instead the quantum oscillator algebra (e.g. from \cite{blz}) is a subalgebra of $\mathcal{U}_q^{-\omega}(\widehat{\sl_2})$.

    Hernandez \cite{sqaahernandez} initiated a systematic study of the representation theory of shifted quantum affine algebras, by introducing for each integral weight $\mu$ a category $\mathcal{O}^{\mu}$ of $\mathcal{U}_q^{\mu}(\hat{\g})$-representations. Then, the direct sum of abelian categories $\mathcal{O}^{\sh}:=\bigoplus_{\mu}\mathcal{O}^{\mu}$ is endowed with an operation called \emph{fusion product} that allows us to define a ring structure on the \Groth group $K_0(\mathcal{O}^{\sh})$. In this work we restrict the attention to $K_0(\mathcal{O}^{\sh}_{\ZZ})$, where $\mathcal{O}^{\sh}_{\ZZ}$ is a certain natural subcategory of $\mathcal{O}^{\sh}$, whose study is sufficient to understand most properties of the whole category $\mathcal{O}^{\sh}$. 

    Geiss, Hernandez and Leclerc \cite{ghl} proved that $K_0(\mathcal{O}^{\sh}_{\ZZ})$ is isomorphic to a completion of a cluster algebra denoted $\mathcal{A}_{w_0}$. In their work they construct an explicit initial quiver $\Gamma_c$ with vertex set $V$ for the cluster algebra $\mathcal{A}_{w_0}$ (see Figures \ref{fig:basic quiver A3} and \ref{fig:alcuni quiver A1}). A crucial property of this cluster algebra is that it is possible to compute the $g$-vectors of initial cluster variables with respect to a \emph{limit} reference seed (see Theorem \ref{teo stabilized g vectors}).  

    In this paper we establish the existence of a quantization of $K_0(\mathcal{O}^{\sh}_{\ZZ})$ and prove that this quantization contains the quantum Grothendieck ring of finite-dimensional representations of the ordinary quantum affine algebra. Our technique involves constructing a quantization of the cluster algebra $\mathcal{A}_{w_0}$ and then considering a completion of this quantum cluster algebra. This approach is partially inspired by the work of Bittmann \cite{theseBittmann} \cite{b21}, where the author defines the quantum \Groth ring $K_t(\mathcal{O}^{\mathfrak{b},\pm}_{\ZZ})$ providing a quantization of the cluster algebra structure isomorphic to $K_0(\mathcal{O}^{\mathfrak{b},\pm}_{\ZZ})$ discovered by Hernandez--Leclerc in \cite{hl16o} (here $\mathcal{O}^{\mathfrak{b},+}_{\ZZ}$ and $\mathcal{O}^{\mathfrak{b},-}_{\ZZ}$ are certain remarkable monoidal categories of representations of Borel subalgebra of the ordinary quantum affine algebra). Several new ingredients also play a role in the present work, particularly the limit reference seed.

    More generally, the interplay between quantum groups, cluster algebras and (quantum) \Groth rings of representations of quantum groups has been intensively studied. Interest in quantum Grothendieck rings for affine quantum groups originated with the works of Nakajima \cite{nak04} and Varagnolo--Vasserot \cite{vv}. For $t$ a formal indeterminate, they defined a non-commutative deformation of the \Groth ring of the category of finite-dimensional representations of the quantum affine algebra. The construction is based on the geometry of quiver varieties, and Nakajima's work provided an algorithm of Kazhdan--Lusztig type to compute $q$-characters of simple finite-dimensional representations of the quantum affine algebras of simply-laced type. Further results about quantum \Groth rings and their relations with cluster algebras can be found in the papers \cite{hl15, qin17, fhoo, kkop} and references therein .
    
    %summary of the paper
    The first main result of this paper is the existence (and construction) of a quantization $\mathcal{A}_{t,w_0}$ of the cluster algebra $\mathcal{A}_{w_0}$. Consider an initial seed of $\mathcal{A}_{w_0}$. In order to define a quantum cluster algebra structure, we first need a quantum torus, which is a $\ZZ\big[t^{\pm \frac{1}{2}}\big]$-algebra $\mathcal{T}_{t,c}$ generated by the initial cluster variables $\{\underline{Q}_v\}_{v\in V}$, subject to the $t$-commutation relations
    \[\underline{Q}_v\underline{Q}_{w}=t^{(\Lambda_c)_{v,w}}\ \underline{Q}_w\underline{Q}_v,\] where the matrix $\Lambda_c$ is the so-called quantization matrix. Thus, the first thing to do is to define $\Lambda_c$. In Definition \ref{definizione lambda c} we introduce
    \[\Lambda_c:=\left(G^{(\infty)}\right)^T\Lambda_e G^{(\infty)},\] where $G^{(\infty)}$ is the \emph{limit} $G$-matrix defined in \cite{ghl} and $(\Lambda_e,B_e)$ is the compatible pair provided in \cite{b21} to construct $K_t(\mathcal{O}^{\mathfrak{b},+}_{\ZZ})$. Let $B_c$ be the exchange matrix of the initial quiver $\Gamma_c$ for $\mathcal{A}_{w_0}$. In Proposition \ref{quantum cluster ade} we prove that $(\Lambda_c,B_c)$ forms a \emph{compatible pair}, in the sense of \cite{bz05}. Our proof partially relies on Theorem \ref{teo convergenza}, which states that the matrix $B_c$ relates to the exchange matrix $B_e$ via the formula
    \[B_e=G^{(\infty)}B_c\left(G^{\infty}\right)^T.\] This generalizes a classical result (Lemma \ref{formula g matrix}) to infinite rank cluster algebras and infinite sequences of mutations.\\ 
    In Definition \ref{def quantum groth} we define the quantum Grothendieck ring $K_t(\mathcal{O}^{\sh}_{\ZZ})$ as a completion of $\mathcal{A}_{t,w_0}$. A key property of our definition of $K_t(\mathcal{O}^{\sh}_{\ZZ})$ is its compatibility with the quantum Grothendieck ring $K_t(\mathcal{O}^{\mathfrak{b},\pm}_{\ZZ})$ of \cite{b21}, and so with the original quantum Grothendieck ring for finite-dimensional representations of the ordinary quantum affine algebra. Indeed, we have the following two results. Let $\mathcal{T}_{t,e}$ be the quantum torus constructed in \cite{b21}. Then, in Theorem \ref{iso tori} we prove that $\mathcal{T}_{t,c}$ and $\mathcal{T}_{t,e}$ are isomorphic algebras. Using this, we prove in Theorem \ref{inclusione anelli} that there are two injective ring morphisms 
    \[\mathcal{I}_t^{\pm}:K_t(\mathcal{O}^{\mathfrak{b},\pm}_{\ZZ})\to K_t(\mathcal{O}^{\sh}_{\ZZ}).\] We provide detailed examples for Lie algebras of type $A_1$ and $A_2$, to clarify the constructions and explain the definitions (see Examples \ref{esempio A1 1}, \ref{esempio A1 2}, \ref{esempio A2}).
    
    Our construction of the quantum Grothendieck $K_t(\mathcal{O}^{\sh}_{\ZZ})$ has several consequences. 
    
    First, recall that the classical $QQ$-systems in the sense of \cite{fh23} are realized in \cite{ghl} as relations satisfied by certain cluster variables of $\mathcal{A}_{w_0}$. Thus, our construction yields to deformations of the $QQ$-systems: we obtain quantum $QQ$-systems as quantum exchange relations. As an example, we explicitly write a quantum version of the $QQ$-system in type $A_1$ (see Theorem \ref{quantum qq}).  In principle we can write quantum $QQ$-systems for all simply-laced types, as we provide an explicit formula for the quantization matrix $\Lambda_c$.
    
    A second application of the quantization of the cluster algebra $\mathcal{A}_{w_0}$ is the following. Let $\mathcal{U}_{t,loc}^+(\widehat{\sl_2})$ denote the quantum oscillator algebra (with quantum parameter $t$) localized at a central Casimir element. Then we consider the cluster subalgebra of $\mathcal{A}_{w_0}$ in type $A_1$ corresponding to a finite ice subquiver of $\Gamma_c$ (see Figure \ref{fig:initial seed q osc}). The quantum cluster algebra structure on $\mathcal{A}_{t,w_0}$ induces a compatible quantization of this subalgebra. Additionally, we assume that the frozen variables are invertible. The resulting quantum cluster algebra is denoted $\mathcal{A}_t$. We prove in Theorem \ref{teo q osc is cluster} that $\mathcal{U}_{t,loc}^+(\widehat{\sl_2})$ is isomorphic to $\mathcal{A}_t$. Consequently, we obtain in Corollary \ref{cor q osc} that
    \[\mathcal{U}^+_{t,\loc}(\widehat{\sl_2})\simeq \CC_t[\SL_2^{w_0,w_0}], \]
    where $\CC_t[\SL_2^{w_0,w_0}]$ is the Berenstein--Zelevinsky's quantum double Bruhat cell ($w_0$ denotes the longest element in the Weyl group of $\sl_2$, i.e. the simple reflection).

    We point out that, during the writing of this paper, we were informed that also Fan Qin \cite{qin24} had discovered a quantization of $\mathcal{A}_{w_0}$, with a different method. Qin's techniques do not involve the limit reference seed, and it is unclear whether his quantization contains the quantum Grothendieck ring for the Borel quantum affine subalgebra. Therefore, we expect our approach to have a more direct impact on the study of representation theory of shifted quantum affine algebras.

    In conclusion, we present some future directions of research that we plan to discuss in a forthcoming paper. Namely, motivated by the existence of the quantization of $\mathcal{A}_{w_0}$ in \cite{qin24}, we investigate whether Qin's quantization coincides with ours. More generally we discuss the uniqueness (or not) of the quantization for a given infinite rank cluster algebra. Another direction is motivated by \cite[Section 10]{ghl}. Let $G$ be a simple, complex algebraic group with Lie algebra $\g$. In \cite{ghl} the authors show that a finite, full subquiver $\gamma_c$ of the initial quiver $\Gamma_c$ for the cluster algebra $\mathcal{A}_{w_0}$ is equal to an initial quiver for the cluster algebra structure on the double Bruhat cell $G^{w_0,w_0}$, which is a result from \cite{fzbruhat}. From \cite{bz05} and \cite{yg}, the quantum double Bruhat cell $\CC_t[G^{w_0,w_0}]$ admits a quantum cluster algebra structure. Thus, we wonder if the subalgebra of our $\mathcal{A}_{t,w_0}$, corresponding to the subquiver $\gamma_c$, is the same quantum cluster algebra found in \cite{bz05} and in \cite{yg} (up to rescaling of the quantization matrix). We have a positive answer in type $A_1$ and $A_2$, and we conjecture that this is true in general. Finally we discuss possible applications to the existence of $(q,t)$-characters for shifted quantum affine algebras.
    
    %organization
    The paper is organized as follows. In Sections \ref{section borel} and \ref{rep sqaa} we establish general notations about Lie algebras and we briefly review the definitions of the quantum affine Borel subalgebra and of shifted quantum affine algebras. We also make some reminders on categories $\mathcal{O}^\mathfrak{b}$, $\mathcal{O}^{\sh}$ and their \Groth rings.
    In Section \ref{section on cluster structure on gr rings} we collect some results about cluster algebra structures on Grothendieck rings that are fundamental for our construction. In particular we recall the definition of the basic infinite quiver and the cluster algebra structure on $K_0(\mathcal{O}^{\mathfrak{b},+}_{\ZZ})$ from \cite{hl16o}, the quantization procedure in \cite{b21} to obtain $K_t(\mathcal{O}^{\mathfrak{b},+}_{\ZZ})$ and the main results from \cite{ghl}. Concerning \cite{ghl}, we concentrate on the notion of stabilized $g$-vector and on the definition of the initial cluster seed, since we will use it strongly in our constructions. In addition, we clarify some ``implicit'' content from that paper. This is done in Lemma \ref{T gives leading term} and in Theorem \ref{teo diag comm}.
    Section \ref{the quantum grothendieck ring} is the core of the paper. We construct the quantum cluster algebra $\mathcal{A}_{t,w_0}$ by defining a quantum torus compatible with the initial seed provided by \cite{ghl}. Then we define the quantum \Groth ring $K_t(\mathcal{O}^{\sh}_{\ZZ})$ and we prove that it contains $K_t(\mathcal{O}^{\mathfrak{b},+}_{\ZZ})$. In Section \ref{section qq} we formulate quantum $QQ$-systems and in Section \ref{section q oscillator} we prove that the localized quantum oscillator algebra has a quantum cluster algebra structure. Finally, in Section \ref{further questions} we present some open questions.

    \section*{Acknowledgments}
    
    The results presented in this paper are part of my Ph.D. thesis under the supervision of David Hernandez and Giovanni Cerullli Irelli. I thank them for introducing me to these subjects, for many helpful discussions, opinions and support. I also thank Bernard Leclerc, Yann Palu, Fan Qin and Lior Silberberg for useful discussions and suggestions. Finally I thank the anonymous referee for the careful reading of the manuscript and precious comments. I acknowledge the support of \emph{Avvio alla ricerca} "Quantum Grothendieck ring for shifted quantum affine algebras" of Sapienza University, of NextGenerationEU - PRIN 2022 -B53D23009430006- \\2022S97PMY-Structures for Quivers, Algebras and Representations (SQUARE) and of \emph{National Group for Algebraic and Geometric Structures, and their Applications} (GNSAGA - INdAM).

    \section{Background: the quantum affine Borel algebra}\label{section borel}
    
    In this section we fix some notations and we recall briefly the definitions and some important properties of the quantum affine algebra $\mathcal{U}_q(\hat{\mathfrak{g}})$ and of its Borel subalgebra $\mathcal{U}_q(\hat{\mathfrak{b}})$.
    
    Let $q$ be a non-zero complex number, not root of unity. We fix throughout this work a complex number $h$ such that $q=e^h$ and for every $r\in\QQ$ we define $q^r=e^{rh}$. Thus for every $r,s\in\QQ$, $q^r=q^s$ if and only if $r=s$. Let $\g$ be a finite-dimensional simple Lie algebra of simply-laced type of rank $n$ and let $I$ be the set $\{1,...,n\}$. Let $\mathfrak{h}$ be a fixed Cartan subalgebra of $\mathfrak{g}$. Let $W$ be the Weyl group of $\g$ and $\{s_i\mid i\in I\}$ the simple reflections. We denote by $\{\omega_i\}_{i\in I}$ the fundamental weights and by $\{\alpha_i\}_{i\in I}$ the simple roots. 
    Then let $P$ be the lattice of integral weights and $P^+$ the set of dominant integral weights. We denote by $\leq$ the standard partial order on $P$, i.e. for all $\lambda$, $\mu$ in $P$, $\mu\leq \lambda$ if and only if $\lambda-\mu$ is a linear combination of the simple roots with non-negative integer coefficients. We use the notation $P_{\QQ}:=P\otimes\QQ$. Let $C=(c_{i,j})$ be the Cartan matrix. Let $\{h_i\}_{i\in I}$ be the basis of the Cartan subalgebra $\mathfrak{h}$ of $\g$ in the Serre presentation of $\g$. We denote by $\gamma$ the Dynkin diagram associated with $\g$.
    The quantum affine algebra $\ug$ has two realizations, known as Drinfeld--Jimbo realization and Drinfeld realization. 
    \begin{definition}\label{def uqg dj}
        The quantum affine algebra $\ug$ in the Drinfeld--Jimbo realization is the associative $\CC$-algebra with generators $e_i$, $f_i$, $k_i^{\pm}$, $0\leq i\leq n$ subject to the quantum Weyl--Serre relations (see for example \cite{hj12} for the list of relations).
    \end{definition}
    \begin{definition}
        The quantum affine Borel subalgebra $\ub$ is the subalgebra of $\ug$ generated by $e_i$, $k_i^{\pm 1}$, $i\in I\sqcup \{0\}$. 
    \end{definition}
    
    On the other hand, $\ug$ can be presented in the Drinfeld realization as the associative algebra with generators $x_{i,r}^{\pm}$ ($i\in I$, $r\in \ZZ$), $\phi_{i,r}^{\pm }$ ($i\in I$, $r\in\ZZ$), $h_{i,r}$ ($i\in I$, $r\in\ZZ\setminus\{0\}$) and central elements $c^{\pm\frac{1}{2}}$, subject to relations (see for example \cite{chari94} for the complete list of relations). Since we will need them later, we only recall the generating series of the $\phi_{i,r}^{\pm}$'s:
    \begin{equation}\label{sviluppo phi}
        \sum_{r=0}^\infty \phi_{i,\pm r}^{\pm}u^{\pm r}=k_i^{\pm 1}\mathrm{exp}\left(\pm(q-q^{-1})\sum_{m=1}^{\infty}h_{i,\pm m}u^{\pm m}\right).
    \end{equation}
    Note that the elements $\phi_{i,r}^+$, ($i\in I$, $r\in\ZZ$) belong to $\mathcal{U}_q(\hat{\mathfrak{b}})$.

    We recall some useful facts and notations about the representations of $\ub$. For more details, we refer to \cite{hj12}. 
    For a $\ub$-module $V$ and a weight $\omega\in P_{\QQ}$, the weight space of $V$ of weight $\omega$ is the linear subspace
    \[V_{\omega}:=\{v\in V\mid k_iv=q^{\omega(h_i)}v,\ 1\le i\leq n\}.\] 
    The notion of category $\mathcal{O}^{\mathfrak{b}}$ is given in \cite{hj12}, in analogy with BGG category $\mathcal{O}$. Namely, its objects are $\mathcal{U}_q(\hat{\mathfrak{b}})$-modules such that they are Cartan diagonalizable, with finite-dimensional weight spaces and weights contained in a finite union of cones of the form $D(\lambda)=\{\omega\in P_{\QQ}\mid \omega\leq\lambda\}$.
    
    \begin{definition}
        A sequence of complex numbers $\Psi=(\psi_{i,m})_{i\in I,m\geq 0}$ such that $\psi_{i,0}\in q^{\QQ}$ for all $i\in I$ is called an $\ell$-weight ($\ell$ stays for ``loop''). The set of $\ell$-weights is denoted $P_{\ell}$.
    \end{definition}
    It is convenient to identify $(\psi_{i,m})_{i\in I,m\geq 0}$ with its generating series, namely
    \[\Psi=(\psi_{i}(z))_{i\in I}\qquad \psi_i(z)=\sum_{m\geq 0}\psi_{i,m}z^m.\] Since $\psi_{i,0}\neq 0$, each $\psi_i(z)$ is an invertible power series, so $P_{\ell}$ has a group structure (with multiplicative notation).
    Then, we can define a surjective group homomorphism 
    \begin{equation}\label{mappa peso}
        \varpi: P_{\ell}\to P_{\QQ}
    \end{equation}
    which assigns to $\Psi\in P_{\ell}$ the weight $\varpi(\Psi)$ such that for all $i\in I$
    \[\psi_{i}(0)=\psi_{i,0}=q^{\varpi(\Psi)(h_i)}.\] Let $V$ be an $\ub$-module, $\Psi\in P_{\ell}$, then the corresponding $\ell$-weight space of $V$ is defined as the linear subspace 
    \[V_{\Psi}:=\{v\in V\mid\exists p\geq 0,\forall i\in I,\forall m\geq 0, (\phi_{i,m}^+-\psi_{,m})^pv=0 \}.\]
    
    As in the classical theory, there is a notion of highest weight module and, for any $\Psi\in P_{\ell}$ there exists a unique (up to isomorphism) simple module $L(\Psi)$ with highest $\ell$-weight $\Psi$.
    
    \begin{example}\label{rappresentazione costante}
        Let $\omega\in P_{\QQ}$, then we can define the $\ell$-weight $[\omega]$ by \[[\omega]=\left([\omega]_i(z)\right)_{i\in I},\qquad [\omega]_i(z)=q^{\omega(h_i)}.\] 
        The associated simple\\representation $L([\omega])$ is one-dimensional and it is called the constant representation. Note that the multiplication of two constant $\ell$-weights $[\omega]$ and $[\omega']$ is $[\omega][\omega']=([\omega][\omega']_i(z))_{i\in I}$, where $([\omega][\omega'])_i(z)=q^{(\omega+\omega')(h_i)}$. We will denote by $[P_{\QQ}]$ the group ring generated by the $[\lambda]$, $\lambda\in P_{\QQ}$.
    \end{example}
    
    \begin{example}\label{prefondamentali borel}
        Another important class of $\ub$-modules is the one of so-called positive/negative prefundamental representations. They have been introduced in \cite{hj12}. For $i\in I$ and $a\in \CC^{*}$ we define
        \[L_{i,a}^{\mathfrak{b},\pm}=L(\Psi_{i,a}^{\pm 1}),\qquad (\Psi_{i,a})_j(z)=\begin{cases}
            
                1-za\quad  &j=i;\\
                1\quad     &j\neq i
           
        \end{cases}\] These representations are all infinite dimensional.
    \end{example}
    
    \begin{definition}\label{rational l weights}
        An $\ell$-weight $\Psi$ is called rational if $\psi_i(z)$ is rational for all $i\in I$. We call $\mathfrak{r}$ the group of rational $\ell$-weights.
    \end{definition}
    Then, any $\ub$-module $V$ in the category $\mathcal{O}^{\mathfrak{b}}$ has rational $\ell$-weights.
    \begin{theorem}[{\cite[Theorem 3.11]{hj12}}]
        The simple modules in the category $\mathcal{O}^{\mathfrak{b}}$ are the $L(\Psi)$ with $\Psi\in\mathfrak{r}$.
    \end{theorem}
    Finally, every module in the category $\mathcal{O}^{\mathfrak{b}}$ is direct sum of its $\ell$-weight spaces.
    
    We define $\mathcal{E}$ to be the additive group of maps $c:P_{\QQ}\to \ZZ$ whose support $\supp(c)$ is contained in a finite union of cones $D(\lambda)$. For any $\lambda\in P$ we define $[\lambda]=\delta_{\lambda,-}\in\mathcal{E}$. Every element in $\mathcal{E}$ can be regarded as a formal sum
    \[c=\sum_{\lambda\in\supp(c)}c(\lambda)[\lambda].\] 
    We can endow $\mathcal{E}$ with a ring structure, defining a product by $[\lambda][\lambda']=[\lambda+\lambda']$ for any $\lambda,\lambda'\in P_{\QQ}$.  This definition for the product in $\mathcal{E}$ is compatible with the multiplication in the group $P_{\QQ}$ and consistent with the notation for the constant representation in Example \ref{rappresentazione costante}.
    
    Thanks to a Jordan--H\"{o}lder type theorem, the multiplicity of a simple representation in a representation in $\mathcal{O}^{\mathfrak{b}}$ is well defined. Hence, we can think of an element in the Grothendieck ring $ K_0(\mathcal{O}^{\mathfrak{b}})$ as a sum
    \[\sum_{\Psi\in\mathfrak{r}}m_{\Psi}[L(\Psi)], \] where $m_{\Psi}\in\ZZ$ verify:
    \[\sum_{\Psi\in\mathfrak{r},\lambda\in P_{\QQ}}|m_{\Psi}|\dim(L(\Psi)_{\lambda})[\lambda]\in\mathcal{E}.\] 
    As in the theory of $q$-characters for representations of the quantum affine algebras by Frenkel and Reshetikhin \cite{fr}, there is a $q$-character morphism for representations in $ \mathcal{O}^{\mathfrak{b}}$. We enlarge $\mathcal{E}$ to obtain $\mathcal{E}_{\ell}$, that is the additive group of maps $\zeta:\mathfrak{r}\to \ZZ$ such that $\varpi(\supp(\zeta))$ is contained in a finite union of cones $D(\lambda)$ ($\varpi $ is defined on the support of $\zeta$ because $\mathfrak{r}$ is a subset of $P_{\ell}$). For every $\Psi\in\mathfrak{r}$, let $[\Psi]\in\mathcal{E}_{\ell}$ be $\delta_{\Psi,-}$. Then, any element $\zeta\in\mathcal{E}_{\ell}$ can be written as a formal sum
    \[\zeta=\sum_{\Psi\in \supp(\zeta)}\zeta(\Psi)[\Psi].\] $\mathcal{E}_{\ell}$ has a ring structure induced by
    \[[\Psi][\Psi']=[\Psi\Psi'],\qquad \forall\Psi,\Psi'\in\mathfrak{r}.\]
    % There is a natural surjective ring homomorphism
    % \[\varpi:\mathcal{E}_{\ell}\to \mathcal{E}.\] 
    The $q$-character of a representation $V\in\mathcal{O}^{\mathfrak{b}}$ is defined as
    \[\chi_q(V):=\sum_{\Psi \in\mathfrak{r}_{\mu}}\dim(V_{\Psi})[\Psi]\in\mathcal{E}_{\ell}.\] 
    The $q$-character morphism $\chi_q: K(\mathcal{O}^{\mathfrak{b}})\to \mathcal{E}_{\ell}$ is injective \cite[Proposition 3.12]{hj12}.

    We conclude the section recalling that in \cite{hl16o} two subcategories of the category $\mathcal{O}^{\mathfrak{b}}$ are defined. They are denoted by 
    \[\mathcal{O}^{\mathfrak{b},+},\quad \mathcal{O}^{\mathfrak{b},-}.\]Here we do not give the definitions, but we point out that they are monoidal, hence we can consider the \Groth rings $K_0(\mathcal{O}^{\mathfrak{b},+})$ (resp. $K_0(\mathcal{O}^{\mathfrak{b},-})$). These categories contain the category of finite-dimensional $\mathcal{U}_q(\hat{\mathfrak{b}})$-modules. In particular, $K_0(\mathcal{O}^{\mathfrak{b},\pm})$ contains as subring the Grothendieck ring of constant representations, that we can naturally identify with $\mathcal{E}$. Moreover, $\mathcal{O}^{\mathfrak{b},+}$ (resp. $\mathcal{O}^{\mathfrak{b},-}$) contains the positive (resp. negative) prefundamental representations $L_{i,a}^{\mathfrak{b},+}$ (resp. $L_{i,a}^{\mathfrak{b},-}$), for $i\in I$ and $a\in\CC^{*}$. Finally, the rings $K_0(\mathcal{O}^{\mathfrak{b},\pm})$ are isomorphic (see \cite{hl16o}).

    \section{Background: shifted quantum affine algebras }\label{rep sqaa}
    
    In this section we give an overview about shifted quantum affine algebras and their category of representations $\mathcal{O}^{\sh}$. The references are \cite{ft} and \cite{sqaahernandez}. We keep the notation from the previous section; in particular $\g$ still denotes a finite-dimensional simple Lie algebra of simply-laced type.

    \begin{definition}
        Let $\mu\in P$. The shifted quantum affine algebra $\mathcal{U}_q^{\mu}(\hat{\g})$ is the associative $\CC$-algebra generated by the same elements $x_{i,r}^{\pm}$, $\phi_{i,r}^{\pm }$ ($i\in I$, $r\in\ZZ$), $h_{i,r}$ ($i\in I$, $r\in \ZZ\setminus\{0\}$) of the Drinfeld presentation of the quantum affine algebra, subject to the same relations, except that 
        \[\phi_i^-(z)=\sum_{m=0}^{\infty}\phi_{i,-m}^-z^{-m}=\phi^-_{i,\mu(h_i)}z^{\mu(h_i)}\mathrm{exp}\left(-(q-q^{-1})\sum_{r=1}^{\infty}h_{i,-r}z^{-r}\right),\] where we recall that $h_i$ denotes a generator of $\mathfrak{h}$ in a fixed Serre presentation of $\g$. Moreover, $\phi_{i,0}^+$ and $\phi_{i,\alpha_i(\mu)}^-$ are invertible, but not necessarily inverse of each other.
    \end{definition}
    Note that we assume that the shift modifies only the relations for $\phi_{i,r}^-$, so, as in the case of the quantum affine algebras (non-shifted), $\phi_{i,0}^+=k_i$. Shifted quantum affine algebras constitute a large framework to deal with other well known quantum groups, namely:
    \begin{itemize}
        \item the quantum affine algebra $\mathcal{U}_q(\hat{\g})$ can be obtained as the quotient of $\mathcal{U}_q^0(\hat{\g})$ by the ideal generated by $\phi_{i,0}^+\phi_{i,0}^- -1$, see \cite{ft};\\
        \item if $\mu$ is antidominant, then $\mathcal{U}_q^{\mu}(\hat{\g})$ contains a subalgebra isomorphic to the quantum affine Borel algebra $\mathcal{U}_q(\hat{\mathfrak{b}})$ from Section \ref{section borel}, see \cite{sqaahernandez};\\
        \item for $\g=\sl_2$, the algebra $\mathcal{U}_q^{-\omega}(\widehat{\sl_2})$ contains as sub-algebras the quantum oscillator algebras $\mathcal{U}_q^+(\sl_2)$ and $\mathcal{U}_q^-(\sl_2)$, whose definitions are given below. 
    
    \end{itemize}

    \begin{definition}\label{quantum osc}
        The \emph{quantum oscillator algebras} $\mathcal{U}_q^{\pm}(\sl_2)$ are generated by the elements $e,f,k^{\pm 1}$ with relations
        \begin{equation*}
            kk^{-1}=k^{-1}k=1,\quad ke=q^2ek,\quad kf=q^{-2}fk,\quad [e,f]=\frac{\pm k^{\pm 1}}{q-q^{-1}}.
        \end{equation*} This definition should be compared with the one of the quantum universal enveloping algebra $\mathcal{U}_q(\sl_2)$. Moreover, note that exchanging $k$ and $k^{-1}$ we obtain that 
        \[\mathcal{U}_q^+(\sl_2)\simeq\mathcal{U}_q^-(\sl_2).\] As observed in \cite[Remark 3.1.iv]{sqaahernandez}, the quantum oscillator algebra $\mathcal{U}_q^+(\sl_2)$ is isomorphic to a sub-algebra of $\mathcal{U}_q^{-\omega}(\widehat{\sl_2})$. Indeed, the defining relations of $\mathcal{U}_q^{-\omega}(\widehat{\sl_2})$ differ from the ones of $\mathcal{U}_q(\widehat{\sl_2})$ only for
        \[[x_r^+,x_{-r}^-]=\frac{\phi_0^+}{q-q^{-1}}.\] Moreover, by definition, $\phi_0^+$ has an inverse that we denote $(\phi_0^+)^{-1}$. Then, if we consider the sub-algebra of $\mathcal{U}_q^{-\omega}(\widehat{\sl_2})$ generated by 
        \[x_0^+,x_0^-,\phi_0^+,(\phi_0^+)^{-1},\]
        we see that the assignment 
        \[e\mapsto x_0^+,\ f\mapsto x_0^-,\ k\mapsto \phi_0^{+},\ k^{-1}\mapsto(\phi_0^+)^{-1}\] defines an isomorphism.
    \end{definition}
    In \cite[Definition 4.8]{sqaahernandez} the category $\mathcal{O}^{\mu}$ of representations of $\mathcal{U}_q^{\mu}(\hat{\g})$ is defined. Here we consider the same definition as in \cite[Definition 9.2]{ghl}, which is analogous to the one of $\mathcal{O}^{\mathfrak{b}}$. The $\ell$-weights of modules in $\mathcal{O}^{\mu}$ are in 
    \[\mathfrak{r}_{\mu}:=\{\Psi=(\psi_i(z))_{i\in I}\in\mathfrak{r}\mid\deg(\psi_i(z))=\mu(h_i)\}.\]
    For more properties and details about $\mathcal{O}^{\mu}$, see \cite[Section 4.4]{sqaahernandez}.
    \begin{example}\label{positive pref}
        The simple representations that we will use the most are the so-called prefundamental representations. Recall the definition of $\Psi_{i,a}^{\pm 1}$ from Example \ref{prefondamentali borel}. We denote $L_{i,a}^{\pm}=L(\Psi_{i,a}^{\pm 1})$ the positive/negative prefundamental representations, which are objects in $\mathcal{O}^{\pm \omega_i}$, $i\in I$, $a\in \CC^{\times}.$ Note that although they share the same name and same highest $\ell$-weight, $L_{i,a}^{\mathfrak{b},\pm}$ and $L_{i,a}^{\pm}$ are different. Indeed, the $\mathcal{U}_q^{\omega_i}(\hat{\g})$-module $L_{i,a}^+$ is one-dimensional, whereas $L_{i,a}^{\mathfrak{b},+}$ is infinite-dimensional. On the other hand, the $\mathcal{U}_q^{-\omega_i}(\hat{\g})$-module $L_{i,a}^-$ is infinite-dimensional and the restriction to the action of $\mathcal{U}_q(\hat{\mathfrak{b}})$ coincides with $L_{i,a}^{\mathfrak{b},-}$ (recall that for antidominant weights $\mathcal{U}_q(\hat{\mathfrak{b}})\subset\mathcal{U}_q^{\mu}(\hat{\g})$).
    \end{example}
    We will consider simultaneously all the categories $\mathcal{O}^{\mu}$, $\mu \in P$ by defining the direct sum of abelian categories
    \begin{equation}
        \mathcal{O}^{\sh}:=\bigoplus_{\mu\in P}\mathcal{O}^{\mu}.
    \end{equation}
     The simple objects in $\mathcal{O}^{\sh}$ are parameterized by $\mathfrak{r}$.

    Let $ K_0(\mathcal{O^{\mu}})$ be the Grothendieck group of the category $\mathcal{O}^{\mu}$. The direct sum
    \[ K_0(\mathcal{O}^{\sh}):=\bigoplus_{\mu\in P} K_0(\mathcal{O}^{\mu})\] can be endowed with a ring structure thanks to a fusion product (see \cite{sqaahernandez}): \[ K_0(\mathcal{O}^{\mu_1})\cdot K_0(\mathcal{O}^{\mu_2})\subset  K(\mathcal{O}^{\mu_1
    +\mu_2}),\] defined in \cite{sqaahernandez}, obtained by Drinfeld coproduct. There is a $q$-character morphism for representations in $\mathcal{O}^{\sh}$. It is defined in a similar way to the one used in the previous section for the category $\mathcal{O}^{\mathfrak{b}}$. \\
    For a representation $V\in\mathcal{O}^{\mu}$, the $q$-character is
    \[\chi_q(V):=\sum_{\Psi \in\mathfrak{r}_{\mu}}\dim(V_{\Psi})[\Psi]\in\mathcal{E}_{\ell}.\] 
    \begin{example}
        For $i\in I$ and $a\in\CC^*$, the positive prefundamental representation $L_{i,a}^+$ of $\mathcal{U}_q^{\omega_i}(\hat{\g})$ satisfies:
    \[\chi_q(L_{i,a}^+)=[\Psi_{i,a}].\]
    \end{example}
    The next result, that we present as a proposition, is established in \cite{ghl} and is proved with the same argument  as in \cite[Theorem 4.19]{w}.
    \begin{proposition}\label{iso q car}
        The $q$-character induces a ring isomorphism
        \[\chi_q: K_0(\mathcal{O}^{\sh})\xrightarrow{\simeq}\mathcal{E}_{\ell}.\]
    \end{proposition}
    Then, $ K_0(\mathcal{O}^{\sh})$ has the structure of topological ring inherited by the one on $\mathcal{E}_{\ell}$, with a topological basis made of $[L(\Psi)]$, $\Psi\in\mathfrak{r}$. In the rest of the paper, similar topologies will be used on other \Groth rings.

    \section{Grothendieck rings and cluster algebras}\label{section on cluster structure on gr rings}
    This section is devoted to the collection of results about cluster algebra structures on \Groth rings of categories of representations for some quantum groups, namely the quantum affine Borel algebras and shifted quantum affine algebras. As in the rest of this work, $\g$ denotes a simple Lie algebra of simply-laced type. 
    \subsection{Basic infinite quiver}
    Let us denote $\Tilde{V}:=I\times \ZZ$. In \cite{hl16o} the authors define the quiver $\Tilde{\Gamma}$: its vertex set is $\Tilde{V}$ and the arrows are given by
    \begin{equation}
        (i,r)\to (j,s)\ \Longleftrightarrow\ (c_{i,j}\neq 0\ \text{and}\ s=r+c_{i,j}).
    \end{equation}
    $\Tilde{\Gamma}$ has two isomorphic connected components. We take one of them, we denote it $\Gamma_e$ and we call it the basic infinite quiver. Its vertex set will be denoted $V$ (see Figure \ref{fig:basic quiver A3} ). In the notations of \cite{ghl}, $\Gamma_e$ is the quiver $\Gamma_C$. 

    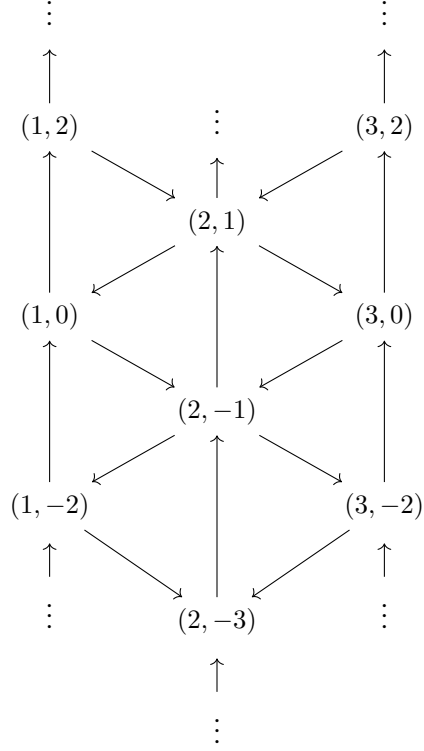
\begin{figure}[H]
        \centering
        
        \[\begin{tikzcd}
    	\vdots && \vdots \\
    	{(1,2)} & \vdots & {(3,2)} \\
    	& {(2,1)} \\
    	{(1,0)} && {(3,0)} \\
    	& {(2,-1)} \\
    	{(1,-2)} && {(3,-2)} \\
    	\vdots & {(2,-3)} & \vdots \\
    	& \vdots
    	\arrow[shorten >=5pt, from=2-1, to=1-1]
    	\arrow[from=2-1, to=3-2]
    	\arrow[shorten >=5pt, from=2-3, to=1-3]
    	\arrow[from=2-3, to=3-2]
    	\arrow[shorten >=5pt, from=3-2, to=2-2]
    	\arrow[from=3-2, to=4-1]
    	\arrow[from=3-2, to=4-3]
    	\arrow[from=4-1, to=2-1]
    	\arrow[from=4-1, to=5-2]
    	\arrow[from=4-3, to=2-3]
    	\arrow[from=4-3, to=5-2]
    	\arrow[from=5-2, to=3-2]
    	\arrow[from=5-2, to=6-1]
    	\arrow[from=5-2, to=6-3]
    	\arrow[from=6-1, to=4-1]
    	\arrow[from=6-1, to=7-2]
    	\arrow[from=6-3, to=4-3]
    	\arrow[from=6-3, to=7-2]
    	\arrow[shorten >=5pt, from=7-1, to=6-1]
    	\arrow[from=7-2, to=5-2]
    	\arrow[shorten >=5pt, from=7-3, to=6-3]
    	\arrow[shorten >=5pt, from=8-2, to=7-2]
        \end{tikzcd}\]
        \caption{$\Gamma_e$ in type $A_3.$}
        \label{fig:basic quiver A3}
    \end{figure}
    \subsection{Categories $\mathcal{O}^{\mathfrak{b},\pm}_{\ZZ}$}\label{sezione categorie O}
    In \cite{hl16o}, the authors have defined the full sub-categories $\mathcal{O}^{\mathfrak{b},\pm}_{\ZZ}$ of $\mathcal{O}^{\mathfrak{b},\pm}$ as follows:
    \begin{definition}[\cite{hl16o}]
        $\mathcal{O}^{\mathfrak{b},\pm}_{\ZZ}$ is the sub-category of $\mathcal{O}^{\mathfrak{b},\pm}$ whose simple constituents have highest $\ell$-weights $\Psi=(\psi_i(z))_{i\in I}$ such that all roots and poles of $\psi_i(z)$ are of the form $q^r$, with $(i,r)\in V$.
    \end{definition}
    In \cite{hl16o}, the authors proved that the Grothendieck rings $ K_0(\mathcal{O}^{\mathfrak{b},\pm}_{\ZZ})$ are isomorphic to a completion of a cluster algebra.
    \begin{definition}\label{def Ae}
        Let $\mathcal{A}_e$ be the cluster algebra with initial seed consisting of the quiver $\Gamma_e$ and initial cluster variables $\{z_{(i,r)}\}_{(i,r)\in V}$. We denote such seed by $\Sigma_e$.
    \end{definition}
    By Laurent phenomenon $\mathcal{A}_e$ is contained in $\ZZ[z_{(i,r)}^{\pm 1}]_{(i,r)\in V}$. 
    \begin{theorem}[{\cite[Theorem 4.2]{hl16o}}]\label{teorema cluster hl16}
        The assignment
        \begin{equation}
            z_{(i,r)}\mapsto \Big[-\frac{r}{2}\omega_i\Big]\Big[L_{i,q^r}^{\mathfrak{b},+}\Big],\qquad (i,r)\in V,
        \end{equation}
        defines an injective morphism of algebras
        \begin{equation}
            f:\mathcal{A}_e\hookrightarrow K_0(\mathcal{O}^{\mathfrak{b},+}_{\ZZ}), 
        \end{equation}
        and the topological closure of $f(\mathcal{A}_e)\otimes_{\ZZ}\mathcal{E}$ is the entire $ K_0(\mathcal{O}^{\mathfrak{b},+}_{\ZZ})$.
    \end{theorem}
    \begin{remark}
        The topological closure of $f(\mathcal{A}_e)\otimes_{\ZZ}\mathcal{E}$ inside $K_0(\mathcal{O}^{\mathfrak{b},+}_{\ZZ})$ is denoted by $f(\mathcal{A}_e)\hat{\otimes}_{\ZZ}\mathcal{E}$, see \cite{hl16o} for more details. We will use an analogous notation for topological closures also in the rest of the paper. 
    \end{remark}
    By Theorem 5.17 in \cite{hl16o}, $ K_0(\mathcal{O}^{\mathfrak{b},+}_{\ZZ})$ and $ K_0(\mathcal{O}^{\mathfrak{b},-}_{\ZZ})$ are isomorphic algebras, thus an analogue of Theorem \ref{teorema cluster hl16} can be written for $ K_0(\mathcal{O}^{\mathfrak{b},-}_{\ZZ})$.
    \subsection{The quantum Grothendieck ring $K_t(\mathcal{O}^{\mathfrak{b},+}_{\ZZ})$}\label{qgrb}
    We recall here the main ingredients of the construction of a quantum \Groth ring for the category $\mathcal{O}^{\mathfrak{b},+}_{\ZZ}$ by Bittmann, see \cite{b21}.\\ The quantum Grothendieck ring $ K_t(\mathcal{O}^{\mathfrak{b},+}_{\ZZ})$ is realized as a quantum cluster algebra, hence it lives inside a quantum torus. In order to define the quantum torus, one needs some information about the so-called Cartan datum, that can be found in \cite[Section 3.2]{b21}. In particular, the entries of the inverse quantum Cartan matrix $\tilde{C}(z)$ ($z$ is an indeterminate) can be expressed as a power series in $z$ as follows.
    For $i,j\in I$,
    \begin{equation}
        \Tilde{C}_{i,j}=\sum_{m=1}^{\infty}\Tilde{C}_{i,j}(m)z^m\quad \in \ZZ[[z]].
    \end{equation}
    Furthermore, for all $i,j\in I$, let $\mathcal{F}_{i,j}:\ZZ\to\ZZ$ be the skew-symmetric map such that, for all $m\geq 0$
    \begin{equation}\label{F}
        \mathcal{F}_{i,j}(m)=-\sum_{\substack{k\geq 1 \\ m\geq 2k-1}}\Tilde{C}_{i,j}(m-2k+1).
    \end{equation}
    \begin{definition}\label{quasi toro lea}
        Let $T_{t,e}$ be the $\ZZ[t^{\pm 1}]$-algebra generated by the $z_{(i,r)}^{\pm 1}$, for $(i,r)\in V$, with non-commutative product $\ast_e$ and $t$-commutation relations
    \begin{equation}
        z_{(i,r)}\ast_e z_{(j,s)}=t^{\mathcal{F}_{i,j}(s-r)}z_{(j,s)}\ast_e z_{(i,r)},\qquad (i,r),(j,s)\in V.
    \end{equation}
    \end{definition}
    Consider now the scalars extension $\ZZ[t^{\pm 1/2}]\otimes_{\ZZ[t^{\pm 1}]}T_{t,e}$ that, by abuse of notation, we keep denoting by $T_{t,e}$.  
    The quantum torus considered in \cite{b21} is 
    \begin{equation}\label{toro di lea}
        \mathcal{T}_{t,e}:=T_{t,e}\hat{\otimes}_{\ZZ}\mathcal{E}.
    \end{equation}
    \begin{example}\label{f}
        For $\g=\sl_2$, $I=\{1\}$, Equation \eqref{F} simplifies, yielding
            \[z_{(1,r)}\ast z_{(1,s)}=t^{f(s-r)}z_{(1,s)}\ast z_{(1,r)},\qquad \forall s,r\in V=2\ZZ,\] where $f : \ZZ\to \ZZ$ is the skew-symmetric function defined by 
        \begin{equation}
            f\vert_{\NN} : m\mapsto \frac{(-1)^m-1}{2}.
        \end{equation}
    \end{example}

    We discuss now the link to quantum cluster algebras. We refer to the original paper \cite{bz05} by Berenstein and Zelevinsky for the definition and properties of these algebras. We give here only the definition of compatible pair, since it is crucial in the sequel. Let us denote by $\llbracket 1,m\rrbracket$ the interval of integers from $1$ to $m$.
    \begin{definition}\label{def compatibile}
       Let $\tilde{B}=(b_{ij})$ be a $m\times n$ integer matrix, with rows labeled by $\llbracket 1,m\rrbracket$ and columns labeled by $\mathbf{ex}\subset \llbracket 1,m\rrbracket$. Let $\Lambda=(\lambda_{ij})_{1\leq i,j\leq m}$ be a $m\times m$ skew-symmetric integer matrix. The couple $(\Lambda, \tilde{B})$ is said to be \emph{compatible} if, for all $i\in \mathbf{ex}$ and $1\leq j\leq m$, we have 
       \[\sum_{k=1}^m b_{ki}\lambda_{kj}=\delta_{ij}d_j,\] for some non-zero integers $d_j$, $j\in\mathbf{ex}$, such that $d_j>0$ for all $j$, or $d_j<0$ for all $j$. In other words, the $n\times m$ matrix $\tilde{D}=\tilde{B}^T\Lambda$ consists of two blocks: the $\mathbf{ex}\times\mathbf{ex}$ diagonal block $D$ with diagonal entries $d_j$ and a zero block.
    \end{definition}
    In particular, if $(\Lambda,\tilde{B})$ is a compatible pair, then $\tilde{B}$ has full rank, see \cite[Proposition 3.3]{bz05}. For completeness we give the definition of compatible pair in the infinite rank case.
    \begin{definition}
        Let $B$ be an infinite integer matrix with a finite number of non-zero entries on each row and column (we call these matrices \emph{locally finite}). Let $\Lambda$ be an infinite skew-symmetric integer matrix. Assume that the rows and columns of $B$ and $\Lambda$ are indexed by a set $J$. The couple $(\Lambda,B)$ is said to be \emph{compatible} if $B^T\Lambda=D$, where $D=\text{diag}(d_j)_{j\in J}$ is an infinite integer matrix such that $d_{j}>0$ for all $j$ in $J$ or $d_{j}<0$ for all $j$ in $J$. The product $B^T\Lambda$ is well defined because $B$ is locally finite.
    \end{definition}
    In \cite{b21} the author considers the same basic infinite quiver $\Gamma_e$ used \cite{hl16o}. Let us denote by $B_e$ its exchange matrix, which is a $V\times V$ skew-symmetric integer matrix. The quantization matrix $\Lambda_e$ encodes the $t$-commutation relations between the initial cluster variables $\{z_{(i,r)}\mid (i,r)\in V\}$. In \cite{b21} $\Lambda_e$ is defined as the $V\times V$ skew-symmetric matrix such that for all $s>r$ 
    \begin{equation}\label{lambda e}
        \Lambda_{(i,r),(j,s)}\coloneqq \mathcal{F}_{i,j}(s-r).
    \end{equation}
    \begin{proposition}[{\cite[Proposition 6.2.4]{b21}}] \label{compatibilita lea}
        The pair $(\Lambda_e,B_e)$ is compatible, in the sense of quantum cluster algebras \cite{bz05}. Moreover, 
        \[B_e^T\Lambda_e=-2\Id.\]
    \end{proposition}

     Let $\mathcal{A}_{t,e}$ be the quantum cluster algebra associated with the mutation-equivalence class of $(B_e,\Lambda_e)$. By quantum Laurent phenomenon \cite[Corollary 5.2]{bz05}, $\mathcal{A}_{t,e}\subset T_{t,e}$.
    \begin{definition}\cite[Definition 6.3.4]{b21}\label{def quantum gr ring lea}
        The quantum Grothendieck ring for the category $\mathcal{O}^{\mathfrak{b},+}_{\ZZ}$ is defined as
        \[ K_t(\mathcal{O}^{\mathfrak{b},+}_{\ZZ}):=\mathcal{A}_{t,e}\hat{\otimes}\mathcal{E},\] where the completed tensor product has been defined in \eqref{toro di lea}. Thus, $ K_t(\mathcal{O}^{\mathfrak{b},+}_{\ZZ})$ is a $\mathcal{E}[t^{\pm 1/2}]$-subalgebra of $\mathcal{T}_{t,e}$.
    \end{definition}
    A crucial point of Bittmann's result is that her construction is compatible with the quantum \Groth ring $ K_t(\mathcal{C}_{\ZZ}^-)$, where $\mathcal{C}^-$ is a full subcategory of the category of finite-dimensional $\mathcal{U}_q(\hat{\g})$-modules $\mathcal{C}$:
    \begin{theorem}[{\cite[Corollary 7.3.5]{b21b}}]\label{inclusione tori hl e lea}
        There is an inclusion of rings:
        \[ K_t(\mathcal{C}_{\ZZ}^-)\subset  K_t(\mathcal{O}^{\mathfrak{b},+}_{\ZZ}).\]
    \end{theorem}
    \begin{remark}
    The definition of the quantum Grothendieck ring $K_t(\mathcal{O}^{\mathfrak{b},-})$ is given in \cite[Sections 4.1.2, 4.2.2]{theseBittmann}. For our purposes it is sufficient to recall that $K_t(\mathcal{O}^{\mathfrak{b},-})$ is isomorphic to $K_t(\mathcal{O}^{\mathfrak{b},+})$. This is what we would expect, since at the classical level we have that $K_0(\mathcal{O}^{\mathfrak{b},+})$ is isomorphic to $K_0(\mathcal{O}^{\mathfrak{b},-})$.
    \end{remark}
    \subsection{Geiss--Hernandez--Leclerc construction}\label{ghl construction}
    Recall that $s_i$, $i\in I$, denotes a simple reflection in the Weyl group of $\g$. First, we recall the construction of the quiver $\Gamma_c$. 
    \subsubsection{The cluster algebra $\mathcal{A}_{w_0}$}
    
    \begin{definition}\cite[Definition 3.1]{ghl}\label{procedura ghl}
        Consider the basic infinite quiver $\Gamma_e$ with vertex set $V$. For $(i,r)\in V$, we obtain $\Gamma_{s_i,r}$ from $\Gamma_e$ following the procedure:
        \begin{enumerate}
            \item insert a new vertex $\ast$ between vertices $(i,r)$ and $(i,r-2)$;\\
            \item replace the arrow $(i,r)\leftarrow (i,r-2)$ by a pair of arrows $(i,r)\rightarrow \ast \leftarrow (i,r-2)$;\\
            \item for $j$ with $c_{i,j}<0$, replace the arrow $(i,r)\rightarrow (j, r+c_{i,j})$ by an arrow $\ast \rightarrow (j,r+c_{i,j})$;\\
            \item change the labels of the vertices on the lower half of column $i$ as follows:
            \[\ast\mapsto (i,r-2),\qquad (i,r-2k)\mapsto (i,r-2k-2)\text{ for }k\geq 1.\]
        \end{enumerate}
    \end{definition}
    To emphasize the roles of vertices $(i,r)$ and $(i,r-2)$, we paint $(1,r)$ in red and $(1,r-2)$ in green. 
        \begin{figure}[H]
        \centering
        
    \[\begin{tikzcd}
    	\vdots & \vdots & \vdots \\
    	{(1,2)} & {} & {(3,2)} \\
    	& {(2,1)} \\
    	{\color{red}{(1,0)}} && {(3,0)} \\
    	{\color{green}{(1,-2)}} \\
    	& {(2,-1)} \\
    	{(1,-4)} && {(3,-2)} \\
    	& {(2,-3)} \\
    	{(1,-6)} & {} & {(3,-4)} \\
    	\vdots & \vdots & \vdots
    	\arrow[from=2-1, to=1-1]
    	\arrow[from=2-1, to=3-2]
    	\arrow[from=2-3, to=1-3]
    	\arrow[from=2-3, to=3-2]
    	\arrow[from=3-2, to=1-2]
    	\arrow[from=3-2, to=4-1]
    	\arrow[from=3-2, to=4-3]
    	\arrow[from=4-1, to=2-1]
    	\arrow[from=4-1, to=5-1]
    	\arrow[from=4-3, to=2-3]
    	\arrow[from=4-3, to=6-2]
    	\arrow[from=5-1, to=6-2]
    	\arrow[from=6-2, to=3-2]
    	\arrow[from=6-2, to=7-1]
    	\arrow[from=6-2, to=7-3]
    	\arrow[from=7-1, to=5-1]
    	\arrow[from=7-1, to=8-2]
    	\arrow[from=7-3, to=4-3]
    	\arrow[from=7-3, to=8-2]
    	\arrow[from=8-2, to=6-2]
    	\arrow[from=8-2, to=9-1]
    	\arrow[from=8-2, to=9-3]
    	\arrow[from=9-1, to=7-1]
    	\arrow[from=9-3, to=7-3]
    	\arrow[from=10-1, to=9-1]
    	\arrow[from=10-2, to=8-2]
    	\arrow[from=10-3, to=9-3]
    \end{tikzcd}\]
        \caption{The quiver $\Gamma_{s_1,0}$ in type $A_3$.}
        \label{fig:enter-label}
    \end{figure}
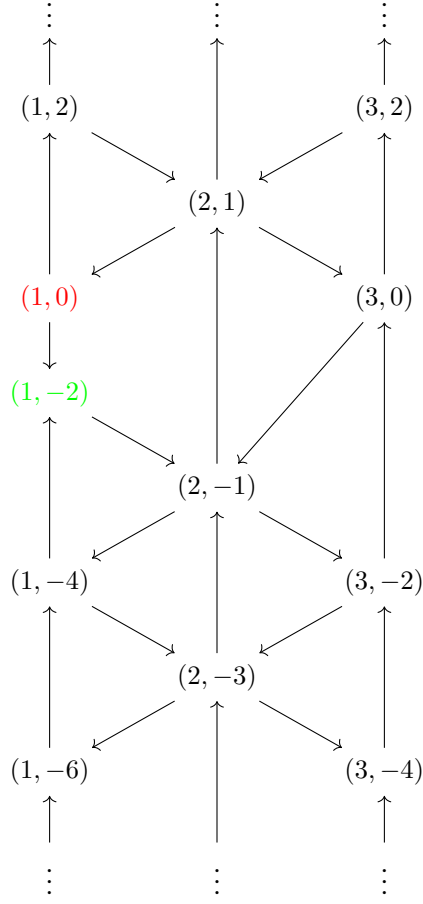
    
    Moreover, as stated in \cite[Lemma 3.2]{ghl}, the mutation of $\Gamma_{s_i,r}$ at vertex $(i,r)$ produces the quiver $\Gamma_{s_i,r+2}$, and the mutation of $\Gamma_{s_i,r}$ at vertex $(i,r-2)$ produces the quiver $\Gamma_{s_i,r-2}$.

    \begin{figure}[H]\label{vari quiver per sl2}
        \centering
        \begin{tikzcd}
    	\vdots &&& \vdots &&& \vdots \\
    	\\
    	2 &&& 2 &&& \color{red}2 \\
    	\\
    	\color{red}0 &&& 0 &&& \color{green}0 \\
    	\\
    	\color{green}{-2} &&& \color{red}{-2} &&& {-2} \\
    	{} \\
    	{-4} &&& \color{green}{-4} &&& {-4} \\
    	\\
    	\vdots &&& \vdots &&& \vdots
    	\arrow[from=3-1, to=1-1]
    	\arrow[from=3-4, to=1-4]
    	\arrow[from=3-7, to=1-7]
    	\arrow[from=3-7, to=5-7]
    	\arrow[from=5-1, to=3-1]
    	\arrow[from=5-1, to=7-1]
    	\arrow[from=5-4, to=3-4]
    	\arrow[from=7-4, to=5-4]
    	\arrow[from=7-4, to=9-4]
    	\arrow[from=7-7, to=5-7]
    	\arrow[from=9-1, to=7-1]
    	\arrow[from=9-7, to=7-7]
    	\arrow[from=11-1, to=9-1]
    	\arrow[from=11-4, to=9-4]
    	\arrow[from=11-7, to=9-7]
    \end{tikzcd}
        \caption{In type $A_1$, $w_0=s$, the simple reflection. From left to right: the quivers $\Gamma_{s,0}$, $\Gamma_{s,-2}$ and $\Gamma_{s,2}$.}
        \label{fig:alcuni quiver A1}
    \end{figure}
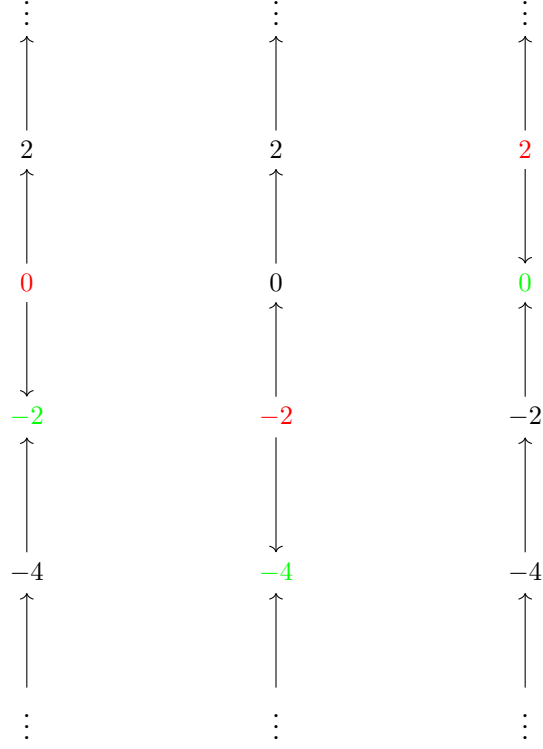
    Next, let $\mathbf{Q}$ be a Dynkin quiver of the same type as $\g$. Note that the basic infinite quiver $\Gamma_e$ is isomorphic to the Auslander-Reiten quiver of the bounded derived category $D^{b}(K\mathbf{Q})$ of the path algebra $K\mathbf{Q}$, over a field $K$, where the vertical arrows in $\Gamma_e$ correspond to Auslander--Reiten translations. For $i\in I$, let $s_i(\mathbf{Q})$ be the quiver obtained from $\mathbf{Q}$ by reversing all the arrows incident at vertex $i$. Consider a reduced word $w=s_{i_1}\cdots s_{i_k}\in W$. The expression $(i_1,\ldots, i_k)$ is adapted to $\mathbf{Q}$ if $i_k$ is a source for $\mathbf{Q}$, $i_{k-1}$ is a source for $s_{i_k}(\mathbf{Q})$, and so on. Then, for every $c=s_{i_1}\cdots s_{i_n}$ Coxeter element of $W$, there is a unique $\mathbf{Q}$ such that $(i_1,\ldots, i_n)$ is adapted to $\mathbf{Q}$. 
    \begin{definition}\label{quiver gamma c}
        Fix $c$ a Coxeter element of $W$ and consider the corresponding quiver $\mathbf{Q}$. We can identify the Auslander--Reiten quiver of the abelian category $\operatorname{mod}(K\mathbf{Q})$ with a finite full subquiver $G_c$ of $\Gamma_e$. Now, at each vertex of $G_c$ we perform the procedure described in Definition \ref{procedura ghl}. We denote the resulting quiver $\Gamma_c$. The choice of $G_c$ in $\Gamma_e$ is not unique, so the same goes for $\Gamma_c$, but in fact different choices give quivers that are all mutation equivalent, see \cite[Proposition 3.4]{ghl}. 
    \end{definition}
    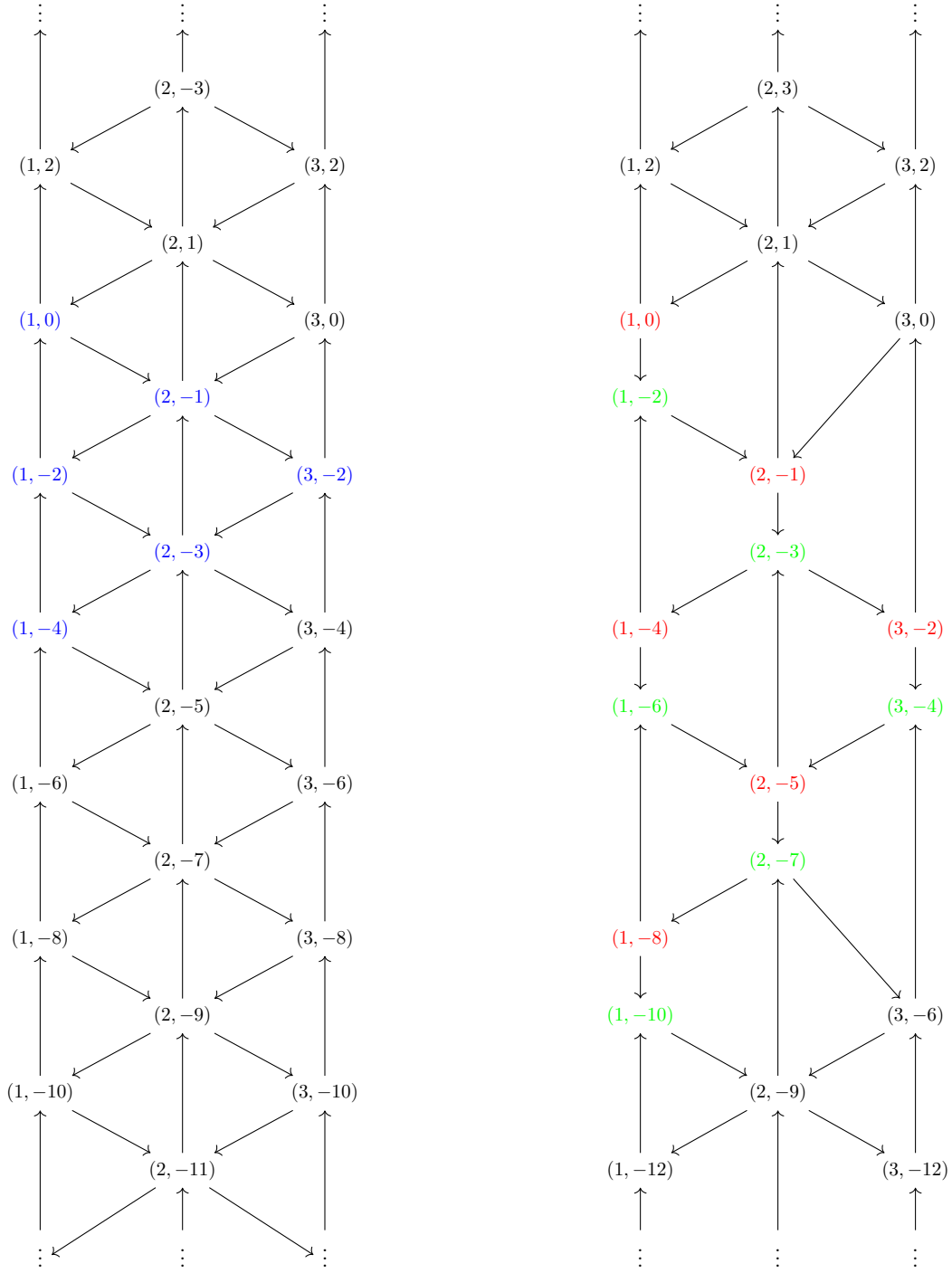
\begin{figure}[H]
        \centering
        \begin{tikzcd}[scale cd=.8]
    	\vdots & \vdots & \vdots &&&& \vdots & \vdots & \vdots \\
    	& {(2,-3)} &&&&&& {(2,3)} \\
    	{(1,2)} && {(3,2)} &&&& {(1,2)} && {(3,2)} \\
    	& {(2,1)} &&&&&& {(2,1)} \\
    	{\color{blue}{(1,0)}} && {(3,0)} &&&& {\color{red}{(1,0)}} && {(3,0)} \\
    	& {\color{blue}{(2,-1)}} &&&&& {\color{green}{(1,-2)}} \\
    	{\color{blue}{(1,-2)}} && {\color{blue}{(3,-2)}} &&&&& {\color{red}{(2,-1)}} \\
    	& {\color{blue}{(2,-3)}} &&&&&& {\color{green}{(2,-3)}} \\
    	{\color{blue}{(1,-4)}} && {(3,-4)} &&&& {\color{red}{(1,-4)}} && {\color{red}{(3,-2)}} \\
    	& {(2,-5)} &&&&& {\color{green}{(1,-6)}} && {\color{green}{(3,-4)}} \\
    	{(1,-6)} & {} & {(3,-6)} &&&&& {\color{red}{(2,-5)}} \\
    	& {(2,-7)} &&&&&& {\color{green}{(2,-7)}} \\
    	{(1,-8)} && {(3,-8)} &&&& {\color{red}{(1,-8)}} \\
    	& {(2,-9)} &&&&& {\color{green}{(1,-10)}} && {(3,-6)} \\
    	{(1,-10)} && {(3,-10)} &&&&& {(2,-9)} \\
    	& {(2,-11)} &&&&& {(1,-12)} && {(3,-12)} \\
    	\vdots & \vdots & \vdots &&&& \vdots & \vdots & \vdots
    	\arrow[from=2-2, to=1-2]
    	\arrow[from=2-2, to=3-1]
    	\arrow[from=2-2, to=3-3]
    	\arrow[from=2-8, to=1-8]
    	\arrow[from=2-8, to=3-7]
    	\arrow[from=2-8, to=3-9]
    	\arrow[from=3-1, to=1-1]
    	\arrow[from=3-1, to=4-2]
    	\arrow[from=3-3, to=1-3]
    	\arrow[from=3-3, to=4-2]
    	\arrow[from=3-7, to=1-7]
    	\arrow[from=3-7, to=4-8]
    	\arrow[from=3-9, to=1-9]
    	\arrow[from=3-9, to=4-8]
    	\arrow[from=4-2, to=2-2]
    	\arrow[from=4-2, to=5-1]
    	\arrow[from=4-2, to=5-3]
    	\arrow[from=4-8, to=2-8]
    	\arrow[from=4-8, to=5-7]
    	\arrow[from=4-8, to=5-9]
    	\arrow[from=5-1, to=3-1]
    	\arrow[from=5-1, to=6-2]
    	\arrow[from=5-3, to=3-3]
    	\arrow[from=5-3, to=6-2]
    	\arrow[from=5-7, to=3-7]
    	\arrow[from=5-7, to=6-7]
    	\arrow[from=5-9, to=3-9]
    	\arrow[from=5-9, to=7-8]
    	\arrow[from=6-2, to=4-2]
    	\arrow[from=6-2, to=7-1]
    	\arrow[from=6-2, to=7-3]
    	\arrow[from=6-7, to=7-8]
    	\arrow[from=7-1, to=5-1]
    	\arrow[from=7-1, to=8-2]
    	\arrow[from=7-3, to=5-3]
    	\arrow[from=7-3, to=8-2]
    	\arrow[from=7-8, to=4-8]
    	\arrow[from=7-8, to=8-8]
    	\arrow[from=8-2, to=6-2]
    	\arrow[from=8-2, to=9-1]
    	\arrow[from=8-2, to=9-3]
    	\arrow[from=8-8, to=9-7]
    	\arrow[from=8-8, to=9-9]
    	\arrow[from=9-1, to=7-1]
    	\arrow[from=9-1, to=10-2]
    	\arrow[from=9-3, to=7-3]
    	\arrow[from=9-3, to=10-2]
    	\arrow[from=9-7, to=6-7]
    	\arrow[from=9-7, to=10-7]
    	\arrow[from=9-9, to=5-9]
    	\arrow[from=9-9, to=10-9]
    	\arrow[from=10-2, to=8-2]
    	\arrow[from=10-2, to=11-1]
    	\arrow[from=10-2, to=11-3]
    	\arrow[from=10-7, to=11-8]
    	\arrow[from=10-9, to=11-8]
    	\arrow[from=11-1, to=9-1]
    	\arrow[from=11-1, to=12-2]
    	\arrow[from=11-3, to=9-3]
    	\arrow[from=11-3, to=12-2]
    	\arrow[from=11-8, to=8-8]
    	\arrow[from=11-8, to=12-8]
    	\arrow[from=12-2, to=10-2]
    	\arrow[from=12-2, to=13-1]
    	\arrow[from=12-2, to=13-3]
    	\arrow[from=12-8, to=13-7]
    	\arrow[from=12-8, to=14-9]
    	\arrow[from=13-1, to=11-1]
    	\arrow[from=13-1, to=14-2]
    	\arrow[from=13-3, to=11-3]
    	\arrow[from=13-3, to=14-2]
    	\arrow[from=13-7, to=10-7]
    	\arrow[from=13-7, to=14-7]
    	\arrow[from=14-2, to=12-2]
    	\arrow[from=14-2, to=15-1]
    	\arrow[from=14-2, to=15-3]
    	\arrow[from=14-7, to=15-8]
    	\arrow[from=14-9, to=10-9]
    	\arrow[from=14-9, to=15-8]
    	\arrow[from=15-1, to=13-1]
    	\arrow[from=15-1, to=16-2]
    	\arrow[from=15-3, to=13-3]
    	\arrow[from=15-3, to=16-2]
    	\arrow[from=15-8, to=12-8]
    	\arrow[from=15-8, to=16-7]
    	\arrow[from=15-8, to=16-9]
    	\arrow[from=16-2, to=14-2]
    	\arrow[from=16-2, to=17-1]
    	\arrow[from=16-2, to=17-3]
    	\arrow[from=16-7, to=14-7]
    	\arrow[from=16-9, to=14-9]
    	\arrow[from=17-1, to=15-1]
    	\arrow[from=17-2, to=16-2]
    	\arrow[from=17-3, to=15-3]
    	\arrow[from=17-7, to=16-7]
    	\arrow[from=17-8, to=15-8]
    	\arrow[from=17-9, to=16-9]
    \end{tikzcd}
        \caption{On the left: in blue the subquiver $G_c$ for $c=s_1s_2s_3$ in type $A_3$. On the right: an example of quiver $\Gamma_c$.}
        \label{fig:quiver type A3}
    \end{figure}
    \begin{definition}\label{def A w0}
    Let $\mathcal{A}_{w_0}$ be the cluster algebra with an initial seed given by the quiver $\Gamma_c$. 
    \end{definition}
    \begin{remark}
        Note that the construction in Definition \ref{quiver gamma c} is proposed in \cite[Section 3.4]{ghl} as an alternative recipe to the main one, presented in Section 3.3 of the same paper, to obtain an initial quiver for $\mathcal{A}_{w_0}$. Then, for any $c$, $c'$ Coxeter elements in $W$, the quivers $\Gamma_c$ and $\Gamma_{c'}$ are mutation equivalent.
    \end{remark}
    We will use the following definition.
    \begin{definition}
        Let $\mathcal{A}$ and $\mathcal{A}'$ be two cluster algebras. We say that a map $f:\mathcal{A}\to\mathcal{A}'$ is an \emph{embedding of cluster algebras} if 
        \begin{itemize}
            \item $f$ is an injective ring morphism;
            \item for every cluster variable $x$ in $\mathcal{A}$, $f(x)$ is a cluster variables in $\mathcal{A}'$.
        \end{itemize}
   \end{definition}

    \begin{remark}\label{remark sullamutazione infinita}
        The sequence of mutations at each green vertex of $\Gamma_c$, produces a quiver which is just a one step downward translation of the initial one (note that the order of these mutations is irrelevant since they commute with each other). Hence, repeating this process \emph{infinitely many times}, we can push down the portion of the quiver consisting of green and red vertices (that we will sometimes call the ``irregular'' part) and recover the quiver $\Gamma_e$. Thus, we will often refer to $\Gamma_e$ as the "limit" quiver and to $\Sigma_e$ (see Definition \ref{def Ae}) as the \emph{limit} reference seed. However, this limit procedure is irreversible, in the sense that there is no way to obtain the quiver $\Gamma_c$ from the quiver $\Gamma_e$ by mutation.
    \end{remark}
    Since in $\mathcal{A}_e$ any cluster variable is obtained from the initial ones by a finite sequence of mutations, in particular it can be obtained applying the same sequence of mutations to one of the initial seeds of $\mathcal{A}_{w_0}$. This means that each cluster variable of $\mathcal{A}_e$ corresponds to a cluster variable of $\mathcal{A}_{w_0}$. Thus we have the following.
    \begin{proposition} There is an embedding of cluster algebras
        \begin{equation}
            \mathcal{A}_e\hookrightarrow \mathcal{A}_{w_0}.
        \end{equation} 
    \end{proposition}

    The proof will be made explicit in Lemma \ref{zzz} below.

    In order to continue, some definitions are needed.
    \begin{definition}[{\cite{ghl}}]
        The height function on the quiver $\mathbf{Q}$ is 
        \[l_c:I\to \ZZ_{\geq 0}\] uniquely determined by the following properties:
        \begin{itemize}
            \item $0\in \mathrm{Im}(l_c)$;
            \item If there is an arrow $i\to j$ in $\mathbf{Q}$, then $l_c(i)=l_c(j)+1$.
        \end{itemize}
    \end{definition}
    A total order on the vertex set $V$ is defined by:
    \begin{equation}\label{total order}
        (i,l_c(i)+2a)<(j,l_c(j)+2b)\Longleftrightarrow (a<b)\ \text{or}\ (a=b,\ i<j).
    \end{equation}
    Thus, when considering matrices with columns and rows indexed by $V$, we assume to use this order.

    \emph{Important convention:}  \textbf{Fix a Coxeter element $c$. Let us call the associated seed $\Sigma_c$. From now, if not differently specified, we fix as initial quiver for $\mathcal{A}_{w_0}$ the quiver $\Gamma_c$ such that its highest red vertex (with respect to the total order \eqref{total order}) is of the form $(i,0)$.}
    \subsubsection{Stabilized $g$-vectors}
    As observed in Remark \ref{remark sullamutazione infinita}, applying mutation at each green vertex produces a quiver which is just $\Gamma_c$ with a one-step downward translation of the labels. Let us denote it by $\Gamma_c^{(1)}$.
    The same procedure applied to $\Gamma_c^{(1)}$ produces the quiver $\Gamma_c^{(2)}$, that is a two-steps downward translation of $\Gamma_c$. In this way we get an infinite sequence of quivers $(\Gamma_c^{(m)})_{m\geq 0}$, where $\Gamma_c^{(0)}$ is understood to be $\Gamma_c$.
    More precisely, for each $k\in \NN$, we denote $V_{\grn}^{(k)}\subset V$ the subset of green vertices of $\Gamma_c^{(k)}$. Define
    \[\mu_{\grn}^{(k)}:=\prod_{(i,l)\in V_{\grn}^{(k)}}\mu_{(i,l)}.\] Thus,
    \begin{equation}\label{successione di quiver}
        \Gamma_c^{(k+1)}=\mu_{\grn}^{(k)}(\Gamma_c^{(k)}).
    \end{equation} 
    Hence we use the notation:
    \begin{equation}
        \lim_{m\to \infty}\Gamma_c^{(m)}=\Gamma_e.
    \end{equation}
    For each quiver $\Gamma_c^{(m)}$ $(m\in \NN)$, let us call $\Sigma_c^{(m)}$ the associated seed. Let $\Sigma_e$ be the seed of $\mathcal{A}_e$ associated with $\Gamma_e$. We denote by $\ZZ^{(V)}$ the free $\ZZ$-module with a basis given by $V$. Every element in $\ZZ^{(V)}$ has only finitely many nonzero components. We denote by $\boldsymbol{e}_{(i,r)}$ the $(i,r)$-th vector of the standard basis of $\ZZ^{(V)}$. Then, for each vertex $(i,r)\in V$ of the quiver $\Gamma_c$, the $g$-vector of the corresponding initial cluster variable with respect to the reference seed $\Sigma_c^{(m)}$ is denoted by $g^{(m)}_{(i,r)}$ and it is an element of $\ZZ^{(V)}$. A natural question is to ask how the sequence $(g^{(m)}_{(i,r)})_{m\geq 0}$ behaves, for each $(i,r)\in V$. Since we have sequences of $g$-vectors $(g^{(m)}_{(i,r)})_{m\geq 0}$ for all $(i,r)\in V$, we also have a sequence of $G$-matrices $(G^{(m)})_{m\geq 0}$, that are $V\times V$ matrices such that the columns of $G^{(m)}$ are the $g^{(m)}_{(i,r)}$, for each $(i,r)\in V$. Note that 
    \[g^{(0)}_{(i,r)}=\boldsymbol{e}_{(i,r)},\] the $(i,r)$-th element of the standard basis of $\ZZ^{(V)}$, since the reference seed in this case coincides with the initial seed. 
    For all $m\in \ZZ$ we consider
    \[I^{(0)}_{\grn}(m):=\{i\in I\mid (i, l_c(i)+2m)\in \Gamma_c\ \text{is green}\}.\]
    Moreover, set
    \[h_c:=\text{min}\{m\in\ZZ\mid I_{\grn}^{(0)}(m)\neq \emptyset\}.\] Note that $h_c<0$.\\
    Following \cite[Definition 4.6]{ghl}, for $m\in \ZZ$, we define
        \[I(m):=\{(i,l_c(i)+2m)\mid i\in I\}\subset V.\] Note that $I(m)$ has the same cardinality as $I$.

    To state the following theorem we need the matrices $\{\mathbf{T}_m\}_{m\in\ZZ}$ which belong to $\ZZ^{I\times I}$ and are product of reflection matrices (see the complete definition in \cite[Section 4.3.2]{ghl})

    \begin{theorem}[{\cite[Theorem 4.11]{ghl}\label{teo stabilized g vectors}}]
        For each $k\geq 0$, the $G$-matrix $G^{(k)}$ is of block diagonal form 
        \[\mathrm{diag}\left(G^{(k)}(m)\mid m\in\ZZ \right),\]
         where each diagonal block $G^{(k)}(m)$ corresponds to the vertex set $I(m)\subset V$. We have
        \[G^{(k)}(m)=\mathbf{T}_{m+k-1}\cdots \mathbf{T}_{m+1}\cdot \mathbf{T}_m.\] In particular, the sequence of $G$-matrices stabilizes, in the sense that, for $k\gg|m|$, we have
        \begin{equation}
            G^{(k)}(m)=\begin{cases}
                \Id_n\qquad & \text{if}\ m>0,\\
                \mathbf{T}_{-1}\mathbf{T}_{-2}\cdots \mathbf{T}_m\quad &\text{if}\ h_c\leq m\leq -1,\\
                \mathbf{T}_{-1}\mathbf{T}_{-2}\cdots \mathbf{T}_{h_c}\quad &\text{if}\ m\leq h_c.
            \end{cases}
        \end{equation}
        Hence the sequence of $G$-matrices $(G^{(k)})_{k\geq 0}$ has a well defined limit
        \[G^{(\infty)}:=\lim_{k\to \infty }G^{(k)}.\] This means that for every initial cluster variable, the sequence of $g$-vectors stabilizes after finitely many steps. 
    \end{theorem}
    As a consequence, if $(i,r)\in I(m)$, then the support of $g^{(k)}_{(i,r)}$ is contained in $I(m)$.

    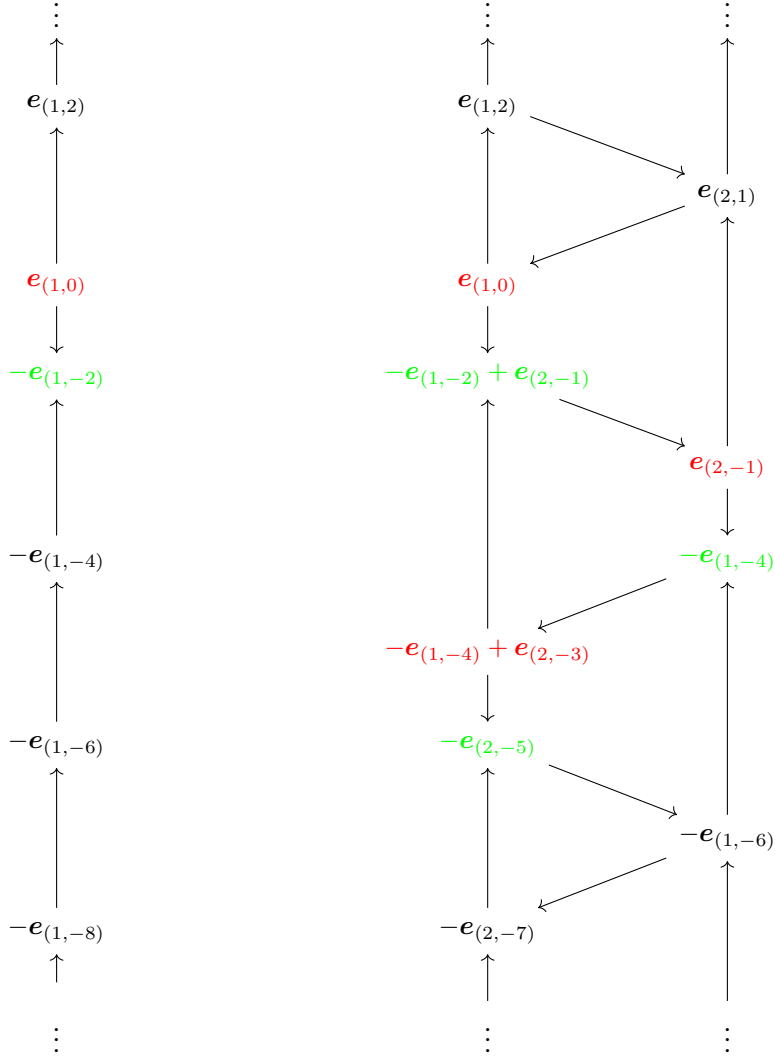
\begin{figure}[H]
        \centering
       
    \[\begin{tikzcd}
    	\vdots &&&& \vdots & \vdots \\
    	{\boldsymbol{e}_{(1,2)}} &&&& {\boldsymbol{e}_{(1,2)}} \\
    	&&&&& {\boldsymbol{e}_{(2,1)}} \\
    	{\color{red}{\boldsymbol{e}_{(1,0)}}} &&&& {\color{red}{\boldsymbol{e}_{(1,0)}}} \\
    	{\color{green}{-\boldsymbol{e}_{(1,-2)}}} &&&& {\color{green}{-\boldsymbol{e}_{(1,-2)}+\boldsymbol{e}_{(2,-1)}}} \\
    	&&&&& {\color{red}{\boldsymbol{e}_{(2,-1)}}} \\
    	{-\boldsymbol{e}_{(1,-4)}} &&&&& {\color{green}{-\boldsymbol{e}_{(1,-4)}}} \\
    	&&&& {\color{red}{-\boldsymbol{e}_{(1,-4)}+\boldsymbol{e}_{(2,-3)}}} \\
    	{-\boldsymbol{e}_{(1,-6)}} &&&& {\color{green}{-\boldsymbol{e}_{(2,-5)}}} \\
    	&&&&& {-\boldsymbol{e}_{(1,-6)}} \\
    	{-\boldsymbol{e}_{(1,-8)}} &&&& {-\boldsymbol{e}_{(2,-7)}} \\
    	\vdots &&&& \vdots & \vdots
    	\arrow[from=2-1, to=1-1]
    	\arrow[from=2-5, to=1-5]
    	\arrow[from=2-5, to=3-6]
    	\arrow[from=3-6, to=1-6]
    	\arrow[from=3-6, to=4-5]
    	\arrow[from=4-1, to=2-1]
    	\arrow[from=4-1, to=5-1]
    	\arrow[from=4-5, to=2-5]
    	\arrow[from=4-5, to=5-5]
    	\arrow[from=5-5, to=6-6]
    	\arrow[from=6-6, to=3-6]
    	\arrow[from=6-6, to=7-6]
    	\arrow[from=7-1, to=5-1]
    	\arrow[from=7-6, to=8-5]
    	\arrow[from=8-5, to=5-5]
    	\arrow[from=8-5, to=9-5]
    	\arrow[from=9-1, to=7-1]
    	\arrow[from=9-5, to=10-6]
    	\arrow[from=10-6, to=7-6]
    	\arrow[from=10-6, to=11-5]
    	\arrow[from=11-1, to=9-1]
    	\arrow[from=11-5, to=9-5]
    	\arrow[shorten <=7pt, from=12-1, to=11-1]
    	\arrow[from=12-5, to=11-5]
    	\arrow[from=12-6, to=10-6]
    \end{tikzcd}\]
        \caption{Stabilized $g$-vectors in type $A_1$ and $A_2$.}
        \label{fig:stabilized g vectors}
    \end{figure}

    Recall, for later use, \cite[Theorem 4.15]{ghl}, where an Auslander--Reiten type algorithm to compute stabilized $g$-vectors is presented.

    Otherwise, one can compute stabilized $g$-vectors using Chari's braid group action \cite{chari}.
    \begin{definition}\label{azione braid su base standard}
    For $i\in I$, we define an automorphism $\Theta_i$ on the free $\ZZ$-module $\ZZ^{(V)}$ by:
    \[\Theta_i(\boldsymbol{e}_{(j,r)}):=
    \begin{cases}
        \boldsymbol{e}_{(j,r)}\quad \mathrm{if}\  i\neq j\\
        -\boldsymbol{e}_{(i,r-2)}+\sum\limits_{k:\ c_{ik}=-1}\boldsymbol{e}_{(k,r-1)}
    \end{cases}\qquad\text{for all }(j,r)\in V.\]
    \end{definition}
    By \cite{chari}, these automorphisms generate an action of the braid group associated with the Weyl group of $\g$ on $\ZZ^{(V)}$.
    By construction of $\Gamma_c$, if we read the first indices of the green vertices of $\Gamma_c$ from top to bottom, we obtain a sequence $(i_1,...,i_N)$ such that $s_{i_1}\cdots s_{i_N}$ is a reduced expression for $w_0$, the longest element of $W$. Let $(i_1,a_1),...,(i_N,a_N)$ be the corresponding sequence of green vertices in $\Gamma_c$. Then, for $1\leq t\leq N$, let $s_t:=\#\{j<t\mid i_j=i_t\}$. Moreover, for $i\in I$, let $(i,m_i)$ denote the highest red vertex in the column $i$ of $\Gamma_c$. The next result shows that the stabilized $g$-vectors can be computed via braid group action.
    \begin{proposition}[{\cite[Proposition 4.19]{ghl}}]\label{g vector braid action}
        For a vector $\boldsymbol{v}=(v_{(i,l)})_{(i,l)\in V}\in\ZZ^{(V)}$, and $s\in\ZZ$, we consider the shifted vector $\boldsymbol{v}[s]=(v'_{(i,l)})_{(i,l)\in V}$, defined by:
         \[v'_{(i,l)}:=v_{(i,l-2s)},\  (i,l)\in V.\]
        Then, for $1\leq t\leq N$, we have
        \begin{equation}
            g^{(\infty)}_{(i_t,a_t)}=\Theta_{i_1}\Theta_{i_2}\cdots \Theta_{i_t}\left(\boldsymbol{e}_{(i_t,m_{i_t})}\right)[-s_t].
        \end{equation}
    \end{proposition}
\begin{remark}\label{remark g vector braid action}
    Although Proposition \ref{g vector braid action} applies only to green vertices of $\Gamma_c$, using shifts as in \cite[Theorem 4.15]{ghl}, we can use the braid group action to compute all stabilized $g$-vectors. Namely, for every $(i,r)\in V$, the stabilized $g$-vector can be written as
    \[g^{(\infty)}_{(i,r)}=\Theta_{i_1}\cdots\Theta_{i_t}\left(\boldsymbol{e}_{(i_t,m_{i_t})}\right)[s]\] for well defined $1\leq t\leq N$ and $s\in \ZZ$.
\end{remark}

    \subsubsection{Variables}
    Now we recap some results from \cite[Section 5]{ghl} and \cite[Section 6]{ghl} about certain polynomial and power series ring that are related with the representation theory of $\mathcal{U}_q(\hat{\g})$ and its shifted versions. Moreover we add some clarifications about results that are ``implicit'' in \cite{ghl}. Note that in \cite[Section 5]{ghl} a lot of importance is given to a Weyl group action defined on a certain ring $\Pi'$. However we do not need it for our purposes, hence we can simplify the content of that section. Consider the ring of Laurent polynomials
    \[\mathcal{Y}':=\ZZ\big[\Psi_{i,q^r}^{\pm 1},[\lambda]\mid (i,r)\in V,\lambda\in P_{\QQ}\big],\] where, coherently with Example \ref{rappresentazione costante}, the multiplication rule in $[P_{\QQ}]$ is $[\lambda][\lambda']=[\lambda+\lambda']$, for all $\lambda,\lambda'$ in $P_{\QQ}$. 
    
    Chari braid group action of Definition \ref{azione braid su base standard} can be equivalently defined on $\mathcal{Y}'$. For all $i\in I$, we have
     \begin{equation}\label{chari braid T}
         T_i:\mathcal{Y}'\to \mathcal{Y}'
     \end{equation}
    defined by the formulas
    \[T_i(\Psi_{i,q^r}):=\Psi_{i,q^{r+2}}^{-1}\prod_{j:c_{ij}=-1}\Psi_{j,q^r},\quad T_i(\Psi_{k,q^r}):=\Psi_{k,q^r},\ (k\neq i),\quad T_i([\omega]):=[\omega],\ \omega\in P_{\QQ}.\]
    Note that this definition holds for simply-laced type Lie algebras. Chari has defined such an action for all types (and a generalization is also defined in \cite[Section 3.2]{fh23}). 

    \begin{definition}
         For $(i,r)\in V$, we define the following monomials in $\mathcal{Y}'$:
         \begin{align}
             Y_{i,q^{r+1}}&:=[\omega_i]\frac{\Psi_{i,q^{r}}}{\Psi_{i,q^{r+2}}};\\
             A_{i,q^r}&:= Y_{i,q^{r-1}}Y_{i,q^{r+1}}\prod_{j:c_{ji}=-1}Y_{j,q^r}^{-1};\\
             \widetilde{\Psi}_{i,q^r}&:= \Psi^{-1}_{i,q^r}\prod_{j:c_{ij}=-1}\Psi_{j,q^{r+1}} \label{psi tilde}
             \end{align} Note the shift of parity in the definition of the $Y$'s.
     \end{definition}

    Now we consider a \emph{complete topological} ring $\widetilde{\mathcal{Y}'}$ obtained from $\mathcal{Y}'$. The construction is explained in \cite[Section 3.7]{fh23} and we omit it. However, we will need the following.
    \begin{remark}
    Recall the ring $\mathcal{E}_{\ell}$ defined in Section \ref{section borel}. Let us define $\mathcal{E}_{\ell,\ZZ}$ to be the $[P_{\QQ}]$-topological subring of $\mathcal{E}_{\ell}$ generated by $[\Psi_{i,q^r}^{\pm 1}]$, such that $(i,r)\in V$.
        Then, we can identify $\widetilde{\mathcal{Y}'}$ with $\mathcal{E}_{\ell,\ZZ}$. Indeed, we can identify the variable $\Psi_{i,q^r}$ with the $\ell$-weight $\Psi_{i,q^r}$ and the variable $[\lambda]$, with the $\ell$-weight $[\lambda]$ (for all $(i,r)\in V$ and for all $\lambda\in P_{\QQ}$) so that the notations are coherent.
    \end{remark}
    In \cite[Section 3.7]{fh23} the authors build a ring $\Pi'$ as the direct sum of complete topological rings parameterized by the elements of the Weyl group $W$ of $\g$. In particular, the component of $\Pi'$ indexed by the identity element is precisely $\widetilde{\mathcal{Y}'}$. Let us denote by $E_e'$ the projection of $\Pi'$ on $\widetilde{\mathcal{Y}'}$. We will only need this component, thus we present the definitions and results of \cite{ghl} in an easier way. 

    Following \cite{fh23}, we will need the following elements of $\widetilde{\mathcal{Y}'}$.
    \begin{definition}
         For $(i,r)\in V$, we define 
         
        \begin{equation}
             \Sigma_{i,q^r}:=\sum_{k\geq 0}\prod_{0<j\leq k}A_{i,q^{r-2j+2}}^{-1}=1+A_{i,q^r}^{-1}(1+A_{i,q^{r-2}}^{-1}(1+...)) \in \widetilde{\mathcal{Y}'}.
         \end{equation}
     \end{definition}
    In \cite[Proposition 3.17]{fh23} (see also \cite{ghl}), for every $i\in I$, some operators $\Theta_{s_i}'$ of $\Pi'$ are defined. They generate an action of the braid group associated with $W$ on $\Pi'$. For every $i\in I$, the projection of the operator $\Theta_{s_i}'$ on $\widetilde{\mathcal{Y}'}$ is the following:
    \[(E_e'\circ\Theta_{s_i}')([\lambda]):=[\lambda],\quad (E_e'\circ\Theta_{s_i}')(\Psi_{j,q^r}):=\begin{cases}
            \Psi_{j,q^r},\ &\mathrm{if}\ i\neq j\\
            \widetilde{\Psi}_{i,q^{r-2}}\Sigma_{i,q^{r-2}},\ &\mathrm{if}\ i=j.
        \end{cases}\]
 
     This allows us do define some elements of $\widetilde{\mathcal{Y}'}$, which will play a key role. They were introduced in \cite{fh23} in terms of a Weyl group action, but we present a realization from \cite{ghl}.
         \begin{definition}\label{q variabli proiezioni}\cite[Lemma 6.5]{ghl}
             For every $w\in W$, $(i,r)\in V$ let
             \[Q_{w(\omega_i),q^r}:=(E_e'\circ\Theta_w')(\Psi_{i,q^r})\] where $\Theta'_w$ denotes the composition $\Theta_{s_{i_1}}'\cdots \Theta_{s_{i_k}}'$, if $w=s_{i_1}\cdots s_{i_k}$. The notation $Q_{w(\omega_i),q^r}$ is well defined, i.e. it depends only on the weight $w(\omega_i)$ and on $q^r$, as explained in \cite[Definition 6.3]{ghl}.
        \end{definition}

    We refer to these elements as $Q$-variables.\\[10pt]
    We denote by $K_{\ZZ}$ the $[\underline{P}]$-subalgebra of $\widetilde{\mathcal{Y}'}$ topologically generated by the elements
    \[Q_{w(\omega_i),q^r},\qquad\text{for all }(i,r)\in V,\ w\in W.\] In particular, $K_{\ZZ}$ contains all the elements of the form $\Psi_{i,q^r}$, $(i,r)\in V$.
    \begin{proposition}[{\cite[Proposition 7.1]{ghl}}]
        For every $w\in W$ and $(i,r)\in V$, the element $Q_{w(\omega_i),q^r}$ is invertible in $K_{\ZZ}$.
    \end{proposition}

    \begin{definition}\label{def truncation}
        Let $\overline{\mathcal{M}'}$ be the multiplicative group of monomials in the variables $Q_{w(\omega_i),q^r}$, for all $(i,r)\in V$, $w\in W$, and their inverses. In particular, for each $w\in W$, $E_e'\circ \Theta'_w$ is an automorphism of $\overline{\mathcal{M}'}$.\\ Moreover, let $\mathcal{M}'$ be the multiplicative group of Laurent monomials in $\mathcal{Y}'$. Then we extend the definition of the truncation morphism from \cite{fh22} to $\overline{\mathcal{M}'}$. Let 
        \[\Xi:\overline{\mathcal{M}'}\to \mathcal{M}'\] be the truncation homomorphism which assigns to each $P\in \overline{\mathcal{M}'}$ its leading term. (In \cite{fh22} the truncation is denoted by $\Lambda$, but here we use $\Xi$ to avoid confusion with the quantization matrices).
    \end{definition}

    We have the following
    \begin{remark}\label{q var hanno hw}
         Each variable $Q_{w(\omega_i),q^r}$ has a unique factorization of the form 
         \[Q_{w(\omega_i),q^r}=\Psi_{w(\omega_i),q^r}\Sigma_{w(\omega_i),q^r},\] where $\Psi_{w(\omega_i),q^r}$ is a Laurent monomial in the variables $\Psi_{j,q^s}$ and $\Sigma_{w(\omega_i),q^r}$ is a formal power series in the variables $A_{j,b}^{-1}$ with constant term equal to $1$. By \cite[Proposition 3.6]{fh23}, $\Psi_{w(\omega_i),q^r}$ is the leading term in $Q_{w(\omega_i),q^r}$, i.e. 
         \[\Xi(Q_{w(\omega_i),q^r})=\Psi_{w(\omega_i),q^r}.\] 
     \end{remark}
    The following lemma is about a fact that is hidden in the lines of \cite{ghl}, but since we need it, we present and prove it. 
    \begin{lemma}\label{T gives leading term}
        Let $\sigma:\mathcal{Y}'\to\mathcal{Y}'$ be the automorphism defined by:
        \[\sigma(\Psi_{i,q^r})=\Psi_{i,q^{-r}}^{-1},\quad \sigma([\omega])=[\omega],\quad (i,r)\in V,\ \omega\in P_{\QQ}.\] Then, for every $w\in W$, we have
        \begin{equation}
            \Xi (Q_{w(\omega_i),q^r})=\sigma\ \circ\ T_w\ \circ \sigma(\Psi_{i,q^r}).
        \end{equation}  
    \end{lemma}
    \begin{proof}
        We only prove the lemma for $w=s_i$. It is a direct computation.\\
        On the left hand side we have:
        \[\Xi(Q_{s_i(\omega_i),q^r})=\Xi(\Theta'_{s_i}(\Psi_{i,q^r}))=\Xi\left(\Psi_{i,q^{r-2}}^{-1}\prod_{j:c_{ij}=-1}\Psi_{j,q^{r-1}}\Sigma_{i,q^{r-2}}\right)=\Psi_{i,q^{r-2}}^{-1}\prod_{j:c_{ij}=-1}\Psi_{j,q^{r-1}}.\] On the right hand side we find:
        \[\sigma\ \circ\ T_i\ \circ \sigma(\Psi_{i,q^r})=\sigma\circ T_i(\Psi_{i,q^{-r}}^{-1})=\sigma\left(\Psi_{i,q^{-r+2}}\prod_{j:c_{ij}=-1}\Psi^{-1}_{j,q^{-r+1}}\right)=\Psi^{-1}_{i,q^{r-2}}\prod_{j:c_{ij}=-1}\Psi_{j,q^{r-1}}\ .\]
    \end{proof}
    Lemma \ref{T gives leading term} can be rephrased saying that the braid group action on $\overline{\mathcal{M}'}$ gives the leading term of the $Q$-variable.
    
    \begin{remark}\label{confronto tra operatori chari}
        Let us call $\mathcal{M}$ the multiplicative group of Laurent monomials in $\ZZ[\Psi_{i,r}^{\pm 1}\mid (i,r)\in V]$. It is a subgroup of $\mathcal{M}'$.
        There is a group isomorphism $\eta:\mathcal{M}\to  \ZZ^{(V)}$ defined by $\Psi_{i,r}\mapsto \boldsymbol{e}_{(i,r)}$. Note that $\Psi_{w(\omega_i),q^r}$ is in $\mathcal{M}$ for every $(i,r)\in V$.\\ Hence, the braid group actions in Definition \ref{azione braid su base standard} and in \eqref{chari braid T} are related as follows:
        \[\eta\ \circ\ \sigma\ \circ T_w\ \circ \sigma (\Psi_{i,q^r})=\Theta_w\ (\boldsymbol{e}_{(i,r)}).\]
    \end{remark}

    In \cite[Section 7.2]{ghl} normalized $Q$-variables are introduced. They are denoted 
    \[\underline{Q}_{w(\omega_i),q^r}\] for every $w\in W$ and $(i,r)\in V$ and they are elements of $K_{\ZZ}$. 
    We will refer to these elements as $\underline{Q}$-variables. The following relations satisfied by the $\underline{Q}$-variables are known as $QQ$-systems and they play a fundamental role in the proof of Theorem \ref{teo 7.4 ghl}. 
    \begin{theorem}[{\cite[Proposition 7.2]{ghl}}]\label{qq system}
    For $(i,r)\in V$ and $w\in W$ such that $ws_i>w$, we have:
    \begin{equation}
        \underline{Q}_{ws_i(\omega_i),q^r}\underline{Q}_{w(\omega_i),q^{r-2}}-\underline{Q}_{ws_i(\omega_i),q^{r-2}}\underline{Q}_{w(\omega_i),q^r}=\prod_{j:c_{ij}=-1}\underline{Q}_{w(\omega_j),q^{r-1}}.
    \end{equation}
    \end{theorem}
    \begin{example}\label{q variabili A1}
        In type $A_1$, i.e. $\g=\sl_2$, we have $I=\{1\}$, so we can drop the index $1$; the Weyl group consists of two elements, namely the identity $e$ and the simple reflection $s$. We have
        \[Q_{\omega,q^r}=\Psi_{q^r},\qquad Q_{s(\omega),q^r}=\Theta'_s(\Psi_{q^r})=\Psi_{q^{r-2}}^{-1}\Sigma_{q^{r-2}}.\] The normalized variables are:
        \[\underline{Q}_{\omega,q^r}=\Big[-\frac{r}{2}\omega\Big]\Psi_{q^r},\qquad \underline{Q}_{s(\omega),q^r}=\Big[\frac{r-2}{2}\omega\Big]\Psi_{q^{r-2}}^{-1}\Sigma_{q^{r-2}}.\] The $QQ$-system reads:
        \begin{align*}
    \underline{Q}_{s(\omega),q^r}\underline{Q}_{\omega,q^{r-2}}-\underline{Q}_{s(\omega),q^{r-2}}\underline{Q}_{\omega,q^r} &= \Big[-\frac{r-2}{2}\omega\Big]\Psi_{q^{r-2}}\Big[\frac{r-2}{2}\omega\Big]\Psi_{q^{r-2}}^{-1}\Sigma_{q^{r-2}}-\\&\quad -\Big[-\frac{r}{2}\omega\Big]\Psi_{q^r}\Big[\frac{r-4}{2}\omega\Big]\Psi_{q^{r-4}}\Sigma_{q^{r-4}}\\
    &=\Sigma_{q^{r-2}}-[-\alpha]\Psi_{q^r}\Psi_{q^{r-4}}^{-1}\Sigma_{q^{r-4}}\\
    &=1+A_{q^{r-2}}^{-1}\Sigma_{q^{r-4}}-A_{q^{r-2}}^{-1}\Sigma_{q^{r-4}}=1,
        \end{align*}
        where we have used the fact that the simple root $\alpha$ is equal to $2\omega$, $A_{q^{r-2}}=[\alpha]\Psi_{q^{r-4}}\Psi_{q^r}^{-1}$ and that $\Sigma_{a}=1+A_{a}^{-1}\Sigma_{aq^{-2}}$.
    \end{example}

    \subsubsection{Cluster structure on $K_0(\mathcal{O}^{\sh}_{\ZZ})$}
    Consider the quiver $\Gamma_c$ that serves as initial seed for the cluster algebra $\mathcal{A}_{w_0}$. Then the initial cluster variables $x_{(i,r)}$, $(i,r)\in V$ are parameterized by their stabilized $g$-vectors and by Proposition \ref{g vector braid action} and Remark \ref{remark g vector braid action}, there exist $s\in \ZZ$ and $t\in\{0,..,N\}$ such that these $g$-vectors can be written as
    \[g^{(\infty)}_{(i,r)}=\Theta_{i_1}\cdots\Theta_{i_t}\left(\boldsymbol{e}_{(i,m_i)}\right)[s].\] Let $L$ be the ring of Laurent polynomials in the initial cluster variables of $\mathcal{A}_{w_0}$; by Laurent phenomenon, $\mathcal{A}_{w_0}\hookrightarrow L$.
    \begin{theorem}[{\cite[Theorem 7.4]{ghl}}]\label{teo 7.4 ghl}
        For $(i,r)\in V$, let $x_{(i,r)}$ be the initial cluster variable of $\mathcal{A}_{w_0}$ with stabilized $g$-vector 
        \[g^{(\infty)}_{(i,r)}=\Theta_{i_1}\cdots\Theta_{i_t}\left(\boldsymbol{e}_{(i,m_i)}\right)[s]\]
    (as in Remark \ref{remark g vector braid action}). Then the assignment
        \begin{equation}
          x_{(i,r)}  \ \mapsto\ \underline{Q}_{s_{i_1}\cdots s_{i_t}(\omega_i),q^{m_i+2s}}
        \end{equation}
        extends to an injective ring homomorphism $F:L\to K_{\ZZ}$ and the topological closure of $[P_{\QQ}]\otimes_{\ZZ}F(\mathcal{A}_{w_0})$ is equal to $K_{\ZZ}$.
    \end{theorem}
    It is explained in \cite[Remark 7.6]{ghl} that the homomorphism $F$ is well-defined, i.e. it does not depend on the choice of the Coxeter element, or equivalently of the initial seed of $\mathcal{A}_{w_0}$.

    In Section 3 we have recalled the definition of the category $\mathcal{O}^{\sh}$. Now we consider a subcategory of it.
    \begin{definition}\cite[Definition 9.13]{ghl}
        Let $\mathcal{O}^{\sh}_{\ZZ}$ be the full subcategory of $\mathcal{O}^{\sh}$ whose simple constituents are of the form $L[\Psi]$, where $\Psi$ is a Laurent monomial in the variables $[\lambda]$, $\lambda\in P_{\QQ}$, and $\Psi_{i,q^r}$, $(i,r)\in V$.
    \end{definition}
    Thus, the Grothendieck ring $ K_0(\mathcal{O}^{\sh}_{\ZZ})$ is the $[P_{\QQ}]$-topological subspace of $ K_0(\mathcal{O}^{\sh})$ generated by the simple classes $L[\Psi]$ with $\Psi\in \ZZ[\Psi_{i,q^r}^{\pm}]$. Recall the definition of $\mathcal{E}_{\ell,\ZZ}$. Then we get a result analogous to Proposition \ref{iso q car}:
    \begin{proposition}[{\cite[Proposition 9.14]{ghl}}]\label{q car iniettivo in o shift}
        $ K_0(\mathcal{O}^{\sh}_{\ZZ})$ is a subring of $ K_0(\mathcal{O}^{\sh})$ and the homomorphism $\chi_q$ restricts to an isomorphism
        \[\chi_q: K_0(\mathcal{O}^{\sh}_{\ZZ})\xrightarrow{\simeq}\mathcal{E}_{\ell,\ZZ}.\]
    \end{proposition}
    All the generators of $K_{\ZZ}$ are contained in $\mathcal{E}_{\ell,\ZZ}$, hence we can identify $K_{\ZZ}$ with $\mathcal{E}_{\ell,\ZZ}$. So, using this, Proposition \ref{q car iniettivo in o shift} and Theorem \ref{teo 7.4 ghl}, we have the following result.
    \begin{theorem}[{\cite[Theorem 9.15]{ghl}}]\label{iso ghl cluster}
        There is an injective ring homomorphism
        \[I:\mathcal{A}_{w_0}\to K_0(\mathcal{O}^{\sh}_{\ZZ})\] and the topological closure of $[P_{\QQ}]\otimes_{\ZZ}I(\mathcal{A}_{w_0})$ is the entire topological ring $K_0(\mathcal{O}^{\sh}_{\ZZ})$.
    \end{theorem}
    \begin{remark}\label{immagine della q var}
        Note that 
        \[\chi_q(L_{i,q^r}^+)=\Psi_{i,q^r}=Q_{\omega
        _i,q^r},\] and by definition 
        \[\underline{Q}_{\omega_i,q^r}=\left[-\frac{r}{2}\omega_i\right]Q_{\omega_i,q^r}.\]
        Thus in Theorem \ref{iso ghl cluster}, the variable $\underline{Q}_{\omega_i,q^r}$ of $\mathcal{A}_{w_0}$, is sent to the class $\left[-\frac{r}{2}\omega_i\right][L_{i,q^r}^+]$ in $K_0(\mathcal{O}^{\sh}_{\ZZ})$.
    \end{remark}
    For completeness, we point out the following
    \begin{remark}
       Let $\mathcal{O}^{\sh,f}_{\ZZ}$ be the full subcategory of $\mathcal{O}^{\sh}_{\ZZ}$ whose objects are finite-length modules. In \cite[Corollary 6.4]{hz25}, Hernandez and Zhang prove that $ K_0(\mathcal{O}^{\sh,f}_{\ZZ})$ is a subring of $ K_0(\mathcal{O}^{\sh}_{\ZZ})$. Moreover, \cite[Conjecture 9.16]{ghl} states that the images through $I$ of all cluster variables in $\mathcal{A}_{w_0}$ are simple representations in $\mathcal{O}^{\sh}_{\ZZ}$. Thus, in \cite[Conjecture 6.5]{hz25} the authors conjecture that $I([\underline{P}]\otimes_{\ZZ}\mathcal{A}_{w_0})$ is in fact the entire $ K_0(\mathcal{O}^{\sh,f})$. At the moment the Conjecture is proven only for $\g=\sl_2$.
    \end{remark}

    \begin{notation}\label{notation v}
        By Theorem \ref{teo 7.4 ghl}, we can assign to each cluster variable in $\mathcal{A}_{w_0}$ in the initial seed $\Sigma_c$ a unique $\underline{Q}$-variable in $K_{\ZZ}$ that depends on the stabilized $g$-vector of such cluster variable. In order to make the notation more readable, we will denote $\underline{Q}_v$ the $\underline{Q}$-variable associated with the initial cluster variable at vertex $v\in V$ and $g^{(\infty)}_v$ its stabilized $g$-vector. We denote $Q_v$ the corresponding non-normalized $Q$-variable. More in general we will denote an element of the set $V$ by $v$ or by $(i,r)$, depending on the need.
    \end{notation}
    In the next lemma we clarify how the embedding $\mathcal{A}_e\hookrightarrow\mathcal{A}_{w_0}$ works.
    Thanks to Theorem \ref{teo 7.4 ghl}, we consider the set $\{\underline{Q}_v\}_{v\in V}$ as initial cluster variables for $\mathcal{A}_{w_0}$.
    \begin{lemma}\label{zzz}
        Let $\{z_{(i,r)}\}_{(i,r)\in V}$ be the initial cluster of $\mathcal{A}_e$ for the initial seed $\Sigma_e$. The assignment
        \begin{equation}
            z_{(i,r)}\mapsto \underline{Q}_{\omega_i,q^r} 
        \end{equation} 
        extends uniquely to an embedding of cluster algebras \[i:\mathcal{A}_e\hookrightarrow\mathcal{A}_{w_0}.\]
    \end{lemma}
    \begin{proof} 
        We first verify that the $\underline{Q}_{\omega_i,q^r}$ are algebraically independent in $\mathcal{A}_{w_0}$. Indeed, each finite subfamily of these cluster variables is contained in a cluster seed of $\mathcal{A}_{w_0}$ (one of the initial seeds 
        considered in \cite{ghl}, that is with an initial quiver of the form $\Gamma_c^{(m)}$ as in \eqref{successione di quiver}). As the cluster variables in a given seed are algebraically independent, this implies this first statement.
        
        In particular, the subalgebra of $\mathcal{A}_e$ generated by the $z_{(i,r)}$ embeds in $\mathcal{A}_{w_0}$ with the assignment given in the statement. Let us denote by $i$ this embedding. It extends uniquely to a morphism $\tilde{i}$ from the fraction field of $\mathcal{A}_e$ to the fraction field of $\mathcal{A}_{w_0}$. Hence, by the algebraic independence of the variables $\underline{Q}_{\omega_i,q^r}$, we have established the uniqueness of $i$ on $\mathcal{A}_e$.
        
        Now, to prove the existence, we have to show that the image of $\mathcal{A}_e$ by $\tilde{i}$ lies inside $\mathcal{A}_{w_0}$. Consider a cluster variable $\chi$ of $\mathcal{A}_e$. It is obtained from our standard initial seed by a finite number of mutations, which involve only a finite number of vertices. Hence, as in the first paragraph of the current proof, we may consider an initial seed of $\mathcal{A}_{w_0}$ which contains all the corresponding initial cluster variables of the form $\underline{Q}_{\omega_i,q^r}$. Moreover, the arrows between the corresponding vertices are the same for the initial seed of $\mathcal{A}_e$ and for the initial seed of $\mathcal{A}_{w_0}$. Consequently, the exchange relations for our finite sequence of mutations are the same in $\mathcal{A}_e$ as in $\mathcal{A}_{w_0}$. This implies that $\tilde{i}(\chi)$ is a cluster variable of $\mathcal{A}_{w_0}$. Hence the result follows. 
    \end{proof}

    An illustration is given in the Figure \ref{fig:xxx}.
    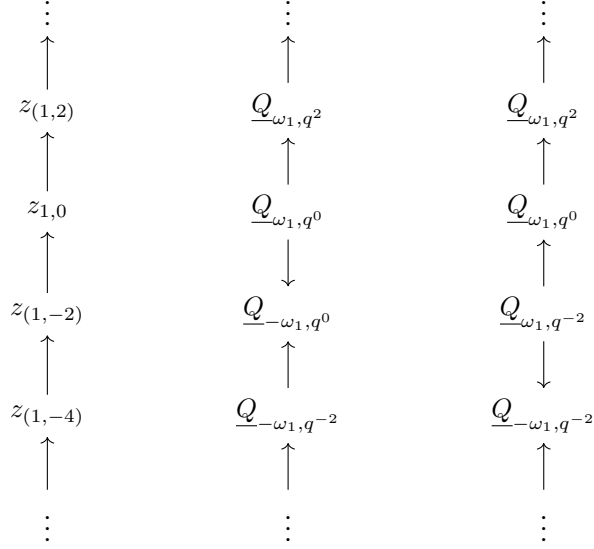
\begin{figure}[H]
        \centering
        
    \[\begin{tikzcd}
    	\vdots && \vdots && \vdots \\
    	{z_{(1,2)}} && {\underline{Q}_{\omega_1,q^2}} && {\underline{Q}_{\omega_1,q^2}} \\
    	{z_{1,0}} && {\underline{Q}_{\omega_1,q^0}} && {\underline{Q}_{\omega_1,q^0}} \\
    	{z_{(1,-2)}} && {\underline{Q}_{-\omega_1,q^0}} && {\underline{Q}_{\omega_1,q^{-2}}} \\
    	{z_{(1,-4)}} && {\underline{Q}_{-\omega_1,q^{-2}}} && {\underline{Q}_{-\omega_1,q^{-2}}} \\
    	\vdots && \vdots && \vdots
    	\arrow[from=2-1, to=1-1]
    	\arrow[from=2-3, to=1-3]
    	\arrow[from=2-5, to=1-5]
    	\arrow[from=3-1, to=2-1]
    	\arrow[from=3-3, to=2-3]
    	\arrow[from=3-3, to=4-3]
    	\arrow[from=3-5, to=2-5]
    	\arrow[from=4-1, to=3-1]
    	\arrow[from=4-5, to=3-5]
    	\arrow[from=4-5, to=5-5]
    	\arrow[from=5-1, to=4-1]
    	\arrow[from=5-3, to=4-3]
    	\arrow[from=6-1, to=5-1]
    	\arrow[from=6-3, to=5-3]
    	\arrow[from=6-5, to=5-5]
    \end{tikzcd}\]
        \caption{From left to right: the initial seed $\Gamma_e$ for $\mathcal{A}_e$, the seed with quiver $\Gamma_c$ for $\mathcal{A}_{w_0}$, the seed with quiver $\Gamma_c^{(1)}$ for $\mathcal{A}_{w_0}$. }
        \label{fig:xxx}
    \end{figure}
    We conclude the section with some considerations about \cite[Section 9.2]{ghl} that are not explicit in \cite{ghl}. Let $\mathcal{C}^{\sh}_{\ZZ}$ be the subcategory of $\mathcal{O}^{\sh}_{\ZZ}$ of finite-dimensional representations.
    
    Then we can consider the Grothendieck ring $ K_0(\mathcal{C}^{\sh}_{\ZZ})$ as a subring of $ K_0(\mathcal{O}^{\sh}_{\ZZ})$.

    The following result is a consequence of \cite[Theorem 8.1]{sqaahernandez}.
    \begin{theorem}[{\cite[Corollary 8.6]{sqaahernandez}}]\label{cor her}
        There is a unique ring isomorphism
        \[K_0(\mathcal{O}^{\mathfrak{b},+}_{\ZZ})\simeq K_0(\mathcal{C}^{\sh}_{\ZZ})\hat{\otimes}_{\ZZ}\mathcal{E}\]
    such that for every $(i,r)\in V$:    
    \begin{equation}\label{lll}
        [L_{i,q^r}^{\mathfrak{b},+}]\mapsto  [L_{i,q^r}^+].
    \end{equation}
    \end{theorem}

    Consequently, we obtain an injective ring morphism
    \begin{equation}\label{aaa}
        K_0(\mathcal{O}^{\mathfrak{b},+}_{\ZZ})\hookrightarrow K_0(\mathcal{O}^{\sh}_{\ZZ}).
    \end{equation}
    We conclude by putting together all the main results of this Section in one theorem, which was not explicitly stated in previous works. 
    \begin{theorem}\label{teo diag comm}
        The following diagram of morphism of algebras is commutative:
        \begin{figure}[H]
        \centering
        \[\begin{tikzcd}
    	\mathcal{A}_e\hat{\otimes}_{\ZZ}\mathcal{E} &&& \mathcal{A}_{w_0}\hat{\otimes}_{\ZZ}\mathcal{E}\\
    	\\
    	K_0(\mathcal{O}^{\mathfrak{b},+}_{\ZZ}) &&&  K_0(\mathcal{O}^{\sh}_{\ZZ})
    	\arrow[hook, from=1-1, to=1-4]
    	\arrow["\simeq"', from=1-1, to=3-1]
    	\arrow["\simeq", from=1-4, to=3-4]
    	\arrow[hook, from=3-1, to=3-4]
    \end{tikzcd}\]
    \label{fig:diagramma clasico}
    \end{figure}
    The top horizontal arrow is due to Lemma \ref{zzz}; the vertical left arrow comes from Theorem \ref{teorema cluster hl16}; the vertical right one comes from Theorem \ref{iso ghl cluster} and Remark \ref{immagine della q var}, and the bottom horizontal one from inclusion \eqref{aaa}.
    \end{theorem}
    
    \begin{proof}
       We establish the commutativity of the diagram. Thanks to Laurent phenomenon for cluster algebras, it is sufficient to prove the commutativity for the initial cluster variables of the seed $\Sigma_e$ of $\mathcal{A}_e$. Following the top and vertical-right arrows, for all $(i,r)\in V$ we have:
       \[z_{(i,r)}\mapsto \underline{Q}_{\omega_i,q^r}\mapsto \Big[-\frac{r}{2}\omega_i\Big][L^+_{i,q^r}],\] where the first assignment is a normalized version of the one in Lemma \ref{zzz} and the second is explained in Remark \ref{immagine della q var}. Following the vertical left and bottom arrows, we have:
       \[z_{(i,r)}\mapsto \Big[-\frac{r}{2}\omega_i\Big][L_{i,q^r}^{\mathfrak{b},+}]\mapsto \Big[-\frac{r}{2}\omega_i\Big][L_{i,q^r}^+],\] where the first assignment is the definition of the morphism in Theorem \ref{teorema cluster hl16} and the second is the same as \eqref{lll}. 
    \end{proof}

    \begin{remark} The category $\mathcal{O}^{\mathfrak{b},+}_{\ZZ}$ contains finite-dimensional representations of the ordinary quantum affine algebras, for which cluster structures have also been intensively studied. We refer to \cite{hl16o} for the compatibility between the cluster structure of these two categories.
    \end{remark}
    \begin{remark} 
    It is established in \cite{sqaahernandez} that $K_0(\mathcal{C}_{\mathbb{Z}}^{\sh})$ is in fact isomorphic to $\mathcal{A}_e$. Consequently, the construction recalled in Section \ref{qgrb} also provides a quantization of $K_0(\mathcal{C}_{\mathbb{Z}}^{\sh})$.
    \end{remark}

    \section{The quantum Grothendieck ring for \texorpdfstring{$\mathcal{O}^{\sh}_{\ZZ}$}{Osh}}\label{the quantum grothendieck ring}

    \newcommand{\rosso}[1]{\textcolor{red}{#1}}
    \newcommand{\verde}[1]{\textcolor{green}{#1}}
    This section is the core of the paper. We construct a quantum deformation of $K_0(\mathcal{O}^{\sh}_{\ZZ})$. It will be defined as a quantum cluster algebra, with quiver for an initial seed of the form $\Gamma_c$. Thus, the first objective is to define a compatible based quantum torus. This will allow us to define the quantum cluster algebra $\mathcal{A}_{t,w_0}$. Moreover, we want this construction to be coherent with the existing quantum Grothendieck ring $K_t(\mathcal{O}^{\mathfrak{b},+}_{\ZZ})$. In other words, we want to obtain a $t$-deformed version of the diagram in Figure \ref{fig:diagramma clasico}, that is:
    \[\begin{tikzcd}
    	\mathcal{A}_{t,e}\hat{\otimes}_{\ZZ}\mathcal{E} &&& \mathcal{A}_{t,w_0}\hat{\otimes}_{\ZZ}\mathcal{E} \\
    	&&& {\text{}} \\
    	K_t(\mathcal{O}^{\mathfrak{b},+}_{\ZZ}) &&& K_t(\mathcal{O}^{\sh}_{\ZZ})
    	\arrow[hook, from=1-1, to=1-4]
    	\arrow["\simeq"', from=1-1, to=3-1]
    	\arrow["{\text{}}", from=1-4, to=3-4]
    	\arrow[hook, from=3-1, to=3-4]
    \end{tikzcd}\]
    where the left vertical arrow is due to Definition \ref{def quantum gr ring lea}, while the other arrows will appear in our construction.

    \subsection{Mutation of infinite matrices}
    Recall that we have fixed the quiver $\Gamma_c$ as initial quiver for the cluster algebra $\mathcal{A}_{w_0}$ and its vertex set is $V$.
    In order to quantize $\mathcal{A}_{w_0}$ we have to define a compatible quantization matrix, namely a skew-symmetric integer matrix $\Lambda_c$ such that, if $B_c$ denotes the exchange matrix for $\Gamma_c$,
    \[B_c^T\Lambda_c=\mathrm{diag}\left((d_i)_{i\in V}\right)\] where the $(d_i)_{i\in V}$ are non-zero integers, all with the same sign (see Definition \ref{def compatibile}). Then, for $t$ a formal variable, the based quantum torus will be defined as the $\ZZ[t^{\pm\frac{1}{2}}]$-algebra with a distinguished set of generators made of the initial cluster variables $(x_{(i,r)})_{(i,r)\in V}$ and product $\ast$ that verifies the following $t$-commutation relations:
    \[x_{(i,r)}\ast x_{(j,s)}=t^{(\Lambda_c)_{(i,r),(j,s)}}x_{(j,s)}\ast x_{(i,r)},\qquad\text{for all }(i,r),(j,s)\in V.\]  
    Recall that $\mathcal{A}_e$ and $\mathcal{A}_{w_0}$ are different cluster algebras (see Remark \ref{remark sullamutazione infinita}), although they are related by an infinite sequence of green mutations. Thus, we cannot deduce $\Lambda_c$ from $\Lambda_e$ by the usual process of (quantization) matrix mutation. A crucial technical point of our definition will be the use of stabilized $g$-vectors.

    Before giving the definition of $\Lambda_c$, we generalize to the case of infinite quivers the following lemma, which is proved in the context of cluster categories by Palu in \cite[Theorem 12]{palu}. For completeness, we give a proof in terms of cluster algebras.
    
    \begin{lemma}\label{formula g matrix}
     Let $B$ be a $n\times n$ skew-symmetric matrix and $\mathcal{A}$ the cluster algebra with initial seed $(B,\boldsymbol{x})$. Let $B'$ be the matrix obtained after a  finite sequence of  mutations. Moreover, let $G$ be the $G$-matrix  for the cluster variables in seed $(B', \boldsymbol{x'})$ with respect to the reference seed $(B,\boldsymbol{x})$. Then, the following formula holds:
        \begin{equation}
            B=GB'G^T.
        \end{equation}
    \end{lemma}

    \begin{proof}
    Recall that $g$-vectors are defined for principal coefficients cluster algebras (see Section 10 in \cite{fz4}). Indeed, we can denote $\tilde{B}$ the principal extension of $B$, that is the $2n\times n$ matrix with top part given by $B$ and bottom part given by the identity matrix of size $n$. Applying the same sequence of mutations in the hypothesis to $\tilde{B}$, we obtain $\tilde{B}'$. The bottom $n\times n$ submatrix of $\tilde{B}'$ is denoted $C'$. 
    By \cite[Theorem 4.1]{nakanishi} we have $G^T=(C')^{-1}$.
    On the other hand, \cite[Equation 6.14]{fz4} states that 
    \begin{equation}
        B'G^T=(C')^TB.
    \end{equation}
    So, multiplying both sides on the left by $\left((C')^{T}\right)^{-1}=G$ we obtain the result.
    \end{proof}

    Now we generalize Lemma \ref{formula g matrix} to the case of an infinite rank cluster algebra. To simplify notations, we consider specifically our cluster algebra $\mathcal{A}_{w_0}$. 
    The key to pass from finite to infinite matrices is that every vertex in $\Gamma_c$ has only finitely many incident arrows (let us call these types of quivers ``locally finite''), thus the mutation at each vertex involves only a finite portion of the quiver, or equivalently of the exchange matrix. Thus, the proof can be adapted to all infinite quivers that are locally finite.

    \begin{proposition}\label{formula g-matrix infinita}
      Let $B_c$ be the exchange matrix of the quiver $\Gamma_c$ and $\Sigma_c$ the associated seed. Let $\Gamma_c'$, $B_c'$ and $\Sigma'_c$ be respectively the quiver, the exchange matrix and the seed obtained after a finite number of mutations. Let $G$ be the $G$-matrix for the cluster variables in the seed $\Sigma_c$ with respect to the reference seed $\Sigma_c'$. Then, the following formula holds:
        \begin{equation}
            B'_c=GB_cG^T.
        \end{equation}
    \end{proposition}
    \begin{proof}  
    We fix a vertex of $\Gamma_c$, say $(i,r)\in V$ and we consider $\mu_{(i,r)}$, the quiver mutation at this vertex. The resulting quiver is $\Gamma_c'=\mu_{(i,r)}(\Gamma_c)$, with exchange matrix $B_c'=\mu_{(i,r)}(B_c)$. In $\Gamma_c$ the vertex $(i,r)$ is connected with a finite subset of vertices in $V$. Let us call $W_{(i,r)}$ such subset.\\
    Hence, recalling the mutation rule (see \cite[Definition 4.2]{fz1}), we see that the only entries of $B_c$ that possibly change after mutation $\mu_{(i,r)}$ are those contained in the submatrix $\hat{B}_c=(\hat{b}_{h,k})_{h,k\in W_{(i,r)}}$. In other words, the effects of the mutation can be seen by looking only at the finite full subquiver of $\Gamma_c$ with vertex set $W_{(i,r)}$ that has $\hat{B}_c$ as exchange matrix. We denote $\hat{B}_c'$ the the matrix obtained from $\hat{B}_c$ by mutation at $(i,r)$. Thus, the mutated matrix $B_c'$ is obtained from $B_c$ substituting $\hat{B}_c$ with $\hat{B}_c'$.  Here is the situation graphically:
        \[
    B_c=\begin{pmatrix}
      
        \ast & \ast 
      & 0 \\
      \ast & 
        \hat{B}_c
      & \ast \\
      0 & 
        \ast 
       & 
        \ast
     
    \end{pmatrix},\qquad B_c'=\begin{pmatrix}
      
        \ast & \ast & 0 \\
      \ast & \hat{B}_c' & \ast \\
      0 & \ast & \ast
      
    \end{pmatrix}
    \] where the asterisks denote parts of the matrix that are unchanged, while the zeros are due to the locally finiteness of $\Gamma_c$. Now, we know how $B_c$ and $\hat{B}_c$ are related by mutation, thanks to Lemma~\ref{formula g matrix}. Let $\hat{G}$ be the $G$-matrix for the cluster variables in the seed associated with $\hat{B}_c$ with respect to the reference seed associated with $\hat{B}_c'$. So, Lemma \ref{formula g matrix} tells us $\hat{B}_c'=\hat{G}\hat{B}_c\hat{G}^T.$
    Let us consider the principal coefficients extension of $B_c$ and let $C$ be the corresponding coefficient matrix. Hence, $C=\Id_{V\times V}$. The mutation $\mu_{(i,r)}$ acts on $C$ by modifying only its $W_{(i,r)}\times W_{(i,r)}$ submatrix. 
    If we denote by $\hat{C}'$ the mutated coefficient matrix for the finite full subquiver with vertex set $W_{(i,r)}$, then the entire mutated coefficient matrix $C'$ is a three blocks diagonal matrix of the form
    \[C'=\begin{pmatrix}
      I & 0 & 0 \\
      0 & \hat{C}' & 0 \\
      0 & 0 & I
    \end{pmatrix}.\]
    Thus, the $G$-matrix $G$ for the cluster variables in $\Sigma_c$ with respect to the reference seed $\Sigma_c'$ is a three-blocks diagonal matrix, where the central block is $\hat{G}=((\hat{C}')^T)^{-1}$ and the other two are identity blocks. 
    \[G=\begin{pmatrix}
      I & 0 & 0 \\
      0 & \hat{G} & 0 \\
      0 & 0 & I
    \end{pmatrix}
    \]
    So we can conclude that
    \[B'_c= \begin{pmatrix}
     
        \ast & \ast & 0 \\
      \ast & \hat{G}\hat{B}_c\hat{G}^T & \ast \\
      0 & \ast &  \ast\\
      
    \end{pmatrix}=\begin{pmatrix}
      I & 0 & 0 \\
      0 & \hat{G} & 0 \\
      0 & 0 & I
    \end{pmatrix}\begin{pmatrix}
      
          \ast & \ast & 0 \\
      \ast & \hat{B}_c & \ast \\
      0 & \ast & \ast
     
    \end{pmatrix}\begin{pmatrix}
      I & 0 & 0 \\
      0 & \hat{G}^T & 0 \\
      0 & 0 & I
    \end{pmatrix}=GB_cG^T.\]
    \newline In order to pass to a finite sequence of mutations $\mu:=\mu_{(i_1,r_1)}\cdots \mu_{(i_n,r_n)}$, for all $(i_1,r_1),\dots,(i_n,r_n)\in V$, we can proceed in similar way. Let $W$ be the union of subsets $W_{(i_1,r_1)}\cup\cdots \cup W_{(i_n,r_n)}$. The effect of $\mu$ can be seen on a finite submatrix of size $|W|\times |W|$ extracted from the matrix $B_c$. Then applying again Lemma \ref{formula g matrix}, we get the result.
    \end{proof}

    In the next Theorem we extend the result of Proposition \ref{formula g-matrix infinita} to the case where the reference seed is the limit reference seed $\Sigma_e$, associated to the basic infinite quiver $\Gamma_e$. In this case, the $G$-matrix is well defined and it is the limit $G$-matrix $G^{(\infty)}$ of Theorem \ref{teo stabilized g vectors}. One difficulty is that this time we consider a reference seed obtained after \emph{infinitely many} mutations. 
    \begin{theorem}\label{teo convergenza}
    Let $B_c$ be the exchange matrix for $\Gamma_c$ and $B_e$ the exchange matrix for $\Gamma_e$. Then we have
    \begin{equation}
        B_e=G^{(\infty)}B_c\left(G^{(\infty)}\right)^T.
    \end{equation}
    \end{theorem}

    \begin{proof}
    With the notation from Section \ref{ghl construction}, let us call $B^{(m)}$ the exchange matrix for the quiver $\Gamma_c^{(m)}$ and $G^{(m)}$ the $G$-matrix for the variables in the initial seed $\Sigma_c$ with respect to $\Sigma_c^{(m)}$, i.e. the $(i,k)$-th column of $G^{(m)}$ is the $g$-vector $g^{(m)}_{(i,k)}$. For all $m\in \NN$, we can apply Proposition \ref{formula g-matrix infinita} to say that 
    \begin{equation}\label{prodotto m-mo}
        B^{(m)}=G^{(m)}B_c(G^{(m)})^T.
    \end{equation} Now, in the limit $m\to \infty$, that is the "limit" for infinite sequence of green mutations explained in Remark \ref{remark sullamutazione infinita}, $G^{(m)}\to G^{(\infty)}$ by Theorem \ref{teo stabilized g vectors}, while $B^{(m)}\to B_e$ since $\Gamma_c^{(m)}\to \Gamma_e$ (by the same Remark \ref{remark sullamutazione infinita}).
    Thus, we can consider the limit $m\to \infty$ of the product \eqref{prodotto m-mo} (which is just the product of the limits since they are all convergent) and obtain
    \[B_e=G^{(\infty)}B_c\left(G^{(\infty)}\right)^T.\]
    \end{proof}
    We stress the fact that all the infinite matrices considered so far ($B_e$, $B_c$, $G^{(\infty)}$) have finitely many non-zero elements for each of their columns and rows, hence the matrix products are well defined. 
    \begin{corollary}\label{G invertibile}
        The stabilized $G$-matrix $G^{(\infty)}$ is invertible. As a consequence, we have
        \begin{equation}\label{formula g matrici infinite}
            B_c=\left(G^{(\infty)}\right)^{-1}B_e\left(G^{(\infty)}\right)^{-T},
        \end{equation}where for a matrix $A$, we denote $A^{-T}:=(A^T)^{-1}$. 
    \end{corollary}
    \begin{proof}
        The matrix $G^{(\infty)}$ is of block diagonal form, so it is enough to prove that each block is invertible. By Theorem \ref{teo stabilized g vectors}, each block is given by a product of matrices $\mathbf{T}_i$ and each of these matrices is invertible, by construction. Hence, Equation \eqref{formula g matrici infinite} can be obtained from Theorem \ref{teo convergenza}.
    \end{proof}
    
    \begin{example}\label{esempio A1 1}
        We consider the case $\g$ of type $A_1$. Following the convention from the previous section, the initial quiver $\Gamma_c$ is as in Figure \ref{fig:due quiver per A1}.
       
     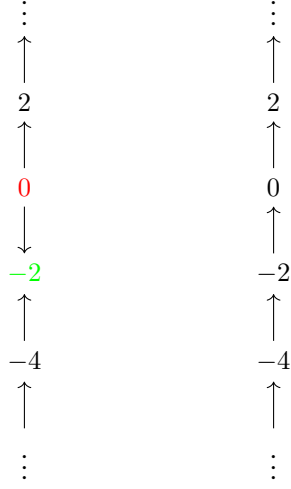
\begin{figure}[h]
         \centering
    
    \[\begin{tikzcd}
    	\vdots &&& \vdots \\
    	2 &&& 2 \\
    	{\color{red}{0}} &&& 0 \\
    	\color{green}{-2} &&& {-2} \\
    	{-4} &&& {-4} \\
    	\vdots &&& \vdots
    	\arrow[from=2-1, to=1-1]
    	\arrow[from=2-4, to=1-4]
    	\arrow[from=3-1, to=2-1]
    	\arrow[from=3-1, to=4-1]
    	\arrow[from=3-4, to=2-4]
    	\arrow[from=4-4, to=3-4]
    	\arrow[from=5-1, to=4-1]
    	\arrow[from=5-4, to=4-4]
    	\arrow[from=6-1, to=5-1]
    	\arrow[from=6-4, to=5-4]
    \end{tikzcd}\]
         \caption{$\Gamma_{c}$ on the left and $\Gamma_e$ on the right.}
         \label{fig:due quiver per A1}
     \end{figure}

        Thus, the exchange matrices $B_e$ and $B_c$ are the following:
        \begin{equation}
            B_e=\begin{pmatrix}
            \ddots & \ddots & \ddots \\
                   &-1 & 0 & 1  \\
                          &&-1 & 0 & 1 \\
                               &&&-1 & 0 & 1\\
                               &&&&-1 & 0 & 1  \\
                              
                                &&&&&\ddots & \ddots & \ddots
        \end{pmatrix}\quad B_c= 
        \begin{pmatrix}
            \ddots & \ddots & \ddots \\
                   
                          &-1 & 0 & 1 \\
                               &&\color{green}-1 & \textcolor{green}{0} & \color{green}-1\\
                               &&& \color{red}1 & \textcolor{red}{0} & \color{red}1\\
                               &&&& -1 & 0 & 1 \\
                              
                                     &&&&&\ddots & \ddots & \ddots
                              
        \end{pmatrix}
        \end{equation} Namely, $B_e$ is skew symmetric, with the super-diagonal made of $1$'s, the sub-diagonal made of $-1$'s and all other entries equal to $0$; $B_c$ is skew-symmetric and it is equal to $B_e$ except for the $-2$-th row (in green) and the $0$-th row (in red). On the other hand, for each $k\in 2\ZZ$, the stabilized $g$-vector are:
        \begin{alignat*}{2}
            g^{(\infty)}_k&=\mathbf{e}_k&\quad\text{if }k\geq 0;\\
            g^{(\infty)}_k&=-\mathbf{e}_k&\quad\text{if }k<0.
        \end{alignat*}
        Hence the $G$-matrix looks like
        \begin{equation}\label{g matrix sl2}
            G^{(\infty)}=\begin{pmatrix}
                \ddots & \\
                       & -1 & \\
                            && -1 & \\
                            &&& -1\\
                                  &&&& \color{red} 1 & \\
                                                 &&&&& 1 & \\
                                                 &&&&&& 1 & \\
                                                      &&&&&&& \ddots
                    
            \end{pmatrix}
        \end{equation}
        where we marked in red the entry $G^{(\infty)}_{0,0}$. In this case we have that $G^{(\infty)}=\left(G^{(\infty)}\right)^{-1}=\left(G^{(\infty)}\right)^T$, so it is a easy computation to verify that formula \eqref{formula g matrici infinite} holds.
       
    \end{example}
    \subsection{Compatible pair}
    Now we are ready to define:
    \begin{equation}\label{definizione lambda c}
        \Lambda_c:= \left(G^{(\infty)}\right)^T \Lambda_e G^{(\infty)}.
    \end{equation}
    Note that $\Lambda_e$ has infinite non-zero entries in each row and column, but thanks to the form of $G^{\infty}$ the product above is well defined.
    \begin{proposition}\label{quantum cluster ade}
        The pair $\left(\Lambda_c,B_c\right)$ is compatible.
    \end{proposition}
    \begin{proof}
        Thanks to our definition of $\Lambda_c$, to Equation \eqref{formula g matrix} and to the compatibility of $\left(\Lambda_e,B_e\right)$ (see Proposition \ref{compatibilita lea}) the proof of compatibility is straightforward:
        \begin{align*}
            \left(B_c\right)^T\Lambda_c &= \left(\left(G^{(\infty)}\right)^{-1}B_e(G^{(\infty)})^{-T}\right)^{T}\left(G^{(\infty)})^T\Lambda_e(G^{\infty})\right)\\
            &=\left(G^{(\infty)}\right)^{-1}B_e^T\Lambda_e G^{(\infty)}\\
            &=\left(G^{(\infty)}\right)^{-1}(-2\Id_{V\times V})\ G^{(\infty)}\\
            &=-2\Id_{V\times V}.
        \end{align*} 
    \end{proof}

    \subsection{Quantum tori}
    We use freely the Notation \ref{notation v}.
    
    \begin{definition}
        The quantum torus $T_{t,c}$ is the $\ZZ[t^{\pm \frac{1}{2}}]$-algebra generated by the $Q_v^{\pm 1}$, for $v\in V$, with non-commutative product $\ast_c$ and $t$-commutation relations
        \[Q_v\ast_c Q_{v'}=t^{(\Lambda_c)_{v,v'}}Q_{v'}\ast_c Q_v,\quad \forall v,v'\in V.\]
    \end{definition}
    
    \begin{remark} \label{commutative monomial}
        When dealing with quantum tori, like $(T_{t,c},\ast_c)$ or $(T_{t,e},\ast_e)$ (see Definition \ref{quasi toro lea}), it is useful to consider the so-called \emph{commutative} monomials. Note that for every $v,w$ in $V$, the expression
        \[t^{\frac{1}{2}(\Lambda_c)_{w,v}}Q_v\ast_c Q_w=t^{\frac{1}{2}(\Lambda_c)_{v,w}}Q_w\ast_c Q_v\] is invariant under permutation of $v$ and $w$. Thus, we denote 
        \[Q_vQ_w:=t^{\frac{1}{2}(\Lambda_c)_{w,v}}Q_v\ast_cQ_w\] and we call it a \emph{commutative monomial}. More in general, the non commutative product of elements in the quantum torus depends heavily on the order, so to get rid of this technicality, it is useful to consider commutative monomials of several variables. Let us denote by $\overset{\rightarrow}{\ast}$ the ordered product of elements in $T_{t,c}$. Then, for a family of integers $(n_v)_{v\in V}$ with only finitely many non-zero components, we define
        \[\prod_{v\in V'}Q_v^{n_v}:=t^{\frac{1}{2}\sum_{v<u}n_vn_u(\Lambda_c)_{u,v}}\ \overset{\rightarrow}{\ast}_cQ_v^{n_v}.\] Then, we can consider either the commutative or non-commutative product of commutative monomials. Similarly one can define commutative monomials for $T_{t,e}$, see \cite[Equation (5.11)]{b21}. 
    \end{remark}

    \begin{definition}
        We extend the quantum torus $T_{t,c}$ to 
        \[\mathcal{T}_{t,c}:=[\underline{P}]\otimes_{[P]}(\mathcal{E}\hat{\otimes}_{\ZZ}T_{t,c}).\]  Note that as $\mathcal{E}$-algebra, $\mathcal{T}_{t,c}$ has another set of generators made of the initial $\underline{Q}$-variables.
    \end{definition}
    
    \begin{definition}\label{def quantum cluster nostra}
        Let $\mathcal{A}_{t,w_0}$ be the quantum cluster algebra associated with the mutation class of $(\Lambda_c,B_c)$. Then, by the quantum Laurent phenomenon, $\mathcal{A}_{t,w_0}\subset \mathcal{T}_{t,c}$.
    \end{definition}

    \begin{example}\label{esempio A1 2}
        We continue Example \ref{esempio A1 1} by computing the quantization matrix $\Lambda_c$. To do so, first we need $\Lambda_e$ and it can be computed using the function $f$ defined in Example \ref{f}. One can verify that the result is 
         \[\Lambda_e=\begin{pmatrix}
             & \vdots & \vdots & \vdots & \vdots & \vdots &  \\
    \dots & \color{blue}0 & -1 & 0 & -1 & 0 & \dots \\
    \dots & 1 & \color{blue}0 & -1 & 0 & -1 & \dots\\
    \dots & 0 & 1 & \color{blue}0 & -1 & 0 & \dots \\
    \dots & 1 & 0 & 1 & \color{blue}0 & -1 & \dots \\
    \dots & 0 & 1 & 0 & 1 & \color{blue}0 & \dots \\
     & \vdots & \vdots & \vdots & \vdots & \vdots & 
        \end{pmatrix}\]where we paint in blue the main diagonal, in order to highlight the structure of the matrix: the entries in the upper triangular part are alternating $0$'s and $-1$'s (thus in the lower triangular part there are alternating $0$'s and $1$'s). Recalling $G^{(\infty)}$ from \eqref{g matrix sl2} and using the definition in \eqref{definizione lambda c}, we obtain
    
        \begin{equation}\label{quantization matrix sl2} \Lambda_{c}=
      \left(
       \begin{array}{ccccc|cccccc}
       
      &  & & & & \overset{0}{\downarrow} & & & & & \\
       
    0 & -1 & 0 & -1 & 0 & 1 & 0 & 1 & 0 & 1 & 0 \\
    1 & 0 & -1 & 0 & -1 & 0 & 1 & 0 & 1 & 0 & 1 \\
    0 & 1 & 0 & -1 & 0 &  1 & 0 & 1 & 0 & 1 & 0 \\
    1 & 0 & 1 & 0 & -1 & 0 & 1 & 0  & 1 & 0 & 1\\
    0 & 1 & 0 & 1 & 0 & 1 & 0 & 1 & 0 & 1 & 0 \\ \hline
    -1 & 0 & -1 & 0 & -1 & 0 & -1 & 0 & -1 & 0 & -1\\
    0 & -1 & 0 & -1 & 0 & 1 & 0 & -1 & 0 & -1 & 0 \\
    -1 & 0 & -1 & 0 & -1 & 0 & 1 & 0 & -1 & 0 & -1 \\
    0 & -1 & 0 & -1 & 0 & 1 & 0 & 1 & 0 & -1 & 0 \\
    -1 & 0 & -1 & 0 & -1 & 0 & 1 & 0 & 1 & 0 & -1 \\
    0 & -1 & 0 & -1 & 0 & 1 & 0 & 1 & 0 & 1  & 0
     
       \end{array}
      \right)
    \end{equation}
    We have highlighted the $0$-th column. Notice that the matrix can be divided in four blocks, where the top-left and bottom-right ones have the same form as the matrix $\Lambda_e$, while the top-right and the bottom-left ones have opposite signs, compared to the corresponding blocks in $\Lambda_e$. 
    \end{example}

    The definition in \eqref{definizione lambda c} is not explicit enough if one wants to compute $t$-commutations between initial cluster variables in $\mathcal{A}_{t,w_0}$. To solve this problem, in the next theorem we make a key observation about the quantum tori $T_{t,c}$ and $T_{t,e}$. First, we need a lemma about $Q$-variables and stabilized $g$-vectors. 
    
    Recall the truncation map (Definition \ref{def truncation}), Remark \ref{confronto tra operatori chari} and the group isomorphism $\eta:\mathcal{M}\to\ZZ^{(V)}$. Moreover recall the Notation \ref{notation v}.

    \begin{lemma}\label{g vettori corrispondon a hw}
         For every $v\in V$, the image through $\eta$ of the leading term of $Q_v$ is precisely $g^{(\infty)}_{v}$; in other words:
        \[(\eta\ \circ\ \Xi)\left(Q_v\right)=g^{(\infty)}_v.\]
    \end{lemma}
    \begin{proof}
    First one can check that for all $s\in \ZZ$, $w\in W$ and $\boldsymbol{u}\in\ZZ^{(V)}$, $\Theta_w(\boldsymbol{u}[s])=\Theta_w(\boldsymbol{u})[s]$. Then, as in Theorem \ref{teo 7.4 ghl}, suppose $g^{(\infty)}_v=\Theta_{i_1}\cdots \Theta_{i_t}\left(\boldsymbol{e}_{(i,m_i)}\right)[s]$. Let us set $w=s_{i_1}\cdots s_{i_t}$. Thus $Q_v=Q_{w(\omega_i),q^{m_i+2s}}$. Using Lemma \ref{T gives leading term} and Remark \ref{confronto tra operatori chari}, we have:
    \begin{align*}
        g^{(\infty)}_v &= \Theta_w(\boldsymbol{e}_{(i,m_i)})[s]\\
        &=\Theta_w(\boldsymbol{e}_{(i,m_i)}[s])\\
        &=\eta\circ\sigma\circ T_w\circ \sigma(\Psi_{i,q^{m_i+2s}})\\
        &=\eta \circ \Xi(Q_{{w(\omega_i),q^{m_i+2s}}})\\
        &= \eta\circ \Xi(Q_v).
    \end{align*}
    \end{proof}
    \begin{remark}\label{notazione leading term}
         Since the columns of the matrix $G^{(\infty)}$ are the stabilized $g$-vectors $g^{(\infty)}_v$, $v\in V$, thanks to Lemma \ref{g vettori corrispondon a hw}, we will sometimes denote the leading term of $Q_v$ by $\Xi(Q_v)=\prod_{v'}\Psi_{v'}^{G^{(\infty)}_{v',v}}\in \mathcal{M}$.
    \end{remark}
    \begin{example}
         Let us consider type $A_2$. Following \cite[Example 6.4]{ghl}, we have:
         \begin{align*}
             Q_{\omega_1,q^r}&=\Psi_{1,q^r}\  &&\Rightarrow\  (\eta\ \circ\ \Xi)\left(Q_{\omega_1,q^r}\right)&&=\boldsymbol{e}_{(1,r)}\\
             Q_{s_1(\omega_1),q^r}&=\Psi_{1,q^{r-2}}^{-1}\Psi_{2,q^{r-1}}\Sigma_{1,q^{r-2}}&&\Rightarrow\ (\eta\ \circ\ \Xi)\left(Q_{s_1(\omega_1),q^r}\right)&&= -\boldsymbol{e}_{(1,r-2)}+\boldsymbol{e}_{(2,r-1)}\\
             Q_{s_2s_1(\omega_1),q^r}&= \Psi_{2,q^{r-3}}^{-1}\Sigma_{21,q^{r-2}}&&\Rightarrow\ (\eta\ \circ\ \Xi)\left(Q_{s_2s_1(\omega_1),q^r}\right)&&= -\boldsymbol{e}_{(2,r-3)},
             \end{align*}
             where 
             \[\Sigma_{21,q^{r-2}}=\sum_{0\leq l\leq k}\left(\prod_{i=0}^{k-1}A_{2,q^{r-3-2i}}^{-1}\prod_{j=0}^{l-1}A_{1,q^{r-2-2j}}^{-1}\right),\] and where we have used the isomorphism $\eta$ between $\mathcal{M}$ and $\ZZ^{(V)}$. We also have
             \[Q_{s_2(\omega_1),q^r}=Q_{\omega_1,q^r},\quad Q_{s_1s_2(\omega_1),q^r}=Q_{s_1(\omega_1),q^r},\quad Q_{s_2s_1s_2(\omega_1),q^r}=Q_{s_1s_2s_1(\omega_1),q^r}=Q_{s_2s_1(\omega_1),q^r}\] and using the symmetry of the root system of type $A_2$, we can find the expressions for $Q$-variables of type $Q_{w(\omega_2),q^r}$ by switching $1$ and $2$ above. 
    
            \begin{figure}[H]\label{q var come cluster var}
                \centering
             
    \[\begin{tikzcd}
    	\vdots \\
    	{\underline{Q}_{\omega_1,q^2}} & \vdots &&&& \vdots \\
    	& {\underline{Q}_{w_2,q}} &&&& {\underline{Q}_{\omega,q^4}} \\
    	{\underline{Q}_{\omega_1,q^0}} &&&&& {\underline{Q}_{\omega,q^2}} \\
    	{\underline{Q}_{s_1(\omega_1),q^0}} &&&&& {\underline{Q}_{\omega,q^0}} \\
    	& {\underline{Q}_{s_1(\omega_2),q^{-1}}} &&&& {\underline{Q}_{s(\omega),q^0}} \\
    	& {\underline{Q}_{s_1s_2(\omega_2),q^{-1}}} &&&& {\underline{Q}_{s(\omega),q^{-2}}} \\
    	{\underline{Q}_{s_1s_2(\omega_1),q^{-2}}} &&&&& {\underline{Q}_{s(\omega),q^{-4}}} \\
    	{\underline{Q}_{s_1s_2s_1(\omega_1),q^{-2}}} &&&&& {\underline{Q}_{s(\omega),q^{-6}}} \\
    	& {\underline{Q}_{s_1s_2s_1(\omega_2),q^{-3}}} &&&& {\underline{Q}_{s(\omega),q^{-8}}} \\
    	{\underline{Q}_{s_1s_2s_1(\omega_1),q^{-4}}} &&&&& {\underline{Q}_{s(\omega),q^{-10}}} \\
    	& \vdots &&&& \vdots \\
    	\vdots
    	\arrow[shorten >=6pt, from=2-1, to=1-1]
    	\arrow[from=2-1, to=3-2]
    	\arrow[from=3-2, to=2-2]
    	\arrow[from=3-2, to=4-1]
    	\arrow[from=3-6, to=2-6]
    	\arrow[from=4-1, to=2-1]
    	\arrow[from=4-1, to=5-1]
    	\arrow[from=4-6, to=3-6]
    	\arrow[from=5-1, to=6-2]
    	\arrow[from=5-6, to=4-6]
    	\arrow[from=5-6, to=6-6]
    	\arrow[from=6-2, to=3-2]
    	\arrow[from=6-2, to=7-2]
    	\arrow[from=7-2, to=8-1]
    	\arrow[from=7-6, to=6-6]
    	\arrow[from=8-1, to=5-1]
    	\arrow[from=8-1, to=9-1]
    	\arrow[from=8-6, to=7-6]
    	\arrow[from=9-1, to=10-2]
    	\arrow[from=9-6, to=8-6]
    	\arrow[from=10-2, to=7-2]
    	\arrow[from=10-2, to=11-1]
    	\arrow[from=10-6, to=9-6]
    	\arrow[from=11-1, to=9-1]
    	\arrow[from=11-1, to=12-2]
    	\arrow[from=11-6, to=10-6]
    	\arrow[from=12-2, to=10-2]
    	\arrow[from=12-6, to=11-6]
    	\arrow[from=13-1, to=11-1]
    \end{tikzcd}\]
    \caption{The image by $F$ of an initial seed for $\mathcal{A}_{w_0}$ in type $A_2$ on the left and type $A_1$ on the right.}
    \label{fig:seed con q variables}
    \end{figure}
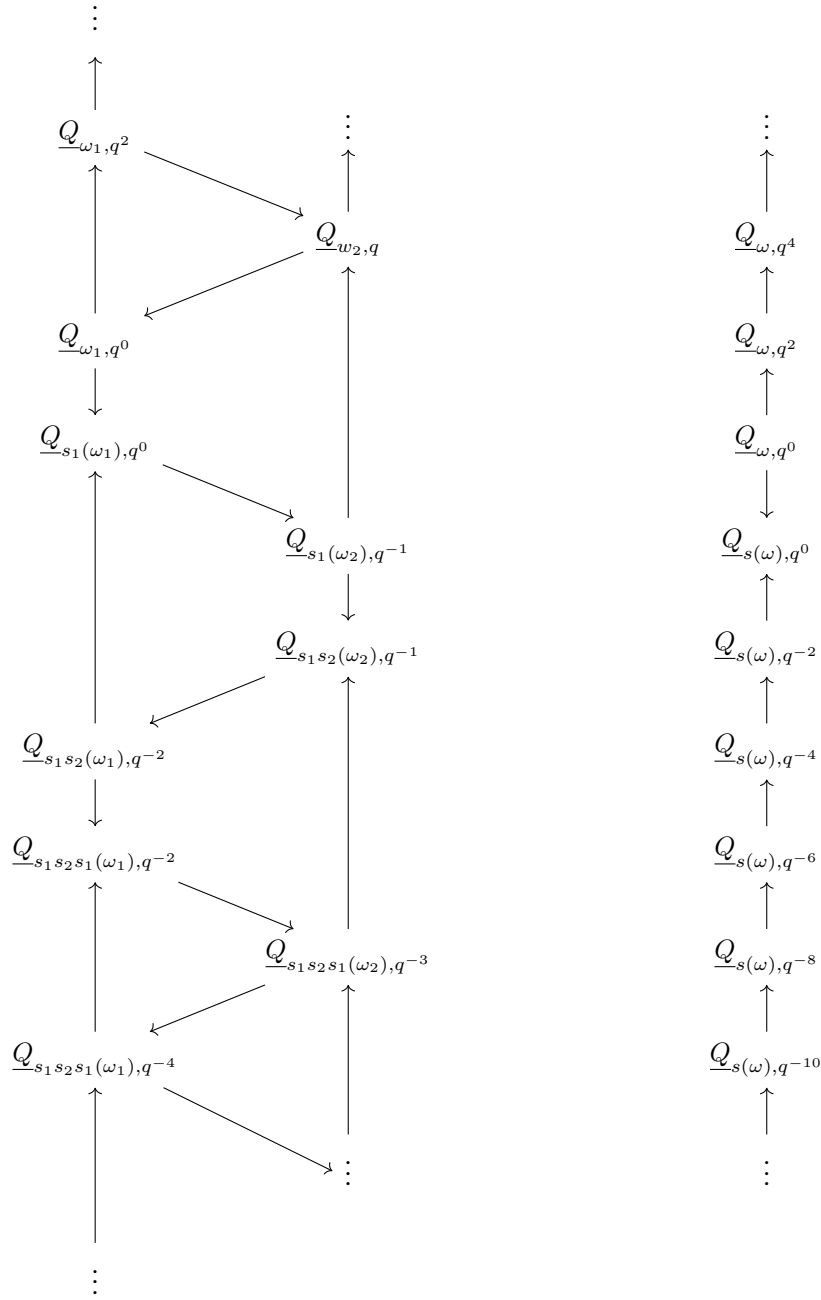 
    Theorem \ref{teo 7.4 ghl} tells us to see the initial cluster variables for $\mathcal{A}_{w_0}$ as in Figure \ref{fig:seed con q variables}. Then one can observe that replacing each $\underline{Q}$-variable in Figure \ref{fig:seed con q variables} with the leading term of the non-normalized $Q$-variable, we recover the quiver on the right of Figure \ref{fig:stabilized g vectors}.
    \end{example}

    Now we want to compare the quantum tori $(T_{t,e},\ast_e)$ from Definition \ref{quasi toro lea} and $(T_{t,c},\ast_c)$. To do so, we identify $T_{t,e}$ with a deformed subalgebra of $\mathcal{Y}'$, namely we take as generators the elements $\Psi_{i,q^r}^{\pm 1}$, for $(i,r)\in V$, that for brevity we might denote $\Psi_v^{\pm 1}$, for $v\in V$. The matrix $\Lambda_e$ is the same defined in Equation~\eqref{lambda e}. 
    \begin{theorem}\label{iso tori}
        There is an isomorphism of $\ZZ\big[t^{\pm\frac{1}{2}}\big]$-algebras
        \[\mathcal{G}:T_{t,c}\to T_{t,e}\] induced by the assignment
        \[Q_v\mapsto \prod_{v'\in V}  \Psi_{v'}^{G^{(\infty)}_{v',v}},\] where $\prod_{v'\in V}\Psi_{v'}^{G^{(\infty)}_{v',v}}$ is a commutative monomial, as defined in Remark \ref{commutative monomial}. Thus, we have also
        \[\mathcal{T}_{t,c}\simeq\mathcal{T}_{t,e}.\]
    \end{theorem}
    \begin{proof}
        The definition of $\mathcal{G}$ on the generators makes sense thanks to Lemma \ref{g vettori corrispondon a hw} and Remark \ref{notazione leading term}. Moreover, since $G^{(\infty)}$ is invertible, $\mathcal{G}$ is an isomorphism. We are left to prove that $\mathcal{G}$ is well defined on the $t$-commutation relations. So, let us fix $v,v'\in V$ and compute:
        \begin{align*}
            \mathcal{G}(Q_v)\ast_e \mathcal{G}(Q_{v'})&= \left(\prod_{u\in V}\left(\Psi_u^{G^{(\infty)}_{u,v}}\right)\right)\ast_e \left(\prod_{w\in V}\left(\Psi_w^{G^{(\infty)}_{w,v'}}\right)\right)\\
            &=t^{\sum_{u,w\in V}G^{(\infty)}_{u,v}G^{(\infty)}_{w,v'}(\Lambda_e)_{u,w}} \mathcal{G}(Q_{v'}) \ast_e  \mathcal{G}(Q_v)\\
            &=t^{\sum_{u,w\in V}(G^{(\infty)})^T_{v,u}(\Lambda_e)_{u,w}G^{(\infty)}_{w,v'}}\mathcal{G}(Q_{v'}) \ast_e  \mathcal{G}(Q_v)\\
            &= t^{(G^{(\infty)^T}\Lambda_eG^{(\infty)})_{v,v'}}\mathcal{G}(Q_{v'}) \ast_e  \mathcal{G}(Q_v)\\
            &=t^{(\Lambda_c)_{v,v'}}\mathcal{G}(Q_{v'}) \ast_e  \mathcal{G}(Q_v).
        \end{align*}
        This concludes the proof.
    \end{proof}
    
    \begin{remark}\label{rk commutazione q var}
        This Theorem enables us to understand how two $Q$-variables of the form $Q_v$, $Q_{v'}$ $t$-commute, by considering their highest weight monomials and then using  $\Lambda_e$. Then, we obtain the relations for the $\underline{Q}$-variables as well, since they are a normalized version of the $Q$-variables.
    \end{remark}

    \begin{example}
        We continue Examples \ref{esempio A1 1} and \ref{esempio A1 2} in type $A_1$. By Theorem \ref{teo 7.4 ghl} and Example \ref{esempio A1 1} (the computation of stabilized $g$-vectors) we know that the initial cluster variables are 
        \[\underline{Q}_{\omega,q^r},\ \forall r\geq 0,\qquad \underline{Q}_{s(\omega),q^r}\ \forall r\leq -2,\qquad (r\in2\ZZ).\] Now, in order to understand how they $t$-commute, we look at the leading terms of the corresponding $Q$-variables, that are respectively:
        \[\Psi_{q^r},\ \forall r\geq 0,\qquad \Psi_{q^{r-2}}^{-1},\ \forall r\leq -2,\qquad (r\in 2\ZZ).\] Therefore, by the observations in Remark \ref{rk commutazione q var}, the cluster variables $t$-commute as the leading terms. 
        Concretely, we have that for all $s,r\in 2\ZZ$, $s\geq  r,$ setting $m:=\frac{s-r}{2}$, using the function $f$ from Example~\ref{f}, we have:
        \begin{align*}
            \underline{Q}_{\omega,q^s}\ast_c\underline{Q}_{\omega,q^r}&=t^
            {f(m)}\underline{Q}_{\omega,q^r}\ast_c \underline{Q}_{\omega,q^s}\\
            \underline{Q}_{s(\omega),q^s}\ast_c \underline{Q}_{s(\omega),q^r}&=t^{f(m)}\underline{Q}_{s(\omega),q^r}\ast_c \underline{Q}_{s(\omega),q^s}\\
            \underline{Q}_{\omega,q^s}\ast_c \underline{Q}_{s(\omega),q^r}&=t^{-f(m)}\underline{Q}_{s(\omega),q^r}\ast_c \underline{Q}_{\omega,q^s}.
            \end{align*} One can compare this with the matrix $\Lambda_c$ obtained in Example \ref{esempio A1 2} and observe that the first two type of $t$-commutations are encoded in the diagonal blocks, while the third relation correspond to the bottom-left block. We will say more about type $A_1$ in Section \ref{section q oscillator}.
    \end{example}
    \begin{example}\label{esempio A2}
           We treat now type $A_2$. See Figure \ref{fig:2 quiver A2} for the quiver $\Gamma_e$ and $\Gamma_c$. Applying Equation~\eqref{lambda e} to our case (see \cite[Example 3.23]{b21} for the Cartan datum) we obtain the matrix $\Lambda_e$, that we represent here as a table to make it more readable.\\
           
          \begin{table}[H]
          $\Lambda_e=$\\[5pt]
            \begin{tabular}{l|l|l|l|l|l|l|l|l|l|l|l}
        
             & (1,-8) & (2,-7) & (1,-6) & (2,-5) & (1,-4) & (2,-3) & (1,-2) & (2,-1) & (1,0) & (2,1) & (1,2)  \\ \hline
        (1,-8)  & 0      & 0      & -1     & -1     & -1     & 0      & 0      & 0      & -1    & -1  & -1   \\ \hline
        (2,-7)  & 0      & 0      & 0      & -1     & -1     & -1     & 0      & 0      & 0     & -1  & -1  \\ \hline
        (1,-6)  & 1      & 0      & 0      & 0      & -1     & -1     & -1     & 0      & 0     & 0   & -1  \\ \hline
        (2,-5)  & 1      & 1      & 0      & 0      & 0      & -1     & -1     & -1     & 0     & 0   & 0  \\ \hline
        (1,-4)  & 1      & 1      & 1      & 0      & 0      & 0      & -1     & -1     & -1    & 0   & 0  \\ \hline
        (2,-3)  & 0      & 1      & 1      & 1      & 0      & 0      & 0      & -1     & -1    & -1  & 0  \\ \hline
        (1,-2)  & 0      & 0      & 1      & 1      & 1      & 0      & 0      & 0      & -1    & -1  & -1  \\ \hline
        (2,-1)  & 0      & 0      & 0      & 1      & 1      & 1      & 0      & 0      & 0     & -1  & -1  \\ \hline
        (1,0)   & 1      & 0      & 0      & 0      & 1      & 1      & 1      & 0      & 0     & 0   & -1  \\ \hline
        (2,1)   & 1      & 1      & 0      & 0      & 0      & 1      & 1      & 1      & 0     & 0   & 0  \\ \hline
        (1,2)   & 1      & 1      & 1      & 0      & 0      & 0      & 1      & 1      & 1     & 0   & 0
        \\
        \end{tabular}
        \end{table}
    
    Moreover, following \cite[Example 4.13]{ghl} we can compute the matrix of stabilized $g$-vectors $G^{(\infty)}$.\\
    
    \begin{table}[H]
    $G^{(\infty)}=$\\[5pt]
    \begin{tabular}{l|l|l|l|l|l|l|l|l|l|l|l}
    
            & (1,-8) & (2,-7) & (1,-6) & (2,-5) & (1,-4) & (2,-3) & (1,-2) & (2,-1) & (1,0) & (2,1) & (1,2) \\ \hline
    
    (1,-8)  & 0      & -1     &        &        &        &        &        &        &       &     &      \\ \hline
    (2,-7)  & -1     & 0      &        &        &        &        &        &        &       &     &  \\ \hline
    (1,-6)  &        &        & 0      & -1     &        &        &        &        &       &     & \\ \hline
    (2,-5)  &        &        & -1     & 0      &        &        &        &        &       &     &  \\ \hline
    (1,-4)  &        &        &        &        & -1     & -1     &        &        &       &     &  \\ \hline
    (2,-3)  &        &        &        &        & 1      & 0      &        &        &       &     &  \\ \hline
    (1,-2)  &        &        &        &        &        &        & -1     & 0      &       &     &   \\ \hline
    (2,-1)  &        &        &        &        &        &        & 1      & 1      &       &     &   \\ \hline
    (1,0)   &        &        &        &        &        &        &        &        & 1     & 0   & \\ \hline
    (2,1)   &        &        &        &        &        &        &        &        & 0     & 1   &  \\ \hline
    (1,2)   &        &        &        &        &        &        &         &       &       &     & 1 \\
    \end{tabular}
    \end{table}
    where the other blocks that we cannot represent are of the form
    \[\begin{pmatrix}
        0 & -1\\
        -1 & 0
    \end{pmatrix}\] on the top left corner and the identity on the bottom right one.
    Thus, using Definition \ref{definizione lambda c}, we can compute $\Lambda_c$, a finite part of which is of the form\\
    
    \begin{table}[H]
    $\Lambda_c=$\\[5pt]
    \begin{tabular}{l|l|l|l|l|l|l|l|l|l|l|l}
           & (1,-8) & (2,-7) & (1,-6) & (2,-5) & (1,-4) & (2,-3) & (1,-2) & (2,-1) & (1,0) & (2,1) & (1,2) \\ \hline
    (1,-8) & 0      & 0      & -1     & 0      & 0      & -1     & 0      & 0      & 0     & 1     & 1     \\ \hline
    (2,-7) & 0      & 0      & -1     & -1     & -1     & -1     & 0      & 0      & 1     & 1     & 1     \\ \hline
    (1,-6) & 1      & 1      & 0      & 0      & 1      & 0      & 0      & 1      & 0     & 0     & 0     \\ \hline
    (2,-5) & 0      & 1      & 0      & 0      & 0      & -1     & -1     & 0      & 0     & 0     & 1     \\ \hline
    (1,-4) & 0      & 1      & -1     & 0      & 0      & 0      & -1     & 0      & 0     & -1    & 0     \\ \hline
    (2,-3) & 1      & 1      & 0      & 1      & 0      & 0      & 0      & 1      & 1     & 0     & 0     \\ \hline
    (1,-2) & 0      & 0      & 0      & 1      & 1      & 0      & 0      & 0      & 1     & 0     & 0     \\ \hline
    (2,-1) & 0      & 0      & -1     & 0      & 0      & -1     & 0      & 0      & 0     & -1    & -1    \\ \hline
    (1,0)  & 0      & -1     & 0      & 0      & 0      & -1     & -1     & 0      & 0     & 0     & -1    \\ \hline
    (2,1)  & -1     & -1     & 0      & 0      & 1      & 0      & 0      & 1      & 0     & 0     & 0     \\ \hline
    (1,2)  & -1     & -1     & 0      & -1     & 0      & 0      & 0      & 1      & 1     & 0     & 0    
    \end{tabular}
    \end{table}
    
    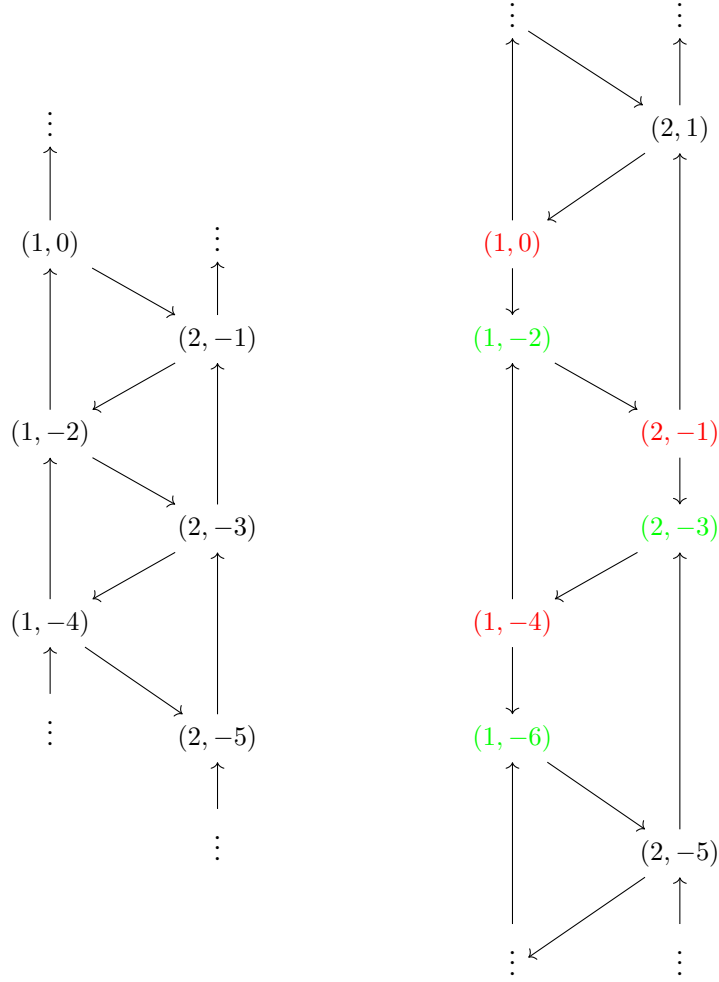
\begin{figure}[H]
        \centering
       \[\begin{tikzcd}
    	&&&& \vdots & \vdots \\
    	\vdots &&&&& {(2,1)} \\
    	{(1,0)} & \vdots &&& {\color{red}{(1,0)}} \\
    	& {(2,-1)} &&& {\color{green}{(1,-2)}} \\
    	{(1,-2)} &&&&& {\color{red}{(2,-1)}} \\
    	& {(2,-3)} &&&& {\color{green}{(2,-3)}} \\
    	{(1,-4)} &&&& {\color{red}{(1,-4)}} \\
    	\vdots & {(2,-5)} &&& {\color{green}{(1,-6)}} \\
    	& \vdots &&&& {(2,-5)} \\
    	&&&& \vdots & \vdots
    	\arrow[from=1-5, to=2-6]
    	\arrow[from=2-6, to=1-6]
    	\arrow[from=2-6, to=3-5]
    	\arrow[from=3-1, to=2-1]
    	\arrow[from=3-1, to=4-2]
    	\arrow[from=3-5, to=1-5]
    	\arrow[from=3-5, to=4-5]
    	\arrow[from=4-2, to=3-2]
    	\arrow[from=4-2, to=5-1]
    	\arrow[from=4-5, to=5-6]
    	\arrow[from=5-1, to=3-1]
    	\arrow[from=5-1, to=6-2]
    	\arrow[from=5-6, to=2-6]
    	\arrow[from=5-6, to=6-6]
    	\arrow[from=6-2, to=4-2]
    	\arrow[from=6-2, to=7-1]
    	\arrow[from=6-6, to=7-5]
    	\arrow[from=7-1, to=5-1]
    	\arrow[from=7-1, to=8-2]
    	\arrow[from=7-5, to=4-5]
    	\arrow[from=7-5, to=8-5]
    	\arrow[from=8-1, to=7-1]
    	\arrow[from=8-2, to=6-2]
    	\arrow[from=8-5, to=9-6]
    	\arrow[from=9-2, to=8-2]
    	\arrow[from=9-6, to=6-6]
    	\arrow[from=9-6, to=10-5]
    	\arrow[from=10-5, to=8-5]
    	\arrow[from=10-6, to=9-6]
    \end{tikzcd}\]
        \caption{The quivers $\Gamma_e$ on the left and $\Gamma_c$ on the right.}
        \label{fig:2 quiver A2}
    \end{figure}
    
    \end{example}

    \begin{definition}\label{def quantum groth}
        We define the quantum Grothendieck ring for the category $\mathcal{O}^{\sh}_{\ZZ}$ to be
        \[K_t(\mathcal{O}^{\sh}_{\ZZ}):=\mathcal{E}\hat{\otimes}_{\ZZ}\mathcal{A}_{t,w_0}.\]  
    \end{definition}
    In particular, we have a natural embedding $$K_t(\mathcal{O}^{\sh}_{\ZZ})\subset\mathcal{T}_{t,c}.$$

    In the next theorem we prove that the embedding \eqref{aaa} admits a quantum deformation. 
    \begin{theorem}\label{inclusione anelli}
        There are two injective ring morphisms
        \begin{align}
            \mathcal{I}^+_t\ &:\ K_t(\mathcal{O}^{\mathfrak{b},+}_{\ZZ})\to K_t(\mathcal{O}^{\sh}_{\ZZ})\\
            \mathcal{I}^-_t\ &:\ K_t(\mathcal{O}^{\mathfrak{b},-}_{\ZZ})\to K_t(\mathcal{O}^{\sh}_{\ZZ}),   
        \end{align}
        where the categories $\mathcal{O}^{\mathfrak{b},\pm}_{\ZZ}$ and their quantum Grothendieck rings have been recalled in Section \ref{section on cluster structure on gr rings}.
    \end{theorem}
    \begin{proof}
        We explain the proof for the first morphism $\mathcal{I}^+_t$, since the other one can be handled similarly. We have already seen that the ambient quantum tori $\mathcal{T}_{t,c}$ and $\mathcal{T}_{t,e}$ are isomorphic by Theorem \ref{iso tori}. Hence it suffices to check that $K_t(\mathcal{O}^{\mathfrak{b},+}_{\ZZ})\subset K_t(\mathcal{O}^{\sh}_{\ZZ})$, inside the isomorphic quantum tori. The arguments are the same as for classical cluster algebras above, for which we have already seen that $\mathcal{A}_e$ embeds into $\mathcal{A}_{w_0}$ (see Lemma \ref{zzz}). 
    \end{proof}
    \begin{remark}
        Combining this with Theorem \ref{inclusione tori hl e lea}, we obtain that the quantum \Groth of finite-dimensional representations of the ordinary quantum affine algebra is a natural subring of our quantum \Groth ring.
    \end{remark}

    \section{Application 1: quantum QQ-systems}\label{section qq}
    Recall that Theorem \ref{qq system} states that the $\underline{Q}$-variables satisfy the $QQ$-systems. These equations are realized as exchange relations for the cluster algebra structure $\mathcal{A}_{w_0}$ on $K_0(\mathcal{O}^{\sh}_{\ZZ})$ by Theorem \ref{iso ghl cluster}. We would like to find quantum analogues of $QQ$-systems and to do so we use our quantum cluster algebra $\mathcal{A}_{t,w_0}$, in particular its quantum exchange relations.

    We have defined in this paper a quantification of the cluster algebra $\mathcal{A}_{w_0}$. To each cluster variable of $\mathcal{A}_{w_0}$ corresponds a quantum cluster variable. In particular, we have now well-defined quantum $\underline{Q}$-variables. Indeed, by \cite[Proposition 8.2]{ghl} for every $(i,r)$ in $V$ and every $w$ in $W$, $\underline{Q}_{w(\omega_i),q^r}$ is the image of a cluster variable of $\mathcal{A}_{w_0}$ through the morphism of Theorem \ref{teo 7.4 ghl}. By abuse of notation, for the quantum $\underline{Q}$-variables we will use  the same notation $\underline{Q}_{w(\omega_i),q^{r}}$ as for the classical $\underline{Q}$-variables.

    The definition of a quantum cluster algebra passes through the notion of toric frame (see \cite[Section 4.3]{bz05} for all the details). So, let $\mathcal{F}$ be the skew field of fractions of $\mathcal{T}_{t,c}$. Let $M:\ZZ^{(V)}\to \mathcal{F}\setminus\{0\}$ be the toric frame defined by 
    \[\boldsymbol{e}_v\mapsto \underline{Q}_v.\] In particular, for all $v,w\in V$ we have
    \[M(\boldsymbol{e}_v)M(\boldsymbol{e}_w)=t^{(\Lambda_c)v,w/2}M(\boldsymbol{e}_v+\boldsymbol{e}_w),\quad M(\boldsymbol{e}_v)M(\boldsymbol{e}_w)=t^{(\Lambda_c)v,w}M(\boldsymbol{e}_w)M(\boldsymbol{e}_v),\] where the matrix $\Lambda_c$ is seen as a bilinear form on $\ZZ^{(V)}$. Moreover, for every $v\in V$, we have 
    \[M(0)=1,\quad M(\boldsymbol{e}_v)^{-1}=M(-\boldsymbol{e}_v).\] By \cite[Proposition 4.9]{bz05}, setting $B_c=(b_{v,w})_{v,w\in V}$ the exchange matrix for the quiver $\Gamma_c$ , we have:
    \begin{equation}\label{bbb}
        \underline{Q}_v^*=M\left(-\boldsymbol{e}_v+\sum_{b_{w,v}>0}b_{w,v}\boldsymbol{e}_w\right)+M\left(-\boldsymbol{e}_v-\sum_{b_{w,v}<0}b_{w,v}\boldsymbol{e}_w\right),
    \end{equation} where $\underline{Q}_v^*$ denotes the mutated variable and the two monomials on the right are commutative (in the sense of Remark \ref{commutative monomial}).
    Thus, in order to obtain a $t$-deformed version of $QQ$-systems, we should determine the correct values of $A,B,C,D$ in 
    \begin{align}
       \underline{Q}_{ws_i(\omega_i),q^r}&= t^A\ \underline{Q}_{w(\omega_i),q^{r-2}}^{-1}\left(\underline{Q}_{ws_i(\omega_i),q^{r-2}}\underline{Q}_{w(\omega_i),q^r}\right) +t^B\ \underline{Q}_{w(\omega_i),q^{r-2}}^{-1}\prod_{j:c_{ij}=-1}\underline{Q}_{w(\omega_j),q^{r-1}}\label{nnn}\\
       \underline{Q}_{ws_i(\omega_i),q^r}&= t^C\ \left(\underline{Q}_{ws_i(\omega_i),q^{r-2}}\underline{Q}_{w(\omega_i),q^r}\right)\underline{Q}_{w(\omega_i),q^{r-2}}^{-1} +t^D\ \left(\prod_{j:c_{ij}=-1}\underline{Q}_{w(\omega_j),q^{r-1}}\right)\underline{Q}_{w(\omega_i),q^{r-2}}^{-1} \label{mmm}
    \end{align} which are equations inside the quantum torus $\mathcal{T}_{t,c}$. Note that in \eqref{nnn} and \eqref{mmm} $\underline{Q}_{ws_i(\omega_i),q^r}$ is the mutated variable. 

    Now from Equation \eqref{bbb} and by the properties of the toric frame, to compute the precise powers of $t$ we only have to check how the $\underline{Q}$-variables in such equations $t$-commute. This can be done using Remark \ref{rk commutazione q var} and Lemma \ref{T gives leading term}. We make below an example explicit in type $A_1$ (see Examples \ref{esempio A1 1} and \ref{esempio A1 2} for other computations in type $A_1$). Hence we derive the quantum $QQ$-system in this case.

    \begin{proposition}\label{quantum qq}
        For all $r\in 2\ZZ$, the following equations hold:
        \begin{align*}
             \underline{Q}_{\omega,q^{r-2}}\underline{Q}_{s(\omega),q^r}-t^{-1}\underline{Q}_{\omega,q^{r}}\underline{Q}_{s(\omega),q^{r-2}}&=1\\
             \underline{Q}_{s(\omega),q^r}\underline{Q}_{\omega,q^{r-2}}-t\  \underline{Q}_{\omega,q^{r}}\underline{Q}_{s(\omega),q^{r-2}}&=1
        \end{align*} 
    \end{proposition}
    These equations come from mutations at the source of one initial quiver for $\mathcal{A}_{t,w_0}$:

    \[\begin{tikzcd}
    	\cdots & {\underline{Q}_{s(\omega),q^{r-4}}} & {\underline{Q}_{s(\omega),q^{r-2}}} & {\underline{Q}_{\omega,q^{r-2}}} & {\underline{Q}_{\omega,q^{r}}} & \cdots
    	\arrow[from=1-1, to=1-2]
    	\arrow[from=1-2, to=1-3]
    	\arrow[from=1-4, to=1-3]
    	\arrow[from=1-4, to=1-5]
    	\arrow[from=1-5, to=1-6]
    \end{tikzcd}\]
    As far as the author knows, this the first time this deformation of the $QQ$-system is obtained. Note that the classical $QQ$-system is sometimes called the quantum Wronskian relation. The two notions of quantizations should not be confused here: our terminology emphasizes the fact that we obtain a relation between non-commutative variables.
    
    We consider again type $A_1$. If we mutate the initial quiver for $\mathcal{A}_{t,w_0}$ at a vertex belonging to an equi-oriented subquiver, we find another type of quantum exchange relations, called quantum Baxter relations. This is coherent with the fact that $K_t(\mathcal{O}^{\mathfrak{b,+}})$ is a subring of $K_t(\mathcal{O}^{\sh}_{\ZZ})$ by Theorem \ref{inclusione anelli}.

    \[\begin{tikzcd}
    	\cdots & {\underline{Q}_{\omega,q^{r-2}}} & {\underline{Q}_{\omega,q^{r}}} & {\underline{Q}_{\omega,q^{r+2}}} & {\underline{Q}_{\omega,q^{r+4}}} & \cdots
    	\arrow[from=1-1, to=1-2]
    	\arrow[from=1-2, to=1-3]
    	\arrow[from=1-3, to=1-4]
    	\arrow[from=1-4, to=1-5]
    	\arrow[from=1-5, to=1-6]
    \end{tikzcd}\]
    The quantum Baxter relations were first discovered by Bittmann. We denote by $x$ the variable obtained from $\underline{Q}_{\omega,q^r}$ by mutation. It is the $(q,t)$-character of a so-called fundamental module for $\mathcal{U}_q(\hat{\mathfrak{g}})$ (see {\cite[Proposition 8.2.1]{b21}}).

    \begin{proposition}
        For all $r\in 2\ZZ $ the following equation holds in $\mathcal{T}_{t,c}$:
        \begin{align}
            x&=t^{\frac{1}{2}}\ \underline{Q}^{-1}_{\omega,q^r}\underline{Q}_{\omega,q^{r+2}}+t^{-\frac{1}{2}}\underline{Q}^{-1}_{\omega,q^r}\underline{Q}_{\omega,q^{r-2}}\\
            x&=t^{-\frac{1}{2}}\ \underline{Q}_{\omega,q^{r+2}}\underline{Q}^{-1}_{\omega,q^r}+t^{\frac{1}{2}}\underline{Q}_{\omega,q^{r-2}}\underline{Q}^{-1}_{\omega,q^r}.
            \end{align} 
    \end{proposition}

    \section{Application 2: the quantum oscillator algebra as a quantum cluster algebra}\label{section q oscillator}

    In this section we prove that the quantum oscillator algebra (Definition \ref{quantum osc}) has the structure of a quantum cluster algebra inside our quantum Grothendieck ring. As a consequence, we obtain that the quantum oscillator algebra is also isomorphic to the quantum double Bruhat cell $\CC_t[\SL_2^{w_0,w_0}]$.\\
    Here we use the symbol $t$ for the quantum parameter in order to be consistent with the $t$-deformation of the Grothendieck ring that we have defined in previous sections. This means that the quantum oscillator algebra is denoted $\mathcal{U}_t^+(\sl_2)$.

    Following \cite{sqaahernandez}, we denote by $\mathcal{U}_{t,loc}^+(\sl_2)$ the localization of $\mathcal{U}_t^+(\sl_2)$ at the Casimir central element 
    \[C=ef+\frac{t^{-1}k}{(t-t^{-1})^2}=fe+\frac{tk}{(t-t^{-1})^2}.\]
    \begin{remark}
        We can use the Casimir element to give an equivalent definition of $\mathcal{U}_t^+(\sl_2)$ without the bracket relation, that is: the generators are $e,\ f,\ k^{\pm1},\ C$, with relations
        \[ke=t^2ek,\  kf=t^{-2}fk,\  kk^{-1}=k^{-1}k=1,\ C=ef+\frac{t^{-1}k}{(t-t^{-1})^2}=fe+\frac{tk}{(t-t^{-1})^2}.\] In this way the bracket relation follows easily:
        \[[e,f]=ef-fe=-\frac{t^{-1}k}{(t-t^{-1})^2}+\frac{tk}{(t-t^{-1})^2}=\frac{k}{(t-t^{-1})^2}(t-t^{-1})=\frac{k}{t-t^{-1}}.\]
    \end{remark}
    In what follows, by $\mathcal{U}_{t,\loc}^{+}$ we mean $\mathcal{U}_t^+(\sl_2)$ localized at the central element $C$ and extended with the square roots $k^{\pm\frac{1}{2}}$ of $k^{\pm 1}.$

    \subsection{The quantum cluster algebra $\mathcal{A}_t$}

    We aim at endowing $\mathcal{U}_{t,\loc}^+$ with the structure of quantum cluster algebra. First, we consider the quiver $\Gamma_c$, that in this case is the quiver on the left in Figure \ref{fig:alcuni quiver A1}. We extract from $\Gamma_c$ the full subquiver with vertices $\{-4,-2,0\}$ and we freeze the boundary vertices. We call $Q$ the resulting ice quiver. 

    \begin{figure}[h]
        \centering
        \[\begin{tikzcd}
    	Q:\  {\stackrel{-4}{\square}} && {\stackrel{-2}{\bullet}} && {\stackrel{0}{\square}}
    	\arrow[from=1-1, to=1-3]
    	\arrow[from=1-5, to=1-3]
    \end{tikzcd}\] 
        \caption{The ice quiver $Q$.}
        \label{fig:initial seed q osc}
    \end{figure}
    This will be an initial quiver for our new cluster algebra. We assign the cluster variables $\{a,b,c\}$ in the following way:

    \[\begin{tikzcd}
    	Q:\  {\stackrel{c}{\square}} && {\stackrel{a}{\bullet}} && {\stackrel{b}{\square}}
    	\arrow[from=1-1, to=1-3]
    	\arrow[from=1-5, to=1-3]
    \end{tikzcd}\] 
    
    As regards the quantization matrix, we can deduce it from the matrix $\Lambda_c$ of $\mathcal{A}_{t,w_0}$ from Equation \ref{def quantum cluster nostra} (see Example \ref{esempio A1 2} for the explicit computation of $\Lambda_c$ in type $A_1$). We extract from $\Lambda_c$ the $3\times 3$ submatrix with entries indexed by couples in $\{-4,-2,0\}$ and we call it $\bar{\Lambda}$. However, for our convenience, we prefer to reorder the columns and rows of $\bar{\Lambda}$ so that the first row/column will be the one indexed by $-2$ the second by $0$ and the third by $-4$. We call $\Lambda$ the resulting matrix.
    \begin{equation}
        \Lambda=\begin{pmatrix}
            0 & 1 & 1\\
            -1 & 0 & 0 \\
            -1 & 0 & 0
        \end{pmatrix}
    \end{equation}

    The exchange matrix for the quiver $\tilde{B}$ is a $3\times 1$ matrix with rows indexed as the rows of $\Lambda$ and the column that corresponds to the vertex $-2$.
    \[\tilde{B}=\begin{pmatrix}
        0\\
        1\\
        1
    \end{pmatrix}\]
    For the sake of clarity, we check that the couple $(\Lambda,\tilde{B})$ is compatible in the sense of quantum cluster algebras (see definition \ref{def compatibile}). Let us compute the product $\tilde{B}^T\Lambda$:
    \[\tilde{B}^T\Lambda=\begin{pmatrix}
        0 & 1 & 1
    \end{pmatrix}\begin{pmatrix}
        0 & 1 & 1\\
        -1 & 0 & 0\\
        -1 & 0 & 0
    \end{pmatrix}=\begin{pmatrix}
        -2 & 0 & 0
    \end{pmatrix}.\]
    
    Since the couple $(\Lambda,\tilde{B})$ is compatible, it defines a quantum cluster algebra, that we denote $\mathcal{A}_t$. The quantum torus in which it embeds is the $\ZZ\big[t^{\pm\frac{1}{2}}\big]$-algebra generated by $a^{\pm 1},b^{\pm 1},c^{\pm 1}$ subject to the $t$-commutation relations induced by $\Lambda.$ For later use, let us write down here the $t$-commutation relations:
    \begin{align}
        ab&=t\ ba \label{t rel 1}\\ 
        ac&=t\ ca \label{t rel 2}\\
        bc&=cb \label{t rel 3}.
    \end{align} 
     The quiver $Q$ has only one mutable vertex, that is $-2$. Let $a^*$ denote the new cluster variable obtained from $a$ by mutation of the seed at vertex $-2$. The quantum exchange relations read:
    \begin{align}
        aa^* &=1+t^{-1}\ bc \label{t ex 1}\\
        a^*a&=1+t\ bc \label{t ex 2}.
    \end{align} 
    Note that this is the same as the quantum $QQ$-system obtained in Proposition \ref{quantum qq}.\\
    The quantum cluster algebra $\mathcal{A}_t$ has only two clusters, so it is easy to present it as a $\ZZ\big[t^{\pm\frac{1}{2}}\big]$-algebra with generators $a,a^*,b^{\pm1},c^{\pm 1}$ subject to the $t$-commutation relations \eqref{t rel 1} \eqref{t rel 2}\eqref{t rel 3} and the quantum exchange relations \eqref{t ex 1}\eqref{t ex 2}. Note that we use the assumptions that frozen variables are invertible. In particular, $\mathcal{A}_{t}$ is not a subalgebra of $\mathcal{A}_{t,w_0}$ because its frozen variables are invertible, but it can be obtained as a localization of the subalgebra of $\mathcal{A}_{t,w_0}$ associated with the quiver $Q$ and with non invertible frozen vertices.
    \begin{remark}
        The element $z:=bc^{-1}$ is central in $\mathcal{A}_t$.
    \end{remark}

\subsection{The quantum double Bruhat cell $\CC[\SL_2^{w_0,w_0}]$}
    Double Bruhat cells were introduced by Fomin and Zelevinsky in \cite{fzbruhat}. We recall the definition of the double Bruhat cell $\SL_2^{w_0,w_0}$, where $w_0$ denotes the longest element of the Weyl group of $\sl_2$, thus in this case it is the simple reflection. It can be described explicitly as the set of complex matrices 
    \[M=\begin{pmatrix}
        a & b\\
        c & d
    \end{pmatrix}\] such that $\mathrm{det}(M)=1$ and $b,c$ are nonzero. This is an open, dense subset of $\SL_2(\CC)$. The quantum double Bruhat cell $\CC_t[\SL_2^{w_0,w_0}]$, defined in \cite{bz05}, has a presentation with generators $a,b^{\pm 1},c^{\pm 1},d$ subject to the following relations:
    \[ab=t\ ba,\ cd=t\ dc,\ ac=t\ ca,\ bd=t\ db,\] 
    \[bb^{-1}=b^{-1}b=1,\  cc^{-1}=c^{-1}c=1,\]
    plus the commutation relation
    \[bc=cb\] and the $t$-determinants relations
    \[ad-t^{-1}bc =1=da-t\ bc.\]

    Recall the presentation of the quantum cluster algebra $\mathcal{A}_t$ by generators $a,a^*,b^{\pm 1},c^{\pm 1}$ and relations.
    \begin{proposition}\label{iso alg cluster anello coord sl2}
        The assignment
        \begin{align*}
            a&\mapsto a\\
            b^{\pm 1} &\mapsto b^{\pm 1}\\
            c^{\pm 1} &\mapsto c^{\pm 1}\\
            d &\mapsto a^*
        \end{align*}
        extends to an isomorphism between $\CC_t[\SL_2^{w_0,w_0}]$ and the quantum cluster algebra $\mathcal{A}_t$.
    \end{proposition}

    \begin{proof}
        Using the assignment in the hypothesis, the defining relations of $\mathcal{A}_t$ correspond to those defining $\CC_t[\SL_2^{w_0,w_0}]$. Indeed the quantum determinant identities correspond to the quantum exchange relations. 
    \end{proof}
    \begin{remark}
        This Proposition is a special case of a more general theorem. We discuss this in Section~\ref{further questions}.
    \end{remark}
    \subsection{Isomorphism between $\mathcal{U}_{t,\loc}^+$ and $\mathcal{A}_t$}

    We consider scalar extensions $\CC(t)\otimes_{\ZZ[t^{\pm 1}]}\mathcal{U}_{t,\loc}^+$ and $\CC(t)\otimes_{\ZZ[t^{\pm 1}]}\mathcal{A}_t$, but we keep denoting them respectively $\mathcal{U}_{t,\loc}^+$ and $\mathcal{A}_t$. Recall the presentation of the $t$-oscillator algebra with the Casimir element. We construct a morphism of algebras
    \[\varphi:\mathcal{U}_{t,\loc}^+\to \mathcal{A}_t\] by assigning
    \begin{alignat*}{2}
        &e &&\mapsto za\\
        &f &&\mapsto a^*\\
        &k^{\pm 1} &&\mapsto \left(-(t-t^{-1})^2b^2\right)^{\pm 1}\\
        &C^{\pm1} &&\mapsto z^{\pm 1}.
    \end{alignat*} This is defined on the generators of $\mathcal{U}_{t,\loc}^+$. We have to check that $\varphi$ is well defined on the relations:
    \begin{gather*}
        \varphi(k)\varphi(e)=-(t-t^{-1})^2b^2za= -t^2(t-t^{-1})^2zab^2=t^2 \varphi(e)\varphi(k)\\
        \varphi(k)\varphi(f)=-(t-t^{-1})^2b^2a^*= -t^{-2}(t-t^{-1})^2a^*b^2=t^{-2}\varphi(f)\varphi (k)\\
        \varphi\left(ef+\frac{t^{-1}k}{(t-t^{-1})^2}\right)=zaa^*-t^{-1}b^2=z(1+t^{-1}bc)-t^{-1}b^2=z+t^{-1}bc^{-1}bc-t^{-1}b^2=z=\varphi(C).
    \end{gather*}
    We leave to the reader the remaining relations since they are similar to prove. We have proved that $\varphi$ is a well defined morphism of algebras.
    We construct the inverse of $\varphi$ in order to get an isomorphism.
    Let us call
    \[\psi:\mathcal{A}_{t}\to \mathcal{U}_{t,\loc}^+,\] the morphism of algebras defined on the generators by
    \begin{align*}
        a &\mapsto C^{-1}e\\
        a^* &\mapsto f\\
        b^{\pm 1} &\mapsto \left(\frac{ik^{\frac{1}{2}}}{t-t^{-1}}\right)^{\pm 1}\\
        c^{\pm 1} &\mapsto \left(\frac{iC^{-1}k^{\frac{1}{2}}}{t-t^{-1}}\right) ^{\pm 1}.
    \end{align*}
    As a consequence, $\psi(z)=C$.
    One can verify that $\psi$ is in fact well defined. For example, on the exchange relation we have
    \[\psi(a^*)\psi(a)=C^{-1}fe=C^{-1}\left(C-\frac{tk}{(t-t^{-1})^2}\right)=1+t\psi(b)\psi(c).\]
    Let us check that $\varphi$ is a right inverse of $\psi$.
    \begin{align*}
     e &\xmapsto{\varphi} za\xmapsto{\psi} CC^{-1}e=e,\\
        f &\xmapsto{\varphi} a^* \xmapsto{\psi} f,\\
        k &\xmapsto{\varphi} -(t-t^{-1})^2b^2 \xmapsto{\psi}-(t-t^{-1})^2\frac{(-k)}{(t-t^{-1})^2}=k,\\
        C&\xmapsto{\varphi} z\xmapsto{\psi} C.
    \end{align*}
    Similarly, one can prove that $\varphi$ is a left inverse for $\psi$. Thus, $\varphi$ and $\psi$ are inverse to each other, so we have proved
    \begin{theorem}\label{teo q osc is cluster}
        The map $\varphi:\mathcal{U}_{t,\loc}^+\to \mathcal{A}_t$ is an isomorphism of algebras. That is, the localized $t$-oscillator algebra is a quantum cluster algebra.
    \end{theorem}
    \begin{corollary}\label{cor q osc}
        The quantum double Bruhat cell $\CC_t[\SL_2^{w_0,w_0}]$ is isomorphic to the localized quantum oscillator algebra $\mathcal{U}_{t,\loc}^+$.
    \end{corollary}

    \section{Further questions and possible applications}\label{further questions}

    We present here some open questions that arise naturally from our previous results. We plan to investigate more these directions in a forthcoming paper.

    \subsection{Quantum double Bruhat cells}\label{bruhat cells}

    An interesting consequence of the construction made in \cite{ghl} is explained in Section 10 of the same work. Namely, let us call $\gamma_c$ the full subquiver of $\Gamma_c$ whose vertices are all the green and red vertices plus the $n$ vertices immediately above the red vertices (see Section \ref{ghl construction} for the construction of $\Gamma_c$). Then, these last $n$ vertices and the last green vertex of each of the $n$ columns are frozen in $\gamma_c$ and we paint them in blue. Let us call $U\subset V$ the vertex set of $\gamma_c$. See Figure \ref{fig:ice quiver} for examples.

    \begin{figure}[H]
        \centering
       
    \[\begin{tikzcd}
    	&&& {\color{blue}{(1,2)}} \\
    	{\color{blue}{2}} &&&& {\color{blue}{(2,1)}} \\
    	0 &&& {(1,0)} \\
    	{\color{blue}{-2}} &&& {(1,-2)} \\
    	&&&& {(2,-1)} \\
    	&&&& {\color{blue}{(2,-3)}} \\
    	&&& {(1,-4)} \\
    	&&& {\color{blue}{(1,-6)}}
    	\arrow[from=2-5, to=3-4]
    	\arrow[from=3-1, to=2-1]
    	\arrow[from=3-1, to=4-1]
    	\arrow[from=3-4, to=1-4]
    	\arrow[from=3-4, to=4-4]
    	\arrow[from=4-4, to=5-5]
    	\arrow[from=5-5, to=2-5]
    	\arrow[from=5-5, to=6-5]
    	\arrow[from=6-5, to=7-4]
    	\arrow[from=7-4, to=4-4]
    	\arrow[from=7-4, to=8-4]
    \end{tikzcd}\]
        \caption{The quivers $\gamma_c$ in type $A_1$ and $A_2$.}
        \label{fig:ice quiver}
    \end{figure}
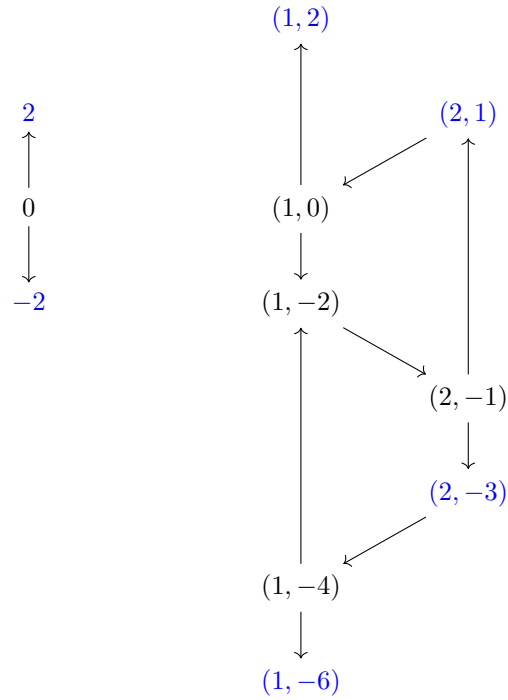
    Let $G$ be a simply-connected complex algebraic group with Lie algebra $\g$. Then, for any two elements of the Weyl group $u,v\in W$, Berenstein, Fomin and Zelevinsky introduced in \cite{bfz3} the double Bruhat cell $G^{u,v}$. In that paper it is shown that the coordinate ring $\CC[G^{u,v}]$ has the structure of cluster algebra, where the cluster variables correspond to some generalized minors. In particular, one can consider the double Bruhat cell associated with the longest element in the Weyl group, $G^{w_0,w_0}$. If we call $\mathcal{B}_c$ the cluster subalgebra of $\mathcal{A}_{w_0}$ with initial seed given by the quiver $\gamma_c$, then $\CC\otimes\mathcal{B}_c$ is isomorphic to the cluster algebra $\CC[G^{w_0,w_0}]$, as stated in \cite[Proposition 10.3]{ghl}. It should be noted that the initial seed used in \cite{bfz3} differs from the one associated with $\gamma_c$, but  they are related by mutation. Yakimov and Goodearl proved in \cite{yg} a conjecture stated by Berenstein and Zelevinsky in \cite{bz05}, that is 
    \begin{theorem}[{\cite[Theorem 9.5]{yg}}]\label{teo yg}
        For all connected, simply-connected, complex, simple, algebraic group $G$ and for all elements $u,v$ in the Weyl group $G$, the quantized coordinate ring $\CC_t[G^{u,v}]$ of the double Bruhat cell $G^{u,v}$ has a quantum cluster algebra structure.  
    \end{theorem}
    We refer the reader to Section 8 in \cite{bz05} for the recipe to obtain an initial quantum seed for $\CC_t[G^{w_0,w_0}.]$
    \begin{remark}
        Proposition \ref{iso alg cluster anello coord sl2} is just an instance of Theorem \ref{teo yg}. In fact, for our case the isomorphism between the quantum double Bruhat cell and the quantum cluster algebra was already proved in \cite[Section 8]{bz05}. 
    \end{remark}

    Since $\mathcal{B}_c$ is a subalgebra of $\mathcal{A}_{w_0}$ and we have just constructed the quantum cluster algebra $\mathcal{A}_{t,w_0}$, we get a quantum cluster algebra structure on $\mathcal{B}_c$ that we denote $\mathcal{B}_{t,c}$. The quantization matrix compatible with the seed associated with $\gamma_c$ is nothing but the submatrix extracted from $\Lambda_c$ with entries in $U\times U$ and we denote it by $\bar{\Lambda}_c$. Thus, it is a natural question to ask if the quantum cluster algebra $\CC_t[G^{w_0,w_0}]$ by Fomin--Zelevinsky and Yakimov--Goodearl is the same as $\mathcal{B}_{t,c}$. For type $A_1$ and $A_2$, where we have explicit computations, the answer is yes and we propose here these examples. 

    First, we set some notations. The order on the vertex set $U$ of the quiver $\gamma_c$ is the one induced by the order on $V$ (see \ref{total order}); the same goes for the columns and rows of the quantization matrices. Moreover, we denote by $Q$ the quiver for the initial seed of $\CC[G^{w_0,w_0}]$ as prescribed in \cite{bz05}; we call $\Lambda_{BZ}$ the quantization matrix associated with such seed  (see \cite[Sec. 8]{bz05}) and $\Lambda_0$ the matrix obtained from $\Lambda_{BZ}$ by matrix mutation (as defined in \cite[Eq.(3.4)]{bz05}), where the mutation sequence is the one relating $Q$ and $\gamma_c$.
 
    \begin{example}
        Let us consider $G=\SL_2(\CC)$, so $\g=\sl_2$. The quiver $\gamma_c$ is represented on the left in Figure \ref{fig:ice quiver}.
        Using Examples \ref{esempio A1 1} and \ref{esempio A1 2}, we have that 
        \begin{equation*}
            \bar{\Lambda}_c:= \begin{pmatrix}
                0 & 1 & 0\\
                -1 & 0  & -1\\
                0 & 1 & 0
            \end{pmatrix}
        \end{equation*} 
        On the other hand, we have that the quiver $Q$ is just the mutation at vertex $0$ of $\gamma_c$, see Figure \ref{fig:Q A1}.
       
    \begin{figure}[H]
        \centering
        \[\begin{tikzcd}
    	{\textcolor{blue}{2}} \\
    	0 \\
    	{\textcolor{blue}{-2}}
    	\arrow[from=1-1, to=2-1]
    	\arrow[from=3-1, to=2-1]
    \end{tikzcd}\]
        \caption{The quiver $Q$ in type $A_1$}
        \label{fig:Q A1}
    \end{figure}
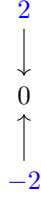
        
        Following \cite[Section 8]{bz05}, we see that the quantization matrix compatible with the seed of $Q$ is
        \begin{equation*}
            \Lambda_{BZ}=
            \begin{pmatrix}
                0 &  -1 & 0\\
                1 & 0 & 1\\
                0 & -1 & 0
            \end{pmatrix}
        \end{equation*}
       Now, applying matrix mutation to $\Lambda_{BZ}$ we obtain $\Lambda_0$ and we see that $\bar{\Lambda}_c=\Lambda_0$.
    \end{example}

    \begin{example}
        We consider $G=\SL_3(\CC)$, so $\g=\sl_3$. In this case the quiver $Q$ is shown in Figure \ref{fig:quiver bz}.
    
           \begin{figure}[H]
             
    \[\begin{tikzcd}
    	& {\textcolor{blue}{(2,1)}} \\
    	{\textcolor{blue}{(1,2)}} \\
    	{(1,0)} \\
    	& {(2,-1)} \\
    	{(1,-2)} \\
    	{(1,-4)} \\
    	& {\textcolor{blue}{(2,-3)}} \\
    	{\textcolor{blue}{(1,-6)}}
    	\arrow[from=1-2, to=3-1]
    	\arrow[from=3-1, to=2-1]
    	\arrow[from=3-1, to=4-2]
    	\arrow[from=4-2, to=1-2]
    	\arrow[from=4-2, to=5-1]
    	\arrow[from=4-2, to=7-2]
    	\arrow[from=5-1, to=3-1]
    	\arrow[from=5-1, to=6-1]
    	\arrow[from=6-1, to=4-2]
    	\arrow[from=6-1, to=8-1]
    	\arrow[from=7-2, to=6-1]
    \end{tikzcd}\]
    
               \caption{The quiver $Q$ in type $A_2$.}
               \label{fig:quiver bz}
           \end{figure}
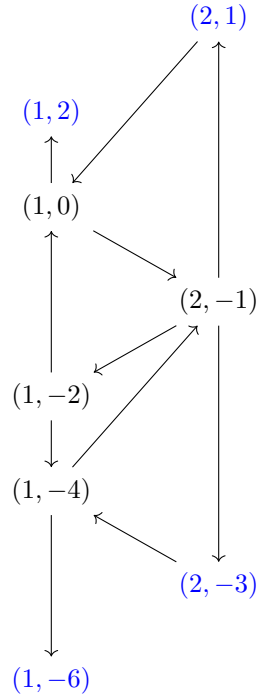
           For such initial seed, the quantization matrix $\Lambda_{\textnormal{BZ}}$ is the following
    
    \begin{equation*}\Lambda_{BZ}=
        \begin{pmatrix}
        0 & -1 & 0 & -1 & -1 & 0 & 0 & 0\\
        1 & 0 & 0 & -1 & 0 & 0 & 1 & 0\\
        0 & 0 & 0 & -1 & -1 & -1 & 0 & 0\\
        1 & 1 & 1 & 0 & 0 & 1 & 1 & 1\\
        1 & 0 & 1 & 0 & 0 & 0 &1 & 1\\
        0 & 0 & 1 & -1 & 0 & 0 & 0 & 1\\
        0 & -1 & 0 & -1 & -1 & 0 & 0 & 0\\
        0 & 0 & 0 & -1 & -1 & -1 & 0 & 0
    \end{pmatrix}
    \end{equation*}
           After mutating $Q$ at vertex $(1,-2)$ we obtain the quiver $\gamma_c$ (on the right in Figure \ref{fig:ice quiver}) and the mutation on $\Lambda_{\textnormal{BZ}}$ produces $\Lambda_0$. The matrix $\bar{\Lambda}_c$ can be obtained from $\Lambda_c$ in Example \ref{esempio A2}. We can observe that $\Lambda_0=-\bar{\Lambda}_c$.
    \end{example}

    These examples motivate the following
    \begin{conjecture}
        Let $G$ be a simply-connected, simple, complex algebraic group with Lie algebra $\g$ of simply-laced type. 
        The quantum cluster algebra structure induced by $\mathcal{A}_{t,w_0}$ on the coordinate ring of the double Bruhat cell $G^{w_0,w_0}$ is the same, up to rescaling, as the one in \cite{bz05} and \cite{yg} on $\CC_t[G^{^{w_0,w_0}}]$. More precisely, using the notation from the previous examples, for every $G$ as in the hypothesis, there exists an integer $k$ such that $\Lambda_0=k\bar{\Lambda}_c$.
    \end{conjecture}
    This conjecture leads us to another problem.
    \subsection{Uniqueness of the quantization}

    After the author had realized the quantum Grothendieck ring $ K_t(\mathcal{O}^{\sh}_{\ZZ})$ as a quantum cluster algebra, Qin provided in \cite{qin24} a quantization of the cluster algebra $\mathcal{A}_{w_0}$ using chains of ``good'' subseeds. We explain here briefly how it works. In Qin's notation, $\ddot{\boldsymbol{t}}_0$ is the initial seed for the coordinate ring of the double Bruhat cell $G^{w_0,w_0}$ \cite{bfz3} with quiver $\gamma_c$. Then, using the combinatorics of Coxeter words, it is possible to construct a chain of seeds $(\ddot{\boldsymbol{t}}_i)_{i\in\NN}$ (see Figures \ref{fig:catena subseed t0} and \ref{fig:catena subseed t1}) such that the colimit $\ddot{\boldsymbol{t}}_{\infty}$ coincides with the initial seed $\Sigma_c$ for the cluster algebra $\mathcal{A}_{w_0}$ with quiver $\Gamma_c$, see \cite[Proposition 7.1]{qin24}. Now, by the fore-mentioned results in \cite{bz05}, there exists a quantization matrix $\Lambda^{(0)}$ for $\ddot{\boldsymbol{t}}_0$ and \cite[Lemma 4.7]{qin24} and \cite[Lemma 7.3]{qin24} assure that for each intermediate subseed $\ddot{\boldsymbol{t}}_s$ there is a unique matrix $\Lambda^{(s)}$ compatible with $\ddot{\boldsymbol{t}}_s$ that extends $\Lambda^{(0)}$. The following result is particularly significant for us:
    \begin{theorem}[{\cite[Proposition 7.4]{qin24}}]
        There is a unique extension of the matrix $\Lambda^{(0)}$ for the seed $\ddot{\boldsymbol{t}}_0$ to a quantization matrix $\Lambda^{(\infty)}$ for the seed $\ddot{\boldsymbol{t}}_{\infty}$.
    \end{theorem}

    \begin{figure}[H]
        \centering
        \begin{tikzcd}
    	{\stackrel{r-2}{\square}} && {\stackrel{r}{\bullet}} && {\stackrel{r+2}{\square}}
    	\arrow[from=1-3, to=1-1]
    	\arrow[from=1-3, to=1-5]
        \end{tikzcd}
        \caption{The subseed $\ddot{\boldsymbol{t}}_{0}$ in type $A_1$}
        \label{fig:catena subseed t0}
    \end{figure}
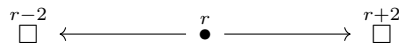
    
    \begin{figure}[H]
        \centering
        \begin{tikzcd}
    	{\stackrel{r-4}{\square}} && {\stackrel{r-2}{\bullet}} && {\stackrel{r}{\bullet}} && {\stackrel{r+2}{\bullet}} && {\stackrel{r+4}{\square}}
    	\arrow[from=1-1, to=1-3]
    	\arrow[from=1-5, to=1-3]
    	\arrow[from=1-5, to=1-7]
    	\arrow[from=1-7, to=1-9]
    \end{tikzcd}
        \caption{The subseed $\ddot{\boldsymbol{t}}_{1}$ in type $A_1$.}
        \label{fig:catena subseed t1}
    \end{figure}
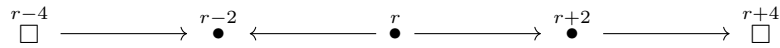

    Thus, a natural question comes to our mind: is there a relation between $\Lambda^{(\infty)}$ and our $\Lambda_c$? Are they the same? Combining with our results, this would imply that the quantum cluster algebra of \cite{qin24} contains $K_t(\mathcal{O}^{\mathfrak{b},+}_{\ZZ})$ of \cite{b21}, a fact that is clear for our construction of the quantization (by Theorem \ref{inclusione anelli}). 
    Or, more in general: does the cluster algebra $\mathcal{A}_{w_0}$ admit a unique quantization?\\[12 pt] In fact, since our definition of the quantization matrix $\Lambda_c$ relies on the matrix $\Lambda_e$ defined by Bittmann in \cite{b21}, we change our question to: is there a unique possible compatible quantization matrix attached to the basic infinite quiver $\Gamma_e$? If the answer is yes, then the uniqueness of $\Lambda_c$ follows immediately by its definition.
    Let us recall the general notion of compatible pair given in Definition \ref{def compatibile}. This definition holds also for infinite rank cluster algebras and we have used it all along this paper. Let us consider from now on only the case without frozen variables, that is, the exchange matrix is squared, skew-symmetric and consists only of its principal part $B$. As in the rest of the paper, the quantization matrix is denoted by $\Lambda$. Note that in the infinite rank case, the matrices $B$ and $\Lambda$ have rows and columns indexed by the vertex set of the quiver associated with the same seed, hence it is a discrete infinite set. In addition, it should be noted that a priori the product of two infinite matrices might not be well defined. In the finite rank case, it follows from the definition of compatible pair that if a quantization exists, then the matrix $B$ has full rank. When, in addition, it is square, $B$ is invertible. In particular, $B^{-1}$ is still skew-symmetric and we have $B^{-1}=\det(B)^{-1}\adj(B)$, where the adjugate is an integer matrix. Consequently, $\adj(B)$ satisfies the compatibility condition. Moreover, it is the unique matrix, up to multiplication by a constant, satisfying 
    \[B^T\Lambda=k\Id.\] So, we can say that in this case the quantization is unique. 

    On the other hand, when the cluster algebra is of infinite rank type, we cannot conclude so easily. This is due to the lack of associativity of the product of infinite matrices, which is the key ingredient to prove the uniqueness of an inverse. However we believe that imposing some further conditions on the quantization matrix, we could obtain that there is only a one-parameter family of matrices that is admissible. We will come back to this problem in a future work.

    \subsection{\texorpdfstring{$(q,t)$}{(q,t)}-characters}
    
    One of the original motivations to study quantum Grothendieck rings is the theory of 
    $(q,t)$-characters by Nakajima \cite{nak04}, which leads to a Kazhdan-Lusztig type algorithm to compute 
    $q$-characters of simple finite-dimensional representations of quantum affine 
    algebras of simply-laced types (the result was recently extended to reachable modules 
    of non simply-laced quantum affine algebras using cluster theoretical methods in \cite{fhoo}). 
    Although $q$-characters have been defined for representation in $\mathcal{O}^{\sh}$, the 
    corresponding $(q,t)$-characters have not been introduced yet, as far as the author knows. 
    Following \cite[Remark 6.14]{b21}, we expect that our cluster theoretic approach 
    to quantum Grothendieck ring of $\mathcal{O}^{\sh}$ will lead to the construction of 
    $t$-deformations of $q$-characters in this context, presumably in terms of $F$-polynomials
     in the quantum cluster algebra. This will be discussed in a forthcoming paper.
    It would also be interesting to compare our quantization with possible deformations of $K_0(\mathcal{O}^{\sh}_{\mathbb{Z}})$
    that could be obtained from the geometry of Coulomb branches, which are known to be closely related to shifted 
    quantum affine algebras and their representations, see for example \cite{vv25}.

    \printbibliography[heading=bibintoc]
    
    \begin{tabular}{ll}
    FRANCESCA PAGANELLI &
    DEPARTMENT S.B.A.I.\\ & SAPIENZA-UNIVERSITÀ DI ROMA\\
    & VIA SCARPA 10, 00164, ROMA (ITALY) \\
    & \&\\
     &UNIVERSITÉ PARIS CITÉ, SORBONNE UNIVERSITÉ\\
    & CNRS, IMJ-PRG\\
    & 8 PL. AURÉLIE NEMOURS, 75013, PARIS (FRANCE)\\
    email: {\tt francesca.paganelli@uniroma1.it}\\[5mm]
    
    \end{tabular}
\end{document}